\documentclass[smallextended]{svjour3}       % onecolumn (second format)
\smartqed  % flush right qed marks, e.g. at end of proof
\usepackage{graphicx}
\usepackage{graphicx}
\usepackage{fullpage}
\usepackage[numbers]{natbib}
\usepackage{hyperref}       % hyperlinks
\usepackage{url}            % simple URL typesetting
\usepackage{booktabs}       % professional-quality tables
\usepackage{amsfonts}       % blackboard math symbols
\usepackage{nicefrac}       % compact symbols for 1/2, etc.
\usepackage{microtype}      % microtypography
\usepackage{amsmath,amssymb}
\usepackage{bm}
\usepackage{graphicx}
\usepackage{subcaption}
\usepackage{caption}
\usepackage{enumitem}
\usepackage{algorithmic,algorithm}
\usepackage{color}
\newcommand{\X}{\mathcal{X}}
\newcommand{\bzt}{\boldsymbol{\zeta}}
\newcommand{\blambda}{\boldsymbol{\lambda}}
\newcommand{\bxi}{\boldsymbol{\xi}}

\newcommand{\reals}{\mathbb{R}}

\newcommand{\argmin}{\mathop{\mathrm{arg\,min}{}}}
\newcommand{\argmax}{\mathop{\mathrm{arg\,max}{}}}
\newcommand{\prox}{\mathbf{prox}}
\newcommand{\dom}{\mathrm{dom\,}}

\newcommand{\gammaInc}{\gamma_\mathrm{inc}}
\newcommand{\gammaDec}{\gamma_\mathrm{dec}}
\newcommand{\gammaSC}{\gamma_\mathrm{sc}}
\newcommand{\thetaSC}{\theta_\mathrm{sc}}
\newcommand{\Lmin}{L_\mathrm{min}}
\newcommand{\Lini}{L_\mathrm{ini}}
\newcommand{\xini}{x^{\mathrm{ini}}}

\newcommand{\ba}{{\mathbf a}}
\newcommand{\bb}{{\mathbf b}}
\newcommand{\bA}{\mathbf A}
\newcommand{\bE}{\mathbf E}

\newcommand{\bs}{{\mathbf s}}
\newcommand{\bp}{{\mathbf p}}
\newcommand{\bx}{{\mathbf x}}
\newcommand{\bz}{{\mathbf z}}
\newcommand{\by}{{\mathbf y}}
\newcommand{\bu}{{\mathbf u}}

\newcommand{\bw}{{\mathbf w}}

\newcommand{\RR}{\mathbb{R}}
\newcommand{\vc}{{\mathbf{c}}}
\newcommand{\vf}{{\mathbf{f}}}
\newcommand{\cB}{{\mathcal{B}}}
\newcommand{\vzero}{{\mathbf{0}}}

\newcommand{\vertiii}[1]{{\left\vert\kern-0.25ex\left\vert\kern-0.25ex\left\vert #1
		\right\vert\kern-0.25ex\right\vert\kern-0.25ex\right\vert}}

\newtheorem{assumption}{Assumption}
\begin{document}
	
	%\title{Complexity of an Inexact Proximal-Point Penalty Method for Constrained Smooth Non-Convex Optimization}
	\title{Inexact Proximal-Point Penalty Methods for Constrained Non-Convex
Optimization}
	
	
	\author{Qihang Lin
		\and Runchao Ma
		\and Yangyang Xu}
	
	
	\institute{Qihang Lin \at
		Department of Business Analytics, University of Iowa, Iowa City, IA, 52242\\
		\email{qihang-lin@uiowa.edu}           
		\and
		Runchao Ma \at
		Department of Business Analytics, University of Iowa, Iowa City, IA, 52242\\
		\email{runchao-ma@uiowa.edu}
		\and
		Yangyang Xu \at
		Department of Mathematical Sciences, Rensselaer Polytechnic Institute, Troy, NY, 12180\\
		\email{xuy21@rpi.edu}
	}
	
	\date{Received: date / Accepted: date}

	\maketitle
	
	\begin{abstract}
		In this paper, an inexact proximal-point penalty method is studied for constrained optimization problems, where the objective function is non-convex, and the constraint functions can also be non-convex. This method approximately solves a sequence of subproblems, each of which is formed by adding to the original objective function a proximal term and quadratic penalty terms associated to the constraint functions. Under a weak-convexity assumption, each subproblem is made strongly convex and can be solved effectively to a required accuracy by an optimal gradient-based method. The computational complexity of this approach is analyzed separately for the cases of convex constraint and non-convex constraint. For both cases, the complexity results are established in terms of the number of proximal gradient steps needed to find an $\varepsilon$-stationary point. When the constraint functions are convex, we show a complexity result of $\tilde O(\varepsilon^{-5/2})$ to produce an $\varepsilon$-stationary point under the Slater's condition. When the constraint functions are non-convex, the complexity becomes $\tilde O(\varepsilon^{-3})$ if a non-singularity condition holds on constraints and otherwise $\tilde O(\varepsilon^{-4})$ if a feasible initial solution is available.
	\end{abstract}
	
	\section{Introduction}
	\label{sec:intro}
	We consider the nonconvex optimization problem with inequality and equality constraints: %\comm{how about to use $c_i(\bx)=0$}
	\begin{eqnarray}
		\label{eq:gco}
		\min_{\bx\in\mathbb{R}^d}f_0(\bx)+g(\bx),\quad\mathrm{s.t.}\quad \vf(\bx)\leq \vzero,~\quad \vc(\bx)=\vzero,
	\end{eqnarray}
	where  $g:\mathbb{R}^d\rightarrow\mathbb{R}\cup\{+\infty\}$, $\vf = [f_1,\ldots, f_m]$ with $f_i:\mathbb{R}^d\rightarrow\mathbb{R}$ for each $i=0,\dots,m$, and $\vc=[c_1,\ldots, c_n]$  with $c_j:\mathbb{R}^d\rightarrow\mathbb{R}$ for each $j=1,\dots,n$. We assume that $g$ is a proper lower-semicontinuous convex function with a compact domain and all other functions are continuously differentiable.

	For a general non-convex function, finding its global minimizer is intractable, and it becomes even more difficult, when there are (non-convex) constraints. Therefore, instead of finding a global minimizer of~\eqref{eq:gco}, we focus on finding a stationary point. We call a point $\bx^*\in\mathrm{dom}(g)$ %A common goal in non-convex optimization is to find 
	a \emph{stationary} point of \eqref{eq:gco}, if there are $\blambda^*\in\RR^m_+$ and $\by^*\in\mathbb{R}^n$, which exist if some constraint qualification is assumed,  such that %which is  that satisfies 
	the Karush-Kuhn-Tucker (KKT) conditions hold: %\cite[Theorem 28.3]{rockafellar1970convex}
	\begin{subequations}
		\label{eq:KKT}
		\begin{align}
			\mathbf{0}\in\nabla f_0(\bx^*)+J_\vf(\bx^*)^\top\blambda^*+J_\vc(\bx^*)^\top\by^*+\partial g(\bx^*),\\
			f_i(\bx^*)\leq0, \, i=1,\ldots,m, \quad c_j(\bx^*)=0, \, j=1,\ldots, n,\\
			\lambda_i^*f_i(\bx^*)=0, \, i=1,\ldots,m,
		\end{align}
	\end{subequations}
	where $\partial g(\bx^*)$ denotes the subdifferential of $g$ at $\bx^*$,  $J_\vf(\bx^*)$ denotes the Jacobian matrix of $\vf$ at $\bx^*$, and $J_\vc(\bx^*)$ denotes the Jacobian matrix of $\vc$ at $\bx^*$. The vectors $\blambda^*$ and $\by^*$ are called Lagrangian multipliers. Due to the inevitable truncation error, it is hard to compute a solution that satisfies the above conditions exactly. %by a numerical approach. %an exact stationary point is hard to compute by an algorithm with only a finite number of iterations, 
	Numerically, it is more reasonable to pursue an approximate stationary point defined as follow. Here, $\|\cdot\|$ stands for the Euclidean norm. 
	
	\begin{definition}[$\varepsilon$-stationary point and its weak version]\label{def:eps-stationary-pt}
		Given $\varepsilon>0$, a point $\bar\bx$ is an {$\varepsilon$-stationary point} of \eqref{eq:gco} if there are $\bar\bxi\in\partial g(\bar\bx)$, $\bar\blambda\in\mathbb{R}^m_+$, and $\bar\by\in\mathbb{R}^n$ such that $\bar\lambda_i=0$ if $f_i(\bar\bx)<0$ for $i=1,\dots,m$ and
		\small
		\begin{subequations}
			\label{eq:eKKT}
			\begin{align}
				\left\|\nabla f_0(\bar\bx)+J_\vf(\bar\bx)^\top\bar\blambda+J_\vc(\bar\bx)^\top\bar \by+\bar\bxi\right\|\leq \varepsilon, \label{eq:eKKT-df}\\
				\sqrt{\|\vc(\bar\bx)\|^2+\big\|[\mathbf{f}(\bar\bx)]_+\big\|^2}\leq\varepsilon, \label{eq:eKKT-pf}\\\label{eq:ecomplementaryslackness}
				\sum_{i=1}^m|\lambda_i f_i(\bar\bx)|\leq \varepsilon.
			\end{align}
		\end{subequations}
		\normalsize
		If only \eqref{eq:eKKT-df} and \eqref{eq:eKKT-pf} hold, $\bar\bx$ is called a \emph{weak $\varepsilon$-stationary point} of \eqref{eq:gco}.	
		%A point $\bar\bx$ is an \textbf{$\varepsilon$-stationary point of Type II} of \eqref{eq:gco-cvx} if it satisfies the same conditions of an $\varepsilon$-stationary point of Type I except \eqref{eq:ecomplementaryslackness}.
	\end{definition}
	
	Here, the three conditions in \eqref{eq:eKKT} are $\varepsilon$-approximation of the three conditions in \eqref{eq:KKT} while the condition that $\bar\lambda_i=0$ if $f_i(\bar\bx)<0$ essentially requires the complementary slackness condition in \eqref{eq:KKT} holds exactly when $f_i(\bar\bx)<0$. When there are only equality conditions, Definition~\ref{def:eps-stationary-pt} is the same as that for the $\varepsilon$-approximate first-order solution considered in several existing papers, e.g., \cite{xie2019complexity,NIPS2019_9545}. When there are inequality constraints, the $\varepsilon$-stationary solution in Definition~\ref{def:eps-stationary-pt} is stronger than the solutions guaranteed by 
	\cite{birgin2019complexity,grapiglia2019complexity}, which do not ensure \eqref{eq:ecomplementaryslackness} but only $\bar\lambda_i=0$ if $f_i(\bar\bx)<-\varepsilon$. %\comm{This result does not make sense to me, in case they also have \eqref{eq:eKKT-df}. Is it $f_i(\bar\bx)<-\varepsilon$? Qihang: It is a typo. They require $f_i(\bar\bx)<-\varepsilon$.}
	
	Our goal is to establish the theoretical complexity of finding an $\varepsilon$-stationary point or a weak $\varepsilon$-stationary point of \eqref{eq:gco}. To achieve this goal, we consider an \emph{inexact proximal-point penalty} (iPPP) method (see Algorithm~\ref{alg:iPPP} below). Our method solves a sequence of strongly-convex unconstrained subproblems that are constructed by combining two classical techniques: the proximal-point method and the quadratic penalty method; see \eqref{eq:iALMsub} below. The \emph{adaptive accelerated proximal gradient} (AdapAPG) method by~\cite{Nesterov2013,lin2015adaptive} (see Algorithm~\ref{alg:adap-APG1}) is applied to approximately solve each subproblem. To show the complexity results, we consider two cases of \eqref{eq:gco} separately and assume different regularity conditions for them. In the first case, the problem has a weakly-convex objective (see Definition~\ref{def:weaklyconvex}) but convex constraint functions, and we assume Slater's condition. In the second case, the objective and constraint functions are all weakly convex, and we assume either a non-singularity condition (see Assumption~\ref{assume:nonconvexconstraintsingular}) or the feasibility of the initial solution (see Assumption \ref{assume:nonconvexconstraintfeasible}). %To guar an $\varepsilon$-stationary point, 
	%Different regularity conditions will be assumed for the two cases. the iPPP method needs Slater condition in the former case, and needs a non-singularity condition on the constraint functions in the latter case. Moreover, in the latter case, when the non-singularity condition does not hold but an initial feasible solution is available, 
	%For the second case, if the non-singularity condition does not hold but an initial feasible solution is available, the iPPP method will be able to find a slightly weaker $\varepsilon$-stationary point defined below.
	%\begin{definition}[weak $\varepsilon$-stationary point]\label{def:weak-eps-stationary-pt} Given $\varepsilon>0$, a point $\bar\bx$ is a {weak $\varepsilon$-stationary point} of \eqref{eq:gco} if it satisfies all the required conditions for an $\varepsilon$-stationary point except~\eqref{eq:ecomplementaryslackness}.
	%\end{definition}

	\subsection{Contributions} 
	\label{sec:contrib}
	%\vspace{-0.2cm}
	We make contributions to understanding the theoretical complexity of finding an $\epsilon$-stationary point of a non-convex constrained problem in the form of \eqref{eq:gco}. Three scenarios are studied and the computational complexity of the iPPP method, measured by the number of proximal gradient steps, is established in each scenario. 
	%For each scenario, our result is valuable compared 
	%For each scenario, our complexity result is currently the best in the literature, in terms of the dependence on the targeted tolerance $\varepsilon$. %\comm{Can we claim this? Qihang: I think so. Guanghui Lan achieve $O(1/\epsilon^3)$ also but they require a strong assumption. } 
	They are summarized as follows. 
	\begin{itemize}[leftmargin=*]\setlength\itemsep{-1pt}
		\item For the case where $f_0$ is weakly convex, $f_i$ is convex for $i=1,\dots,m$, and $c_j$ is affine for $j=1,\dots,n$, we show that, when Slater's condition holds, the proposed iPPP method can find an $\varepsilon$-stationary point within $\tilde O(\varepsilon^{-5/2})$ proximal gradient steps.\footnote{Here and in the rest of paper, we suppress all logarithmic terms in $\tilde O$.} \emph{This complexity is first achieved by this paper and remains by far the best complexity for \eqref{eq:gco} under these assumptions.}
		%when a sublinearly growing penalty parameter and a constant proximal parameter are adopted. 
		%This complexity result %is currently the best in the literature for the considered class of problems. It improves the result $\tilde O(\varepsilon^{-3})$ achieved in \cite{kong2019complexity} that considers only affinely-constrained non-convex problems. 
		%	The recent methods proposed by~\cite{melo2020iterationnew,melo2020iteration,zhang2020proximal,zhang2020global} achieve complexity of $O(\varepsilon^{-3})$, $O(\varepsilon^{-5/2})$, $O(\varepsilon^{-2})$, and $O(\varepsilon^{-2})$, respectively, but they all assume affine constraints also.
		\vspace{5pt} 
		\item When $\{f_i\}_{i=0}^m$ and $\{c_j\}_{j=1}^n$ are all weakly convex, we show that, if a non-singularity condition (see Assumption~\ref{assume:nonconvexconstraintsingular}) is satisfied by the constraint functions, the iPPP method can find an $\varepsilon$-stationary point within $\tilde O(\varepsilon^{-3})$ proximal gradient steps. %when a sublinearly growing penalty parameter and a varying proximal parameter are adopted. 
		This complexity %is also currently the best in the literature. It 
		improves the one $\tilde O(\varepsilon^{-4})$ achieved in \cite{NIPS2019_9545} that uses an inexact augmented Lagrangian method under the same assumptions.\footnote{A complexity of $\tilde O(\varepsilon^{-3})$ is claimed in Corollary 4.2 in \cite{NIPS2019_9545}. However, there is an error in its proof. The authors claimed the complexity of solving their subproblem is $O(\frac{\lambda^2_{\beta_k}\rho^2}{\epsilon_{k+1}})$ but it should be $O(\frac{\lambda^2_{\beta_k}\rho^2}{\epsilon_{k+1}^2})$. (See \cite{NIPS2019_9545} for the definitions of $\lambda_{\beta_k}$, $\rho$ and $\epsilon_{k+1}$.) After correcting this error, following the same proof they used gives a total complexity of $\tilde O(\epsilon^{-4})$ .}
		%(See Remark~\ref{lemma:Sahin}).
		%\comm{We need to put a remark or footnote here.qihang: I add a remark below.}
		\vspace{5pt} 
		\item When $\{f_i\}_{i=0}^m$ and $\{c_j\}_{j=1}^n$ are all weakly convex, we show that, if an initial feasible solution is available (but the aforementioned non-singularity condition is not needed), the iPPP method can find a \emph{weak} $\varepsilon$-stationary point within $\tilde O(\varepsilon^{-4})$ proximal gradient steps. In Section~\ref{sec:relatedwork}, we will discuss how this result is compared with other works that also consider non-convex constraints without the non-singularity condition.
		%This complexity result improves that of the penalty method by \cite{cartis2011evaluation} which requires solving  $\tilde O(\varepsilon^{-5})$ trust-region subproblems whose complexity is no less than $\tilde O(\varepsilon^{-5})$ proximal gradient steps.
		%when a constant penalty parameter and a constant proximal parameter are adopted.
	\end{itemize}
	%\begin{remark}
	%	\label{lemma:Sahin}
	%	A complexity of $\tilde O(\varepsilon^{-3})$ is claimed in Corollary 4.2 in \cite{NIPS2019_9545}. However, there is an error in its proof. The authors claimed the complexity of solving their subproblem is $O(\frac{\lambda^2_{\beta_k}\rho^2}{\epsilon_{k+1}})$ but it should be $O(\frac{\lambda^2_{\beta_k}\rho^2}{\epsilon_{k+1}^2})$. (See \cite{NIPS2019_9545} for the definitions of $\lambda_{\beta_k}$, $\rho$ and $\epsilon_{k+1}$.) After correcting this error, following the same proof they used gives a total complexity of $\tilde O(\epsilon^{-4})$ .
	%\end{remark}

	\subsection{Organization of the paper}
	The rest of the paper is organized as follows. In Section~\ref{sec:relatedwork}, we discuss related works on convex and non-convex constrained optimization. 
	In Section~\ref{sec:preliminary}, we introduce some definitions, notations, and a subroutine used in the proposed algorithm. Details of the proposed algorithm are described in Section~\ref{sec:algorithm}. The complexity analysis is conducted in Section~\ref{sec:penaltymethod} for the convex constrained case and in Section~\ref{sec:penaltymethodnonconvex} for the non-convex constrained case. Numerical results are presented in Section~\ref{sec:exp}, and Section~\ref{sec:conclusion} concludes the paper.
	
	\section{Related Works}
	\label{sec:relatedwork}
	%\vspace{-3mm}
	There has been growing interest in first-order algorithms for non-convex minimization problems with no constraints or simple constraints~\footnote{\label{footnote}Here, simple constraints mean the constraints allow a closed-form projection onto the feasible set.} in both stochastic and deterministic settings. See, e.g., \cite{DBLP:journals/siamjo/GhadimiL13a,DBLP:journals/mp/GhadimiL16,Reddi:2016:SVR:3045390.3045425,DBLP:journals/corr/abs/1805.05411,DBLP:conf/icml/Allen-Zhu17,Davis2018,davis2018stochastic,Drusvyatskiy2018,davis2017proximally,zhang2018convergence}. However, for \eqref{eq:gco} with constraints that are not simple, these methods are not applicable. There is a long history of studies on continuous optimization with constraints. The recent works on first-order methods for convex optimization with convex constraints include~\cite{tran2014primal,yang2017richer,wei2018solving,xu2018primal,xu2017global,xu2017first,yu2017online,lin2018levelfinitesum,bayandina2018mirror} for deterministic constraints and \cite{Lan2016,yu2017simple,basu2019optimal} for stochastic constraints. Different from these works, this paper studies the problems with a non-convex objective function and with potentially non-convex constraints. 
	
	When all constraint functions in \eqref{eq:gco} are affine, %$f_i$ is linear for $i=1,\dots,m$, 
	a primal-dual Frank-Wolfe method is proposed in \cite{wei2018primal}, and it finds an $\varepsilon$-stationary point with a complexity of $O(\varepsilon^{-3})$ in general and $O(\varepsilon^{-2})$ when there exists a strictly feasible solution. We adopt a notion of $\varepsilon$-stationary point different from that in \cite{wei2018primal}, and our constraint functions can be nonlinear and non-convex.
	
	As a classical approach for solving problems in the form of \eqref{eq:gco}, a penalty method finds an approximate solution by solving a sequence of unconstrained subproblems, where %with positively weighted penalty terms in the objective functions that 
	the violation of constraints is penalized by the positively weighted penalty terms in the objective function of the subproblems. Unconstrained optimization techniques are then applied to the subproblems along with an updating scheme for the weighting parameters. The computational complexity of penalty methods for convex problems has been well established~\cite{lan2013iteration,necoara2019complexity,tran2019proximal}. For non-convex problems, most existing studies focus on the asymptotic convergence to a stationary point. See, e.g., ~\cite{di1985continuously,di1986exact,fletcher1983penalty,powell1986recursive,gould1989convergence,burke1991exact,byrd2005convergence}. On the contrary, we analyze the finite complexity of penalty methods for finding an $\varepsilon$-stationary point.
	
	An exact penalty method has been studied in \cite{cartis2011evaluation} as an application of a trust region method for a composite non-smooth problem. When applied to \eqref{eq:gco} with $g\equiv 0$, the method in \cite{cartis2011evaluation} either finds an $\varepsilon$-infeasible and $\varepsilon$-critical point of~\eqref{eq:gco} (see~\cite{cartis2011evaluation} for the definition) or finds a solution that is infeasible to \eqref{eq:gco} but $\varepsilon$-critical to the infeasibility measure $\sum_{j=1}^n\left |c_j(\bx) \right |+\sum_{i=1}^m\max\{f_i(\bx),0\}$. It needs to exactly solve $O(\varepsilon^{-2})$ linearized trust-region subproblems if the penalty parameter is bounded above and solve $O(\varepsilon^{-5})$ subproblems otherwise. In a subsequent study~\cite{cartis2014complexity} and its corrigendum \cite{cartis2017corrigendum}, a target-following algorithm is developed. It can find an approximate Fritz-John (instead of KKT) solution %\comm{Is the term Fritz-John OK here?Qihang: I think it is OK} 
	with similar guarantee as~\cite{cartis2011evaluation} by solving $O(\frac{1}{\varepsilon^2})$ subproblems regardless of the boundedness of penalty parameter.
	This method has been extended to the case when $f_0$ is the expectation of a stochastic function in \cite{wang2017penalty}. We want to emphasize that the complexity result of~\cite{cartis2011evaluation,cartis2014complexity} is given in terms of the number of exactly-solved trust-region subproblems, and thus it is not exactly computational (time) complexity, especially when the subproblem is not trivially solvable. On the contrary, we directly analyze the total computational (time) complexity of the proposed method. When the constraints are non-convex, our method has complexity $\tilde O(\varepsilon^{-3})$ and $\tilde O(\varepsilon^{-4})$, respectively, when a non-singularity condition (Assumption~\ref{assume:nonconvexconstraintsingular}) is assumed and when an feasible initial solution (Assumption~\ref{assume:nonconvexconstraintfeasible}) is assumed. Neither assumption is needed in~\cite{cartis2011evaluation}. Suppose the time complexity of solving a trust-region subproblem in~\cite{cartis2011evaluation} is the same as a proximal gradient step in our method. The complexity of~\cite{cartis2011evaluation} is lower than ours if their penalty parameters are bounded and higher than ours, otherwise. Moreover, the method by~\cite{cartis2011evaluation} did not always guarantee an $\epsilon$-feasible solution while our method does, which is mainly because of Assumption~\ref{assume:nonconvexconstraintsingular} or \ref{assume:nonconvexconstraintfeasible} we make.  
	
	On solving a problem with a non-convex objective and linear constraint, \cite{kong2019complexity} has developed a quadratic-penalty accelerated
	inexact proximal point method. %is developed to solve~\eqref{eq:gco} with non-convex objective function but only linear constraints~\cite{kong2019complexity}. 
	That method can generate an $\varepsilon$-stationary point in the sense of \eqref{eq:eKKT} %equation (2) in \cite{kong2019complexity} 
	with a complexity of $O(\varepsilon^{-3})$. Our method is similar to that in \cite{kong2019complexity} by utilizing the techniques from both the proximal point method and the quadratic penalty method. Although we make a little stronger assumption than~\cite{kong2019complexity} by requiring the boundedness of $\mathrm{dom}(g)$, our method and analysis apply to the problems with non-convex objectives and convex/non-convex nonlinear constraint functions. When the constraints are convex (but possibly nonlinear), our method can find an $\varepsilon$-stationary point with a complexity of $\tilde O(\varepsilon^{-5/2})$ that is a nearly $O(\varepsilon^{-1/2})$ improvement over the complexity in \cite{kong2019complexity}. %in the dependency on $\varepsilon$ up to logarithmic terms. 

	Barrier methods are another traditional class of algorithms for constrained optimization. Similar to penalty methods, they also solve a sequence of unconstrained subproblems with barrier functions added to the objective function. The barrier functions will increase to infinity as the iterates approach the boundary of the feasible set, and thus enforce the iterates to stay in the interior of the feasible set. The convergence rate of barrier methods has been studied by~\cite{tran2014inexact,tran2013composite,nesterov2011barrier,tran2018single} for convex problem. For a general non-convex problems,  most studies only focus on asymptotic convergence analysis. Recent works~\cite{haeser2019optimality,o2019log} proposed algorithms based on logarithmic barrier function for non-convex problems with only non-negative and linear constraints. They established the complexity of their algorithms for finding first-order and second-order $\varepsilon$-KKT point (whose definitions are slightly different in \cite{haeser2019optimality,o2019log} and different from our definition). However, they do not consider nonlinear constraints as we do.
	
	The augmented Lagrangian method (ALM) is another effective approach for constrained optimization. At each iteration, ALM updates the primal variable by minimizing the augmented Lagrangian function and then performs a dual gradient ascent step to update the dual variable. The iteration complexity of ALM has been established for convex problems~\cite{lan2013iteration,xu2018primal,xu2017first,xu2017global,necoara2019complexity}.  
	For non-convex problems, asymptotic convergence or local convergence rate of ALM has been studied by~\cite{friedlander2005globally,birgin2010global,wang2014convergence,curtis2016adaptive,wang2019global,birgin2018augmented,fernandez2012local}. 
	The computational complexity of ALM and its variants (e.g. ADMM) for finding an $\varepsilon$-stationary point for linearly constrained non-convex problems has been studied by  \cite{hong2016decomposing,melo2017iteration,jiang2019structured,hajinezhad2019perturbed,melo2020iterationnew,melo2020iteration,gonccalves2017convergence,zhang2020proximal,zhang2020global}. For example, the proximal inexact ALM method by~\cite{melo2020iteration} achieves complexity of $O(\varepsilon^{-5/2})$ and a related but different ALM method by~\cite{zhang2020proximal,zhang2020global} achieves complexity of $O(\varepsilon^{-2})$, the latter of which is by far the best result for nonconvex problems with linear constraints.
	
	ALM and its proximal variant are analyzed by \cite{NIPS2019_9545,birgin2019complexity,grapiglia2019complexity,xie2019complexity, li2020augmented,li2020rate,kong2020iteration} for non-convex problems with nonlinear constraints. In each main iteration of those methods, an approximate stationary point of the (proximal) augmented Lagrangian function is computed by first- or second-order methods. Utilizing the Hessian information, the methods in \cite{NIPS2019_9545,xie2019complexity} can  find a second-order $\varepsilon$-stationary point while our method cannot. Without Hessian information, the methods by \cite{NIPS2019_9545,xie2019complexity} can still find a first-order $\varepsilon$-stationary point. In \cite{NIPS2019_9545}, under a \emph{non-singularity assumption} that the smallest singular value of the Jacobian matrix of the constraint functions is uniformly bounded away from zero, it is showed that ALM finds an $\varepsilon$-stationary point with complexity of $\tilde O(\varepsilon^{-3})$. %\comm{Should we point out that this complexity is actually not correct? Qihang: I add a sentence below.} 
	However, the proof of their complexity contains an error and the actual complexity derived from their proof is $\tilde O(\varepsilon^{-4})$. 
	%(See Remark~\ref{lemma:Sahin} for details)
	On the contrary, our method has a complexity of $\tilde O(\varepsilon^{-3})$ for problems with non-convex constraints under the assumptions similar to \cite{NIPS2019_9545,xie2019complexity} and only has a complexity of $\tilde O(\varepsilon^{-2.5})$ for convex constrained problems. Moreover, we consider both inequalities and equality constraints while \cite{NIPS2019_9545,xie2019complexity} only consider equality constraints. 
	%Although one can replace the inequality constraints in \eqref{eq:gco} by equality constraints $\max\{f_i(\bx),0\}^2=0$, $i=1,\dots,m$, doing this will fail the aforementioned non-signularity assumption in~\cite{NIPS2019_9545,xie2019complexity}. 
	In addition, even if the non-signularity assumption does not hold, our method can still find an $\varepsilon$-stationary point as long as an initial feasible solution is available. This result benefits the applications where the constraints are non-convex but have some special structure that allows finding a feasible  solution  easily (e.g. \cite{wang2019clustering} and \cite{jiang2019exact}). %\comm{Which reference here?Qihang: I add two references}. See \cite{wang2019clustering} and \cite{jiang2019exact} for such examples. 
	After the release of the first draft \cite{lin2019inexact-pp} of this paper, \cite{li2020augmented} gave a hybrid of the quadratic penalty method and ALM, which also achieves an $\tilde O(\varepsilon^{-2.5})$ complexity for non-convex problems with convex constraints under the same assumptions as we make in Assumption~\ref{assume:convexconstraint}. However, \cite{li2020augmented} shows that the complexity of the pure-ALM-based first-order method is $\tilde O(\varepsilon^{-3})$, and the same-order complexity result has also been established in \cite{kong2020iteration} for a proximal ALM on solving non-convex problems with nonlinear convex constraints. \cite{li2020rate} adopts a proximal-point based subroutine and improves to $\tilde O(\varepsilon^{-3})$ the complexity result of the first-order ALM in \cite{NIPS2019_9545} for equality-constrained nonconvex problems.

	In \cite{birgin2019complexity}, the authors assume neither the non-singnularity assumption nor a feasible initial solution while are still able to achieve $O(\varepsilon^{-3})$ complexity for ALM. However, in their setting, ALM does not necessarily guarantee an $\varepsilon$-stationary point of \eqref{eq:gco} but may only return a point that is infeasible to \eqref{eq:gco} and $\varepsilon$-stationary to the infeasibility measure (similar to the guarantee by \cite{cartis2011evaluation}). The totaly number of iterations needed by ALM is also analyzed by~\cite{grapiglia2019complexity} when the constraints are linear or quadratic. However, they solve the ALM subproblems by a second-order or high-order method so their complexity per iteration is much higher than ours. Finally, we want to emphasize again that the $\varepsilon$-stationary point we consider in  Definition~\ref{def:eps-stationary-pt} is stronger than the solutions guaranteed by 
	\cite{birgin2019complexity,grapiglia2019complexity} which do not satisfy \eqref{eq:ecomplementaryslackness} and only satisfy $\bar\lambda_i=0$ when $f_i(\bar\bx)<-\varepsilon$. %\comm{Should it be $<-\varepsilon$?}
	%In \cite{zhang2020global}, the authors show that a 
	
	In addition to \cite{NIPS2019_9545,xie2019complexity}, the algorithms by \cite{hong2018gradient,nouiehed2018convergence,lu2019snap} also utilize Hessian information to find a second-order $\varepsilon$-stationary point for linearly constrained non-convex optimization. Different from these works, we focus on finding an approximate first-order stationary point for nonlinear constrained non-convex optimization using only gradient information.

	Two recent works \cite{boob2019proximal, ma2019proximally} proposed similar algorithms for non-convex constrained optimization based on the proximal-point technique. In their approaches, a strongly convex constrained subproblem is constructed in each main iteration by adding proximal terms to the objective and constraints. When applied to non-convex smooth constrained optimization, both methods find an $\varepsilon$-stationary point in complexity of $\tilde O(\varepsilon^{-3})$. Their analysis requires (1) a feasible initial solution and (2) the uniform boundedness of the dual solutions of all subproblems. To satisfy the latter requirement,  \cite{ma2019proximally} assumes that a uniform Slater's condition holds, and \cite{boob2019proximal} assumes that the Mangasarian-Fromovitz constraint qualification holds at the limiting points of the generated iterates. On the contrary, when the non-signularity assumption holds, our method also has complexity $\tilde O(\varepsilon^{-3})$ but does not require a feasible initial solution as they do. When an initial feasible solution is indeed available, our method can be analyzed in an alternative way without the non-signularity assumption or the constraint qualification conditions required by~\cite{boob2019proximal, ma2019proximally}, although the complexity becomes $\tilde O(\varepsilon^{-4})$.

	\section{Preliminary}\label{sec:pre}
	\label{sec:preliminary}
	%\vspace{-3mm}
	In this section, we provide some basic definitions and describe a numerical subroutine used within our main algorithm.
	
	\subsection{Definitions and Assumptions}
	We denote $\|\cdot\|$ as the $\ell_2$-norm and $\|\cdot\|_1$ as the $\ell_1$-norm. Let
	\small
	\begin{equation}\label{eq:dom-g}
		\mathcal{X}=\dom(g):=\{\bx\in\RR^d: g(\bx)<+\infty\}
	\end{equation}
	\normalsize
	be the domain of a function $g$. The interior and boundary of $\mathcal{X}$ are respectively denoted by $\mathrm{int}(\mathcal{X})$ and $\partial \mathcal{X}$. We use $\mathcal{N}_{\mathcal{X}}(\bx)$ for the normal cone of $\mathcal{X}$ at $\bx$. Given $a>0$, we use $\cB_a$ to represent the ball $\{\bx\in\RR^d: \|\bx\|\le a\}$. We denote $\mathbf{0}$ as an all-zero vector whose dimension is clear from the context, and $[\ba]_+=\max\{\mathbf{0}, \ba\}$ is the vector of component-wise maximum between $\mathbf{0}$ and $\ba$. For a convex set $\mathcal{S}$, we use $\mathrm{dist}(\bx,\mathcal{S})=\min_{\by\in \mathcal{S}}\|\by-\bx\|$ for the distance of $\bx$ to $\mathcal{S}$. For any $\bx\in\RR^d$, $J_\vf(\bx)=[\nabla f_1(\bx),\ldots, \nabla f_m(\bx)]^\top\in\RR^{m\times d}$ and $J_\vc(\bx)=[\nabla c_1(\bx),\ldots, \nabla c_n(\bx)]^\top\in\RR^{n\times d}$ denote the Jacobian matrices of $\vf$ and $\vc$ at $\bx$, respectively.
	
	We adopt the following definitions. 
	
	\begin{definition}[subdifferential]\label{def:subdiff} Given a proper lower-semicontinuous convex function $h:\reals^d\rightarrow \reals\cup\{+\infty\}$, its subdifferential at any $\bx$ in the domain is defined as
		\small
		\begin{align*}
			\partial h(\bx)=
			\big\{\bzt\in\mathbb{R}^d\,\big| \,
			h(\bx')
			\geq h(\bx)+\bzt^\top(\bx'-\bx), \forall \bx'\in\RR^d
			%+o(\|\bx'-\bx\|), ~\bx'\rightarrow\bx
			\big\},
		\end{align*}
		\normalsize
		and each $\bzt\in\partial h(\bx)$ is called a subgradient of $h$ at $\bx$. 
	\end{definition}
	%We say  $h$ is \emph{$\mu$-strongly convex ($\mu\geq0$)} if 
	%$$
	%h(\by)\geq h(\bx')+\bzt^\top(\bx-\bx')+\frac{\mu }{2}\|\bx-\bx'\|^2
	%$$
	%for any $\bx$, $\bx'$ and $\bzt\in\partial h(\bx')$. We say $h$ is \emph{$\rho$-weakly convex} ($\rho\geq0$) if
	%$$
	%h(\bx)\geq h(\bx')+\bzt^\top(\bx-\bx')-\frac{\rho}{2}\|\bx-\bx'\|_2^2
	%$$
	%for any $\bx$, $\bx'$ and $\bzt\in\partial h(\bx')$. 
	%We denote the normal cone of $\X$ at $\bx$ by $\mathcal{N}_\X(\bx)$ and the distance from $\bx$ to a set $S$ by $\mathrm{Dist}(\bx,S)=\min_{\by\in S}\|\bx-\by\|$. 
	
	\begin{definition}[$L$-smoothness]
		\label{def:smooth}
		A function $h:\reals^d\rightarrow \reals$ is $L$-smooth if it is differentiable on $\RR^d$ and satisfies %\comm{are they equivalent? Qihang: Yes. They are equivalent. But I think we don't have to mention both definitions. I just delete one.}
		\small
		\begin{eqnarray}
			\label{eq:L}
			\textstyle	h(\bx) \leq h(\bx')+\left\langle \nabla h(\bx'),\bx-\bx' \right\rangle+\frac{L}{2}\|\bx'-\bx\|^2, \, \forall\, \bx, \bx'\in\RR^d.
		\end{eqnarray} 
		\normalsize
	\end{definition}
	
	\begin{definition}[$\mu$-strong convexity]
		\label{def:stronglyconvex}
		A function $h:\reals^d\rightarrow \reals\cup\{+\infty\}$ is $\mu$-strongly convex for $\mu\geq 0$ if $h-\frac{\mu}{2}\|\cdot\|^2$ is convex. When $\mu=0$, $\mu$-strong convexity is reduced to convexity. %, for any $\bx$, $\bx'$, and $\bzt\in\partial h(\bx')$, 
		%	\begin{eqnarray}
		%	\label{eq:rho}
		%	h(\bx) \geq h(\bx')+\left\langle \bzt,\bx-\bx' \right\rangle+\frac{\mu}{2}\|\bx'-\bx\|^2.
		%	\end{eqnarray} 
	\end{definition}
	
	\begin{definition}[$\rho$-weak convexity]
		\label{def:weaklyconvex}
		A function $h:\reals^d\rightarrow \reals\cup\{+\infty\}$ is $\rho$-weakly convex for $\rho\geq 0$ if $h+\frac{\rho}{2}\|\cdot\|^2$ is convex. When $\rho=0$, $\rho$-weak convexity is reduced to convexity.
	\end{definition}
	
	When $h$ is smooth and $\rho$-weakly convex, it holds that for any $\bx$ and $\bx'$, 
	\small
	\begin{eqnarray}
		\label{eq:rho}
		h(\bx) \geq h(\bx')+\left\langle \nabla h(\bx'),\bx-\bx' \right\rangle-\frac{\rho}{2}\|\bx'-\bx\|^2.
	\end{eqnarray} 
	\normalsize
	
	\begin{definition}[proximal mapping]\label{def:prox-map}
		Given a proper lower-semicontinuous convex function $h:\reals^d\rightarrow \reals\cup\{+\infty\}$, its \emph{proximal mapping} at $\bx$ is defined as
		\small
		$$
		\textstyle	\prox_{h}(\bx)=\argmin_{\bz\in\mathbb{R}^d}\left\{ h(\bz)+\frac{1}{2}\|\bz-\bx\|^2\right\}.
		$$
		\normalsize
	\end{definition}

	The following assumption on problem \eqref{eq:gco} is made throughout the paper.
	\begin{assumption}
		\label{assume:stochastic}
		The following statements hold:
		\begin{itemize}
			\item[A.] $f_i$ is $L_{f_i}$-smooth with $L_{f_i}\geq0$ for $i=0,1,\dots,m$; $c_j$ is $L_{c_j}$-smooth with $L_{c_j}\geq0$ for $j=1,\dots,n$;
			%			\begin{eqnarray}
			%			\label{eq:rhoL}
			%			-\frac{\rho}{2}\|\bx'-\bx\|^2\leq f_i(\bx)-[f_i(\bx')+\left\langle \nabla f_i(\bx'),\bx-\bx' \right\rangle] \leq\frac{L}{2}\|\bx'-\bx\|^2
			%			\end{eqnarray} 
			\item[B.] $\X$ is compact, and its diameter is denoted by $D=\max_{\bx,\bx'\in\X}\|\bx-\bx'\|$;
			%\item[B.] For each $i=0,1,\dots,m$, $\max\{|f_i(\bx)|,\|\nabla f_i(\bx)\|\} \leq B_{f_i},\, \forall \bx\in\X$ for a constant $B_{f_i}$; for each $j=1,\ldots,n$, $\max\{|c_j(\bx)|,\|\nabla c_j(\bx)\|\} \leq B_{c_j},\, \forall \bx\in\X$ for a constant $B_{c_j}$
			%\item[C.] $\|\bx-\bx'\|\leq D$ for a constant $D$ for all $\bx$ and $\bx'$ in $\mathcal{X}$.
			\item[C.] %$\mathrm{int}(\mathcal{X})\neq \emptyset$, 
			There exist constants $G$ and $M$ such that $|g(\bx)|\le G$, $\partial g(\bx)\neq\emptyset$, and $\partial g(\bx)\subseteq\mathcal{N}_{\mathcal{X}}(\bx)+ \cB_M, \forall \bx\in \mathcal{X}$. %\{\bzt\in\mathbb{R}^d|\|\bzt\|\leq M\}$	for a constant $M$.
			\item[D.] $\prox_{g}(\bx)$ can be computed easily, e.g., in a closed form.
			%\item[E.] $\bA$ has a full row-rank.
			%\item[F.] $f_{lb}\equiv \min_{\bx}f_0(\bx)+g(\bx)>-\infty$.
		\end{itemize}
	\end{assumption}
	With Assumption~\ref{assume:stochastic}A and \ref{assume:stochastic}B, there must exist constants $\{B_{f_i}\}_{i=0}^m$ and $\{B_{c_j}\}_{j=1}^n$ such that
	\small
	\begin{subequations}\label{eq:bd-cj-fi}
		\begin{align}
			&\max\big\{|f_i(\bx)|,\|\nabla f_i(\bx)\|\big\} \leq B_{f_i},\, \forall \bx\in\X, \, \forall \, i=0,1,\ldots,m,\label{eq:bd-fi}\\[0.1cm]
			&\max\big\{|c_j(\bx)|,\|\nabla c_j(\bx)\|\big\} \leq B_{c_j},\, \forall \bx\in\X , \, \forall \, j=1,\ldots,n.\label{eq:bd-ci}
		\end{align}
	\end{subequations}
	\normalsize
	Assumption~\ref{assume:stochastic}C holds, for exmaple, if $g(\bx)=r(\bx)+\mathbf{1}_{\X}(\bx)$, where $\mathbf{1}_{\X}$ is the indicator function on $\X$, and $r$ is a real-valued function with the norm of every subgradient bounded by $M$. In addition to Assumption~\ref{assume:stochastic}, we will make more assumptions on the (weak) convexity of the constraint functions. Details will be given in Sections~\ref{sec:penaltymethod} and \ref{sec:penaltymethodnonconvex}, where we conduct the complexity analysis.

	\subsection{Adaptive Accelerated Proximal Gradient Method}
	In each main iteration of the algorithm we propose for \eqref{eq:gco}, a \emph{convex composite optimization} subproblem of the form
	\small
	\begin{eqnarray}
		\label{eq:compositeopt}
		\min_{\bx\in\mathbb{R}^d}\big\{F(\bx):=\phi(\bx)+r(\bx)\big\}
	\end{eqnarray}
	\normalsize
	will be approximately solved, where $\phi:\mathbb{R}^d\rightarrow\mathbb{R}$ is $\mu_\phi$-strongly convex and $L_\phi$-smooth, and $r:\mathbb{R}^d\rightarrow\mathbb{R}\cup\{+\infty\}$ is a proper lower-semicontinuous convex function. We use a first-order method to solve the subproblems. 
	
	The standard \emph{accelerated proximal gradient} (APG) method (e.g. \cite{nesterov-2004-introductory}) requires knowing the exact values of $\mu_\phi$ and $L_\phi$ which may be unknown. To address this issue, \emph{adaptive accelerated proximal gradient} (AdapAPG) methods \cite{Nesterov2013,lin2015adaptive} have been developed 
	to dynamically estimate  $\mu_\phi$ and $L_\phi$ during the algorithm at the cost of a little higher complexity than APG. We apply the AdapAPG method in \cite{lin2015adaptive} to solve our subproblems. %The method is described as follows. 
	
	Given $\bw\in\mathbb{R}^d$ and a constant $L>0$, we define a local model of $\phi(\bx)$ as 
	\small
	\begin{eqnarray}
		\label{eq:quadmodel}
		\psi_{L}(\bw;\bx):=\phi(\bw)+\nabla \phi(\bw)^\top(\bx-\bw)+\frac{L}{2}\|\bx-\bw\|^2+r(\bx).
	\end{eqnarray}
	\normalsize
	Then the \emph{proximal gradient step} of \eqref{eq:compositeopt} at $\bw$ is defined as
	\small
	\begin{eqnarray}
		\label{eq:quadmodel_sol}
		T_{L}(\bw):=\argmin_{\bx\in\mathbb{R}^d}\psi_{L}(\bw;\bx)=\prox_{L^{-1}r}(\bw-L^{-1}\nabla \phi(\bw)),
	\end{eqnarray}
	\normalsize
	and the \emph{proximal gradient mapping} of \eqref{eq:compositeopt}  at $\bw$ is
	\small
	\begin{eqnarray}
		\label{eq:proximalgrad}
		g_L(\bw):=L(\bw-T_{L}(\bw)).
	\end{eqnarray}
	\normalsize
	%It can be show that
	%\begin{eqnarray}
	%\label{eq:quadmodel_upperbound}
	%L\geq L_{\phi}\Longrightarrow F(T_{L}(\bw))\leq\psi_{L}(\bw;T_{L}(\bw)).
	%\end{eqnarray}
	By the optimality of $T_{L}(\bw)$ and  $L_\phi$-smoothness of $\phi$, we can easily show that there exists $\bxi\in\partial r(T_{L}(\bw))$ such that 
	\small
	\begin{eqnarray}
		\label{eq:pgboundsubg}
		\omega(T_{L}(\bw))\leq\|\nabla \phi(T_{L}(\bw))+\bxi\|\leq \left(1+\frac{S_L(\bw)}{L}\right)\|g_L(\bw)\|,
	\end{eqnarray}
	\normalsize
	where
	\small
	\begin{eqnarray}
		\label{eq:subg}
		\omega(\bx):=\min_{\bxi'\in\partial r(\bx)}\|\nabla \phi(\bx)+\bxi'\|
	\end{eqnarray}
	\normalsize
	is a first-order optimality measure of $\bx$ and
	\small
	\begin{eqnarray}
		\label{eq:SL}
		\textstyle S_L(\bw):=\frac{\|\nabla \phi(T_{L}(\bw))-\nabla\phi(\bw)\|}{\|T_{L}(\bw)-\bw\|}\leq L_{\phi}
	\end{eqnarray}
	\normalsize
	is a local Lipschitz constant of $\phi$.
	Inequality \eqref{eq:pgboundsubg} means that, when $\|g_L(\bw)\|$ is small, the solution generated by a proximal gradient step from $\bw$ has a small subgradient and thus is a high-quality solution of \eqref{eq:compositeopt}.
	
	\begin{algorithm}[t]
		\caption{$\{\bx^{(t+1)},M_t,\bp^{(t)},S_t\}\gets\texttt{LineSearch}(\phi,r,\bx^{(t)},L_t)$}
		\label{alg:line-search}
		\begin{algorithmic}[1]
			\STATE {\bfseries Choose:}  $\gammaInc>1$
			\STATE $L \gets L_t / \gammaInc$
			\REPEAT
			\STATE $L \gets L \gammaInc$
			\STATE $\bx^{(t+1)} \gets T_L(\bx^{(t)})$
			\UNTIL{$F(\bx^{(t+1)}) \leq \psi_L(\bx^{(t)};\bx^{(t+1)})$}
			\STATE$M_t \gets L$
			\STATE$\bp^{(t)} \gets M_t (\bx^{(t)}-\bx^{(t+1)})$
			\STATE$S_t \gets S_L(\bx^{(t)})$
		\end{algorithmic}
	\end{algorithm}

	\begin{algorithm}[t]
		\caption{\\$\{\bx^{(t+1)}, M_t, \alpha_t, \bp^{(t)}, S_t\}\gets\texttt{AccelLineSearch}
			(\phi,r,\bx^{(t)}, \bx^{(t-1)}, L_t, \mu,\alpha_{t-1})$}
		\label{alg:accel-line-search}
		\begin{algorithmic}[1]
			\STATE {\bfseries Choose:} $\gammaInc>1$
			\STATE $L \gets L_t/ \gammaInc$
			\REPEAT
			\STATE $L \gets L \gammaInc$
			\STATE $\alpha_t \gets \sqrt{\frac{\mu}{L}}$
			\STATE $\bw^{(t)} \gets \bx^{(t)} + \frac{\alpha_t(1-\alpha_{t-1})}{\alpha_{t-1}(1+\alpha_t)}(\bx^{(t)}-\bx^{(t-1)})$ %\comm{changed $\alpha_{k-1}$ to $\alpha_{t-1}$ ql:I agree}
			\STATE $\bx^{(t+1)} \gets T_L(\bw^{(t)})$
			\UNTIL{$F(\bx^{(t+1)}) \leq \psi_L(\bw^{(t)};\bx^{(t+1)})$}
			\STATE $M_t \gets L$
			\STATE $\bp^{(t)} \gets M_t (\bw^{(t)}-\bx^{(t+1)})$
			\STATE $S_t \gets S_L(\bw^{(t)})$
		\end{algorithmic}
	\end{algorithm}

	\begin{algorithm}[tb]
		\caption{$\{\hat \bx, \hat M, \hat \mu\} \gets \texttt{AdapAPG} \,
			(\phi, r, \mathbf{\xini}, \Lini, \mu_0, \hat{\varepsilon} )$}
		\label{alg:adap-APG1}
		\begin{algorithmic}[1]
			\STATE {\bfseries Choose:} $\Lmin\geq\mu_0$, ~$\gammaDec\geq 1$, ~$\gammaSC>1$, ~$\thetaSC\in (0,1)$
			\STATE $\{\bx^{(0)},M_{-1},\bp^{(-1)},S_{-1}\}\gets\texttt{LineSearch}(\phi,r,\mathbf{\xini},\Lini)$
			%\STATE $\{x^{(0)},M_{-1},\alpha_{-1},G^{(-1)},S_{-1}\}\gets \texttt{AccelLineSearch}(\xini, \xini, \Lini, \mu_0, 1 )$
			\STATE $\bx^{(-1)}\gets \bx^{(0)}$, $L_{0}\gets M_{-1}$, %\comm{should $L_{-1}$ be $L_0$?ql:I agree} 
			$\mu\gets\mu_0$, $\alpha_{-1} \gets 1$, $\tau_0 \gets 1$, $t\gets 0$
			\REPEAT
			\STATE$\{\bx^{(t+1)},M_t, \alpha_t, \bp^{(t)}, S_t\}\gets \texttt{AccelLineSearch}(\phi,r, \bx^{(t)}, \bx^{(t-1)}, L_t, \mu, \alpha_{t-1} )$
			\STATE  $\tau_{t+1}\gets\tau_t(1-\alpha_t)$
			\IF%(\tcp*[f]{restart from new $x^{(0)}$ with same $\mu$})
			{$\bigl\|\bp^{(t)}\|_2\leq \thetaSC \bigl\|\bp^{(-1)}\bigr\|_2$}
			\STATE$\bx^{(0)} \gets \bx^{(t+1)}$, $\bx^{(-1)} \gets \bx^{(t+1)}$, $L_{0}\gets M_t$, %\comm{should $L_{-1}$ be $L_0$?ql:I agree} 
			$\bp^{(-1)} \gets \bp^{(t)}$, $M_{-1} \gets M_t$, $S_{-1} \gets S_t$
			\STATE$t\gets 0$
			\ELSE\IF%(\tcp*[f]{restart from old $x^{(0)}$ with reduced $\mu$\!\!})
			{$2\sqrt{2\tau_t} \, \frac{M_t}{\mu}
				\left(1+\frac{S_{-1}}{M_{-1}}\right) \leq \thetaSC$}
			\STATE $\mu\gets \mu/\gammaSC$
			\STATE $t\gets 0$
			\ELSE%(\tcp*[f]{continue iteration without restart})
			\STATE$L_{t+1} \gets \max\{\Lmin, M_t/\gammaDec\}$
			\STATE$t \gets t+1$
			\ENDIF
			\ENDIF
			\UNTIL{$\omega(\bx^{(t+1)}) \leq \hat\varepsilon$ %\comm{How about $\|\nabla \phi(\bx^{t+1}) -\nabla \phi(\bx^t) - L_{t+1}(\bx^{t+1}-\bx^t)\|\le \hat\varepsilon$?} 
			}
			\STATE$\hat{\bx} \gets \bx^{(t+1)}$, ~$\hat{M} \gets M_t$, ~$\hat{\mu} \gets \mu$
		\end{algorithmic}
	\end{algorithm}
	
	With these notations, we briefly describe the AdapAPG method in Algorithm~\ref{alg:adap-APG1}, where we treat $\phi$ and $r$ as the input because we will apply it to different instances of \eqref{eq:compositeopt}. We refer the interested readers to \cite{lin2015adaptive} for details. AdapAPG calls two different line-search schemes that are described in Algorithm~\ref{alg:line-search} and Algorithm~\ref{alg:accel-line-search}, respectively. Here, Algorithm~\ref{alg:line-search} is only used for initialization while Algorithm~\ref{alg:accel-line-search} is the main subroutine in each iteration of Algorithm~\ref{alg:adap-APG1}. AdapAPG maintains and updates estimations of $\mu_\phi$ and $L_\phi$. In iteration $t$, the AdapAPG method
	calls Algorithm~\ref{alg:accel-line-search}, which performs the updating steps of the APG method~\cite{nesterov-2004-introductory} by using an estimation of $\mu_\phi$, denoted as $\mu$, in place of $\mu_\phi$ and using a line search scheme to update the estimation of $L_\phi$, denoted as $M_t$.  After the $t$th call of  Algorithm~\ref{alg:accel-line-search}, the AdapAPG method stores $\bp^{(t)}=g_{M_{t}}(\bw^{(t)})$ and $\bx^{(t+1)}=T_{M_{t}}(\bw^{(t)})$. It can be shown that, if $\mu\leq \mu_\phi$, the value of $\|\bp^{(t)}\|$ should decrease geometrically to zero. Therefore, if such a decrease is not observed, it must happen that $\mu>\mu_\phi$. Then the algorithm is restarted with $\mu$ divided by $\gammaSC>1$, which leads to the adaptivity to the unknown $\mu_\phi$. The following theorem shows the complexity of the AdapAPG method for solving \eqref{eq:compositeopt}.
	
	%\comm{ql:Let's not formally define condition A and B. We can certainly do it if you want but my goal here is to quickly bring up the complexity of AdapAPG without distracting the reader too much before they see the main algorithm.}
	\begin{theorem}[Theorem 2 in \cite{lin2015adaptive}]
		\label{thm:conv-adap-APG1}
		Assume $\mu_0 \geq \mu_\phi > 0$. Let $\bp^{\text{ini}}$ denote the first $\bp^{(-1)}$ computed by
		Algorithm~\ref{alg:adap-APG1}. The total number of proximal gradient steps performed by Algorithm~\ref{alg:adap-APG1} is at most
		\small
		% 	\begin{equation}\label{eq:complexity-adapapg}
		% 	\begin{aligned}
		% 	\left( \left\lceil \log_{1/\theta_{sc}} \!\left(\left(1+\frac{L_{\phi}}{\Lmin}\right)
		% 	\frac{\|\bp^{\text{ini}}\|}{\hat\varepsilon}\right)\right\rceil +  \left\lceil \log_{\gammaSC} \! \left( \frac{\mu_0}{\mu_{\phi}}\right) \right\rceil\right)\cdot
		% 	\sqrt{\frac{L_{\phi}\gammaInc \gamma_{sc}}{\mu_{\phi}}}
		% 	\,\ln \left(8\left(\frac{L_{\phi}\gammaInc  \gamma_{sc}}{\mu_{\phi}\theta_{sc}}\right)^2
		% 	\left(1+\frac{L_{\phi}}{\Lmin}\right)^2\right) .
		% 	\end{aligned}
		% 	\end{equation} 
		\begin{align} \label{eq:complexity-adapapg}
			&\left( \left\lceil \log_{1/\theta_{sc}} \!\left(\left(1+\frac{L_{\phi}}{\Lmin}\right)
			\frac{\|\bp^{\text{ini}}\|}{\hat\varepsilon}\right)\right\rceil +  \left\lceil \log_{\gammaSC} \! \left( \frac{\mu_0}{\mu_{\phi}}\right) \right\rceil\right)
			\,\\ \nonumber
			&\times 	\sqrt{\frac{L_{\phi}\gammaInc \gamma_{sc}}{\mu_{\phi}}} \ln \left(8\left(\frac{L_{\phi}\gammaInc  \gamma_{sc}}{\mu_{\phi}\theta_{sc}}\right)^2
			\left(1+\frac{L_{\phi}}{\Lmin}\right)^2\right) .  
		\end{align}
		\normalsize
	\end{theorem}
	%\comm{ql:$\|\bp^{\text{ini}}\|_2$ is just $\|\bp^{\text{ini}}\|$. I add a remark below to address your question about what if  $\mu_0 \leq \mu_\phi$.  }
	\begin{remark}
		If $0<\mu_0 < \mu_\phi$, following the same proof as for Theorem 2 in \cite{lin2015adaptive}, we can show that 
		the total number of proximal gradient steps performed by Algorithm~\ref{alg:adap-APG1} is at most
		\small
		\begin{align*}
			\left\lceil \log_{1/\theta_{sc}} \!\left(\left(1+\frac{L_{\phi}}{\Lmin}\right)
			\frac{\|\bp^{\text{ini}}\|}{\hat\varepsilon}\right)\right\rceil 
			\sqrt{\frac{L_{\phi}\gammaInc \gamma_{sc}}{\mu_0}}
			\,\ln \left(8\left(\frac{L_{\phi}\gammaInc  \gamma_{sc}}{\mu_0\theta_{sc}}\right)^2
			\left(1+\frac{L_{\phi}}{\Lmin}\right)^2\right),
		\end{align*}
		\normalsize
		which is obtained by replacing $\mu_{\phi}$ by $\mu_0$ in \eqref{eq:complexity-adapapg}. %the complexity in Theorem~\ref{thm:conv-adap-APG1}.
	\end{remark}

	In this paper, we measure the \emph{complexity} of an algorithm using the total number of proximal gradient steps it performs. According to Theorem~\ref{thm:conv-adap-APG1}, the total complexity of the AdapAPG method to find a solution $\bar\bx$ to  \eqref{eq:compositeopt} satisfying $\omega(\bar\bx)\leq \hat\varepsilon$
	is %\comm{Should the term $\log\left(\frac{\kappa_\phi}{\hat\varepsilon}\right)$ be $\log\left(\frac{1}{\hat\varepsilon}\right)$? I think it should be $\log\left(\frac{\kappa_\phi}{\hat\varepsilon}\right)$ because I need to bound $\frac{L_{\phi}}{\Lmin}\leq \frac{L_{\phi}}{\mu_{\phi}}=\kappa_\phi$}
	$
	O\left(\kappa_\phi^{1/2}\log(\kappa_\phi)
	\log\left(\frac{1}{\hat\varepsilon}\right)\right),
	%+O\left(\kappa_\phi^{1/2}\log(\kappa_\phi)\right),
	$
	where $\kappa_\phi=\frac{L_\phi}{\mu_{\phi}}$ is the condition number of \eqref{eq:compositeopt}. Compared to APG whose complexity is $O\left(\kappa_\phi^{1/2}\log\left(\frac{1}{\hat\varepsilon}\right)\right)$, AdapAPG has an additional factor of $\log(\kappa_\phi)$ in the complexity result, but the latter does not require knowing the exact values of $\mu_\phi$ and $L_\phi$.

	\section{Inexact proximal-point penalty methods }
	\label{sec:algorithm}
	In this section, we describe the inexact proximal-point penalty (iPPP) method for \eqref{eq:gco} in details. This method incorporates the ideas of the proximal point method and the quadratic penalty method by iteratively updating the estimated solution $\bar\bx^{(k)}$ as follows
	\small
	\begin{equation}
		%\nonumber
		\bar\bx^{(k+1)}\approx\widetilde\bx^{(k)}\label{eq:iALMsub}
		:=\argmin_{\bx\in\mathbb{R}^d} f_0(\bx)+g(\bx)+\frac{\gamma_k}{2}\|\bx-\bar\bx^{(k)}\|^2+\frac{\beta_k}{2}\left(\|\vc(\bx)\|^2+\big\|[\vf(\bx)]_+\big\|^2\right),
	\end{equation}
	\normalsize
	where $\beta_k>0$ is the penalty parameter, and $\gamma_k>0$ is the proximal parameter. %with a fixed $\rho$ and a dynamic $\beta_k$. 
	When all $f_i$'s and $c_j$'s are weakly convex,  a sufficiently large $\gamma_k$ can be chosen such that the minimization problem in \eqref{eq:iALMsub} becomes strongly convex, and Algorithm \ref{alg:adap-APG1} can be applied to approximately solve it to obtain $\bar\bx^{(k+1)}$.  We formally describe our method in Algorithm~\ref{alg:iPPP}, where $\phi_k$ in \eqref{eq:compositepart} is the smooth part of the objective function in \eqref{eq:iALMsub}. 
	\begin{algorithm}
		\caption{Inexact Proximal-Point Penalty (iPPP) Method for \eqref{eq:gco}} % with Convex Constraints}
		\label{alg:iPPP}
		\begin{algorithmic}[1]
			\normalsize
			\STATE {\bfseries Input:} Initial solution $\bar \bx^{(0)}$, initial estimation of Lipschitz constant $M^{\text{ini}}$,  initial estimation of strong convexity constant $\mu^{\text{ini}}$, proximal parameters $\{\gamma_k\}_{k\geq0}$, penalty parameters $\{\beta_k\}_{k\geq0}$, and the targeted optimality measure for subproblems $\{\hat\varepsilon_k\}_{k\geq0}$.
			\STATE $\hat M\leftarrow M^{\text{ini}}$ and $\hat \mu\leftarrow \mu^{\text{ini}}$
			\FOR{$k=0,1,\dots,$}
			\STATE Let
			\small
			\begin{eqnarray}
				\label{eq:compositepart}
				\phi_k(\bx):=f_0(\bx)+\frac{\gamma_k}{2}\|\bx-\bar\bx^{(k)}\|^2+\frac{\beta_k}{2}\left(\|\vc(\bx)\|^2+\big\|[\vf(\bx)]_+\big\|^2\right),%\quad\mathrm{ and }\quad r(\bx)=g(\bx).
			\end{eqnarray}
			\normalsize
			\STATE Call Alg.~\ref{alg:adap-APG1}: $\{\bar\bx^{(k+1)}, \hat M, \hat \mu\} \gets \texttt{AdapAPG} \,
			(\phi_k, g, \bar\bx^{(k)}, \hat M, \hat \mu, \hat{\varepsilon}_k )$%, where $\phi_k$ is given in \eqref{eq:compositepart}.
			\STATE Set $\bar\by^{(k+1)}\gets\beta_k \vc(\bar\bx^{(k+1)})$ and $\blambda^{(k+1)}\gets\beta_k[\vf(\bar\bx^{(k+1)})]_+$
			\STATE Compute $\textbf{S}_{k+1}$, $\textbf{F}_{k+1}$, and $\textbf{C}_{k+1}$ by
			\begin{subequations}\label{eq:def-sfc}
				\small
				\begin{align}
					\label{eq:def-S}
					&\textbf{S}_{k+1}\gets\mathrm{dist}\big(\nabla f_0(\bar\bx^{(k+1)})+J_{\vf}(\bar\bx^{(k+1)})^\top\bar\blambda^{(k+1)}+J_{\vc}(\bar\bx^{(k+1)})^\top\bar \by^{(k+1)}, ~ -\partial g(\bar\bx^{(k+1)})\big)\\
					&\textbf{F}_{k+1}\gets\sqrt{\|\vc(\bar\bx^{(k+1)})\|^2+\big\|[\vf(\bar\bx^{(k+1)})]_+\big\|^2}\label{eq:def-F}\\
					&\textbf{C}_{k+1}\gets\sum_{i=1}^m|\bar\lambda^{(k+1)}_i f_i(\bar\bx^{(k+1)})|. \label{eq:def-C}
				\end{align}
				\normalsize
			\end{subequations}
			%\STATE $\textbf{F}_{k+1}\gets\sqrt{\|\vc(\bar\bx^{(k+1)})\|^2+\big\|[\vf(\bar\bx^{(k+1)})]_+\big\|^2}$
			%\STATE $\textbf{C}_{k+1}\gets\sum_{i=1}^m|\bar\lambda^{(k+1)}_i f_i(\bar\bx^{(k+1)})|$
			\STATE Generate index $R_{k+1}$ by one of the two options:
			\STATE ~~\textbf{Option I:} $R_{k+1}= \underset{1\le l \leq k+1}\argmin\max\left\{	\textbf{S}_l,\textbf{F}_l,\textbf{C}_l\right\}$
			\STATE ~~\textbf{Option II:} $R_{k+1}= \underset{1\le l \leq k+1}\argmin\max\left\{	\textbf{S}_l,\textbf{F}_l\right\}$
			\IF{ a stopping condition is satisfied} \STATE {Return $\bar\bx^{(R_{k+1})}$ and stop}  \ENDIF
			\ENDFOR
			%\STATE Return $\bar\bx^{(R_{k+1})}$
		\end{algorithmic}
	\end{algorithm}
	
	%The output $\bar x^{(R)}$ of Algorithm~\ref{alg:iPPP} is a random vector due to the randomness of index $R_k$. 
	%and thus we will analyze its behavior in expectation.
	In the rest of the paper, we analyze the theoretical properties of Algorithm \ref{alg:iPPP}. For technical reasons, we consider two different cases. In the first case, we assume that the objective function is weakly convex while the constraint functions are  convex.  In the second case, we assume that the objective function and the constraint functions are all weakly convex. The parameters $\{\gamma_k\}$ and $\{\beta_k\}$ will be chosen differently for the two cases. We will show that the output of Algorithm~\ref{alg:iPPP} with appropriate settings is an $\varepsilon$-stationary point or a weak $\varepsilon$-stationary point of \eqref{eq:gco} in each case. We will also analyze the computational complexity of Algorithm \ref{alg:iPPP}, measured by the total number of proximal gradient steps it performs.
	
	\begin{remark}
		Computing $\textbf{S}_{k+1}$ in \eqref{eq:def-S} requires projection onto $-\partial g(\bar\bx^{(k+1)})$, which is also needed when evaluating the stopping condition $\omega(\bx^{(t+1)})$ in Algorithm~\ref{alg:adap-APG1}. We have assumed $g$ is simple enough to allow a closed-form solution for the proximal gradient step. For such $g$, complexity of this projection is usually no more than a proximal gradient step, e.g., when $g(\bx)=\mathbf{1}_{\X}(\bx)$ where $\X$ is a Euclidean ball or a box. 
		%Moreover, this projection can be even avoided by slightly modifying the definition of $\textbf{S}_{k+1}$ without changing any result of this paper. In fact, according to the stopping condition of Algorithm~\ref{alg:adap-APG1}, there exists $\bar\bxi^{(k+1)}\in\partial g(\bar\bx^{(k+1)})$ such that $\|\nabla \phi_k(\bar\bx^{(k+1)})+\bar\bxi^{(k+1)}\|\le \hat\varepsilon_k$. Note that this $\bar\bxi^{(k+1)}$ will be computed when 
	\end{remark}
	
	\begin{lemma}
		Suppose $\phi_k$ in \eqref{eq:compositepart} is convex. Let $\{\bar\bx^{(k)}\}$ be generated from Algorithm~\ref{alg:iPPP}. Then 
		for any $\bx\in\mathcal{X}$, it holds that
		\small
		\begin{eqnarray}
			\label{eq:nesterovconv1}
			\phi_k(\bar\bx^{(k+1)})+g(\bar\bx^{(k+1)})-\phi_k(\bx)-g(\bx)\leq \hat{\varepsilon}_kD,~\forall k\geq 0, \text{ and }
		\end{eqnarray}
		\normalsize
		{\small
			\begin{align}\label{eq:nesterovconv2-vary-beta-nc-constant}
				%\nonumber
				&~\sum_{k=0}^{K-1}\frac{\gamma_k}{2}\|\bar \bx^{(k+1)}-\bar\bx^{(k)}\|^2+\frac{\beta_{K-1}}{2}\left(\|\vc(\bar \bx^{(K)})\|^2+\big\|[\vf(\bar \bx^{(K)})]_+\big\|^2\right)\\\nonumber
				\le &~2B_{f_0}+2G+\frac{\beta_0}{2}\biggl(\|\vc(\bx^{(0)})\|^2+\big\|[\vf(\bar \bx^{(0)})]_+\big\|^2\biggr)\\\nonumber &
				+\frac{1}{2}\sum_{k=1}^{K-1}(\beta_k-\beta_{k-1})\biggl(\|\vc(\bx^{(k)})\|^2 +\big\|[\vf(\bar \bx^{(k)})]_+\big\|^2\biggr)+\left(\sum_{k=0}^{K-1}\hat{\varepsilon}_k\right)D,~\forall K\geq 1.
			\end{align}
		}
	\end{lemma}
	\begin{proof}
		According to the termination condition of Algorithm~\ref{alg:adap-APG1} and the definition of $\omega(\cdot)$ in \eqref{eq:subg}, there exists $\bar\bxi^{(k+1)}\in\partial g(\bar\bx^{(k+1)})$ such that $\|\nabla \phi_k(\bar\bx^{(k+1)})+\bar\bxi^{(k+1)}\|\leq \hat{\varepsilon}_k$. Since $\phi_k$ is convex, so is $\phi_k+g$. Hence, we obtain \eqref{eq:nesterovconv1} by noting
		\small
		\begin{eqnarray*}
			\nonumber
			%\mathcal{L}_{\beta_k}(\bar\bx^{(k+1)};\bar\bx^{(k)})-\mathcal{L}_{\beta_k}(\bx;\bar\bx^{(k)})
			&&\phi_k(\bar\bx^{(k+1)})+g(\bar\bx^{(k+1)})-\phi_k(\bx)-g(\bx)\\
			&\leq& \left(\nabla \phi_k(\bar\bx^{(k+1)})+\bar\bxi^{(k+1)}\right)^\top(\bar\bx^{(k+1)}-\bx)
			\leq\hat{\varepsilon}_k\|\bar\bx^{(k+1)}-\bx\|\leq \hat{\varepsilon}_kD,
		\end{eqnarray*}
		\normalsize
		which gives \eqref{eq:nesterovconv1}. Now let $\bx=\bar\bx^{(k)}$ in \eqref{eq:nesterovconv1} and sum it over $k=0$ through $K-1$ to obtain \eqref{eq:nesterovconv2-vary-beta-nc-constant} by Assumption~\ref{assume:stochastic} and the equation \ref{eq:bd-fi}.	
	\end{proof}
	%For any $K\ge1$, we let $\bx=\bar\bx^{(k)}$ in \eqref{eq:nesterovconv1} and sum it over $k=0$ through $K-1$ to have by the definition of $\phi_k$ that
	
	By the definitions of $\textbf{S}_{k+1}$, $\textbf{F}_{k+1}$, and $\textbf{C}_{k+1}$ in \eqref{eq:def-sfc}, we have for any $K\ge1$ that if $\{R_k\}$ is chosen by Option I in Algorithm~\ref{alg:iPPP}, then
	{\small
		\begin{align}
			\nonumber
			\textstyle	\max\left\{\textbf{S}_{R_K},\textbf{F}_{R_K},\textbf{C}_{R_K}\right\}
			\leq& \textstyle \frac{1}{K}\sum_{k=0}^{K-1}\max\left\{	\textbf{S}_{k+1},\textbf{F}_{k+1},\textbf{C}_{k+1}\right\} \\
			\leq& \textstyle \frac{1}{K}\sum_{k=0}^{K-1}\textbf{S}_{k+1}
			+\frac{1}{K}\sum_{k=0}^{K-1}\textbf{F}_{k+1}+\frac{1}{K}\sum_{k=0}^{K-1}\textbf{C}_{k+1},\label{eq:threeterms}
		\end{align}	
	}
	and if $\{R_k\}$ is chosen by Option II in Algorithm~\ref{alg:iPPP}, then	
	{\small
		\begin{equation}
			\textstyle		\max\left\{\textbf{S}_{R_K},\textbf{F}_{R_K}\right\}\leq\frac{1}{K}\sum_{k=0}^{K-1}\max\left\{	\textbf{S}_{k+1},\textbf{F}_{k+1}\right\}
			\leq\frac{1}{K}\sum_{k=0}^{K-1}\textbf{S}_{k+1}
			+\frac{1}{K}\sum_{k=0}^{K-1}\textbf{F}_{k+1}.\label{eq:twoterms}
		\end{equation}
	}
	
	\begin{lemma}
		Let $\{\textbf{S}_{k+1}\}$ be defined in \eqref{eq:def-S}. Then for any $K\ge1$,
		\begin{align}\label{eq:xtilde-x-dual-feas-option3}
			\textstyle	\frac{1}{K}\sum_{k=0}^{K-1}\textbf{S}_{k+1}
			\leq\frac{1}{K}\sum_{k=0}^{K-1}\hat\varepsilon_k+\frac{1}{K}\sum_{k=0}^{K-1}\gamma_k\| \bar\bx^{(k+1)}-\bar\bx^{(k)}\|.
		\end{align}
	\end{lemma}
	
	\begin{proof}
		%	\small
		%	$$
		%	\left\|
		%	\begin{array}{c}
		%	\nabla f_0(\bar\bx^{(k+1)})+\bar\bxi^{(k+1)} + \gamma_k (\bar\bx^{(k+1)}-\bar\bx^{(k)})\\
		%	+\beta_k J_\vc(\bar\bx^{(k+1)})^\top \vc(\bar\bx^{(k+1)})+\beta_k J_\vf(\bar\bx^{(k+1)})^\top[\vf(\bar\bx^{(k+1)})]_+
		%	\end{array}
		%	\right\|\leq\hat\varepsilon_k.
		%	$$
		%	\normalsize
		First, note 
		\small\begin{align}\label{eq:grad-phik}
			\nabla \phi_k(\bar\bx^{(k+1)}) = & ~\nabla f_0(\bar\bx^{(k+1)})+\gamma_k (\bar\bx^{(k+1)}-\bar\bx^{(k)})+\beta_k J_\vc(\bar\bx^{(k+1)})^\top \vc(\bar\bx^{(k+1)}) \\
			&+\beta_k J_\vf(\bar\bx^{(k+1)})^\top [\vf(\bar\bx^{(k+1)})]_+.\nonumber
		\end{align}
		\normalsize
		Second, according to the stopping condition of Algorithm~\ref{alg:adap-APG1}, there exists $\bar\bxi^{(k+1)}\in\partial g(\bar\bx^{(k+1)})$ such that $\|\nabla \phi_k(\bar\bx^{(k+1)})+\bar\bxi^{(k+1)}\|\le \hat\varepsilon_k$, which, by the definition of $\bar\by^{(k+1)}$ and $\bar\blambda^{(k+1)}$ in Algorithm~\ref{alg:iPPP}, implies 
		\small
		\begin{align*}
			&~\left\|\nabla f_0(\bar\bx^{(k+1)})+\bar\bxi^{(k+1)}+J_\vc(\bar\bx^{(k+1)})^\top\bar\by^{(k+1)}+J_\vf(\bar\bx^{(k+1)})^\top \bar\blambda^{(k+1)}\right\| \\
			\leq &~\hat\varepsilon_k+\gamma_k\| \bar\bx^{(k+1)}-\bar\bx^{(k)}\|.
		\end{align*}
		\normalsize
		Hence, by the definition of $\textbf{S}_{k+1}$, we have $\textbf{S}_{k+1} \le \hat\varepsilon_k+\gamma_k\| \bar\bx^{(k+1)}-\bar\bx^{(k)}\|$,	
		which, after taking average over $k=0,\dots,K-1$, implies the desired result.
	\end{proof}

	\section{Complexity with convex constraints}
	\label{sec:penaltymethod}
	Throughout this section, we assume that $f_i$ is convex for each $i=1,\ldots,m$ and $c_j$ is affine for each $j=1,\ldots, n$, namely, we consider the following problem with only convex constraints: 
	\begin{equation}
		\label{eq:gco-cvx}
		\min_{\bx\in\mathbb{R}^d}f_0(\bx)+g(\bx),\quad\mathrm{s.t.}\quad \bA\bx=\bb,
		\quad  \vf(\bx)\leq \mathbf{0}, 
	\end{equation}
	where $\bA\in\RR^{n\times d}$ and $\bb\in\RR^n$ are given. Without loss of generality, we assume $\bA$ to be full row-rank. Otherwise, redundant affine equality constraints can be removed. In addition to 
	Assumption~\ref{assume:stochastic}, we make the following assumption.
	\begin{assumption}
		\label{assume:convexconstraint}
		The following statements hold:	
		\begin{itemize}
			\item[A.]	$f_0$ is $\rho_0$-weakly convex for $\rho_0\ge 0$ and $f_i$ is convex for $i=1,\dots,m$.
			%\item[B.] $f_i$ is convex for each $i=1,\dots,m$; $c_j$ is affine, i.e., $c_j(\bx)=\va_j^\top \vx - b_j$, for each $j=1,\ldots,n$ 
			\item[B.] There exists $\bx_{\mathrm{feas}}\in\mathrm{int}(\X)$ satisfying $\bA\bx_{\mathrm{feas}}=\bb$ and $\vf(\bx_{\mathrm{feas}})<\mathbf{0}$. 
		\end{itemize}
	\end{assumption}
	
	Here, we only require the existence of $\bx_{\mathrm{feas}}$ but not its availability to our algorithm. In addition, the assumption on $\mathrm{int}(\X)\neq \emptyset$ does not lose generality. If $\X$ does not have a full dimension, it can be written as $\X'\cap\{\bx\in\RR^d: \mathbf{C}\bx = \mathbf{d}\}$ for some full-dimensional convex compact set $\X'\subset \RR^d$. Then, we can put $\mathbf{C}\bx = \mathbf{d}$ as a part of the affine constraints and replace $\X$ with $\X'$.  
	
	Under Assumptions~\ref{assume:stochastic} and \ref{assume:convexconstraint}, the function $\phi_k$ in \eqref{eq:compositepart} is $L_{\phi_k}$-smooth with
	\begin{equation}
		\label{eq:Lk}
		\textstyle L_{\phi_k} = L_{f_0}+\gamma_k+\beta_k\left(\|\bA^\top\bA\|+\sum_{i=1}^m B_{f_i}(B_{f_i}+L_{f_i})\right)
	\end{equation}
	and is $(\gamma_k-\rho_0)$-strongly convex if $\gamma_k > \rho_0$. To facilitate our analysis in this section, we define
	\small
	\begin{equation}\label{eq:phit}
		\widehat\bx^{(k)}\equiv\argmin_{\bx\in\mathbb{R}^d}\left\{f_0(\bx)+g(\bx)+\frac{\gamma_k}{2}\|\bx-\bar\bx^{(k)}\|^2, \quad\mathrm{s.t.}\quad \bA\bx=\bb,~\quad  \vf(\bx)\leq\mathbf{0}\right\}.
	\end{equation}
	\normalsize
	
	\subsection{Technical lemmas}\label{sec:bd-dual}
	From Assumption~\ref{assume:convexconstraint}B, we have Slater's condition, and thus $\widehat\bx^{(k)}$ must be a KKT point of \eqref{eq:phit}, i.e., there  are $\widehat\bzt^{(k)}\in \partial g(\widehat\bx^{(k)})$,  Lagrangian multipliers $\widehat\by^{(k)}\in\RR^n$, and $\widehat\blambda^{(k)}\in\RR^m$  associated to $\widehat\bx^{(k)}$ such that (c.f. \cite[Theorem 28.2]{rockafellar1970convex}): 
	\small
	\begin{subequations}
		\label{eq:KKTproxk}
		\begin{align}
			\label{eq:KKTproxk1}
			\nabla f_0(\widehat\bx^{(k)})+\widehat\bzt^{(k)}+\gamma_k(\widehat\bx^{(k)}-\bar\bx^{(k)})+\bA^\top\widehat\by^{(k)}+J_\vf(\widehat\bx^{(k)})^\top\widehat\blambda^{(k)}=\mathbf{0},\\\label{eq:KKTproxk2}
			\widehat\blambda^{(k)}\geq \mathbf{0}, \quad\bA\widehat\bx^{(k)}=\bb,
			\quad \vf(\widehat\bx^{(k)})\leq \mathbf{0},\\
			\widehat\lambda_i^{(k)}f_i(\widehat\bx^{(k)})=0,\, i=1,\dots,m. \label{eq:KKTproxk3}
		\end{align}
	\end{subequations}
	\normalsize
	We next prove the  boundedness of $(\widehat\by^{(k)}, \widehat\blambda^{(k)})$. %under Assumption \ref{assume:convexconstraint}.
	%Such Lagrangian multipliers exist when $\rho>\rho_0$ because of the Slater's condition we assume in Assumption~\ref{assume:stochastic}. 
	%Moreover, whenever there exist multiple  $\widehat\by^{(k)}$ and $\widehat\blambda^{(k)}$ satisfying \eqref{eq:KKTproxk1} and \eqref{eq:KKTproxk2}, we choose  $\widehat\by^{(k)}$ and $\widehat\blambda^{(k)}$ to be the ones with the smallest $\|\widehat\by^{(k)}\|$.
	%Furthermore, the Slater's condition also implies the boundness of $\widehat\by^{(k)}$ and $\widehat\blambda^{(k)}$.
	
	%The following lemma shows that $\widehat\by^{(k)}$ and $\widehat\blambda^{(k)}$ are uniformly bounded for all $k$ under Assumption~\ref{assume:stochastic}.
	\begin{lemma}
		\label{thm:boundlambda}
		Suppose Assumptions~\ref{assume:stochastic} and \ref{assume:convexconstraint} hold.
		Let $(\widehat\bx^{(k)},\widehat\by^{(k)}, \widehat\blambda^{(k)})$ be the solution satisfying the conditions in \eqref{eq:KKTproxk} for $k\ge 0$. Then 
		{\small
			\begin{align}\label{eq:def-Mlam}
				\|\widehat\blambda^{(k)}\|\leq& M_\lambda(\gamma_k):=\frac{Q_k}{\min_{i}|f_i(\bx_{\mathrm{feas}})| }\\
				\label{eq:def-My}
				\|\widehat\by^{(k)}\|\leq& M_y(\gamma_k):=\frac{Q_k\sqrt{\lambda_{\max}(\bA\bA^\top)}}{\lambda_{\min}(\bA\bA^\top)}\left(\frac{1}{D}+\frac{1}{\mathrm{dist}(\bx_{\mathrm{feas}},\partial\mathcal{X})}+\frac{\max_{i}B_{f_i}}{\min_{i}|f_i(\bx_{\mathrm{feas}})| }\right),
			\end{align}
		}
		where $Q_k=D(B_{f_0}+\gamma_k D+M)$.	
	\end{lemma}
	\begin{proof}
		Let $\bx_{\mathrm{feas}}$ be the point in Assumption~\ref{assume:convexconstraint}. Then from the convexity of $\{f_i\}_{i=1}^m$ and the fact $\widehat\blambda^{(k)}\ge\vzero$, it follows that
		\small
		$$\textstyle \sum_{i=1}^m\widehat\lambda_i^{(k)}f_i(\bx_{\mathrm{feas}}) 
		\geq\sum_{i=1}^m\widehat\lambda_i^{(k)}\left[f_i(\widehat\bx^{(k)})+(\bx_{\mathrm{feas}}-\widehat\bx^{(k)})^\top\nabla f_i(\widehat\bx^{(k)})\right].$$
		\normalsize
		The above inequality together with \eqref{eq:KKTproxk1} and \eqref{eq:KKTproxk3} yields
		\small
		\begin{align}
			\nonumber
			\sum_{i=1}^m\widehat\lambda_i^{(k)}f_i(\bx_{\mathrm{feas}}) 
			\geq &~-(\bx_{\mathrm{feas}}-\widehat\bx^{(k)})^\top\left[\nabla f_0(\widehat\bx^{(k)}) +\widehat\bzt^{(k)}  + \bA^\top\widehat\by^{(k)} + \gamma_k(\widehat\bx^{(k)} -\bar\bx^{(k)} )\right]\\\nonumber
			=&~-(\bx_{\mathrm{feas}}-\widehat\bx^{(k)})^\top\left[\nabla f_0(\widehat\bx^{(k)}) +\widehat\bzt^{(k)}  + \gamma_k(\widehat\bx^{(k)} -\bar\bx^{(k)} )\right]\\
			\geq&~-(\bx_{\mathrm{feas}}-\widehat\bx^{(k)})^\top\widehat\bzt^{(k)}- DB_{f_0}-\gamma_k D^2,\label{eq:boundlmabda1}
		\end{align}
		\normalsize
		where the equality follows from $\bA(\bx_{\mathrm{feas}}-\widehat\bx^{(k)})=\bb-\bb=\mathbf{0}$ and the last inequality is by Assumption~\ref{assume:stochastic}B and \eqref{eq:bd-fi}.
		
		By Assumption~\ref{assume:stochastic}C, we have $\widehat\bzt^{(k)}=\widehat\bzt_1+\widehat\bzt_2$ with $\widehat\bzt_1\in\mathcal{N}_{\mathcal{X}}(\widehat\bx^{(k)})$ and $\|\widehat\bzt_2\|\leq M$, and thus
		\small
		\begin{equation}\label{eq:bd-bzt1}
			(\bx_{\mathrm{feas}}-\widehat\bx^{(k)})^\top\widehat\bzt^{(k)} \le (\bx_{\mathrm{feas}}-\widehat\bx^{(k)})^\top\widehat\bzt_1 + DM.
		\end{equation} 
		\normalsize
		Next we bound the term $(\bx_{\mathrm{feas}}-\widehat\bx^{(k)})^\top\widehat\bzt_1$. If $\widehat\bx^{(k)}\in\mathrm{int}(\mathcal{\X})$, then $\widehat\bzt_1=\mathbf{0}$. Hence, we only need to consider the case  %by considering two cases: (i) when $\widehat\bx^{(k)}\in\mathrm{int}(\mathcal{\X})$ and (ii) 
		when $\widehat\bx^{(k)}\in\partial\mathcal{\X}$ and $\widehat\bzt_1\neq\mathbf{0}$. In this case, $\mathcal{H}=\big\{\bx\in\mathbb{R}^d\,|\,(\bx-\widehat\bx^{(k)})^\top \widehat\bzt_1=0\big\}$ is a supporting hyperplane of $\mathcal{X}$ at $\widehat\bx^{(k)}$. Hence, $\mathrm{dist}(\bx_{\mathrm{feas}},\mathcal{H}) \ge \mathrm{dist}(\bx_{\mathrm{feas}},\partial \mathcal{X})>0$, and thus 
		\small
		$$(\widehat\bx^{(k)}-\bx_{\mathrm{feas}})^\top\widehat\bzt_1=\mathrm{dist}(\bx_{\mathrm{feas}},\mathcal{H})\|\widehat\bzt_1\| \ge \mathrm{dist}(\bx_{\mathrm{feas}},\partial \mathcal{X})\|\widehat\bzt_1\|.$$
		\normalsize
		Applying this inequality to \eqref{eq:bd-bzt1} and using \eqref{eq:boundlmabda1}, and also noting $f_i(\bx_{\mathrm{feas}})<0$, we have
		\small
		$$
		\sum_{i=1}^m\widehat\lambda_i^{(k)}|f_i(\bx_{\mathrm{feas}})| + \mathrm{dist}(\bx_{\mathrm{feas}},\partial \mathcal{X})\|\widehat\bzt_1\|  \le DB_{f_0}+\gamma_k D^2+ DM = Q_k.
		$$
		\normalsize
		The above inequality	 implies
		\small
		\begin{eqnarray}
			\label{eq:boundlambdak}
			\|\widehat\blambda^{(k)}\|\leq \|\widehat\blambda^{(k)}\|_1\leq \frac{\sum_{i=1}^m\widehat\lambda_i^{(k)}|f_i(\bx_{\mathrm{feas}})|}{\min_{i}|f_i(\bx_{\mathrm{feas}})| }
			\leq\frac{Q_k}{\min_{i}|f_i(\bx_{\mathrm{feas}})| },
		\end{eqnarray}
		\normalsize
		and
		\small
		\begin{equation}\label{eq:boundlmabda1-2}
			\|\widehat\bzt^{(k)}\| \leq \|\widehat\bzt_1\| + \|\widehat\bzt_2\| \leq \frac{Q_k}{\mathrm{dist}(\bx_{\mathrm{feas}},\partial\mathcal{X})} + M.
		\end{equation}	
		\normalsize
		
		Furthermore, since $\bA$ has a full row-rank, we have from
		\eqref{eq:KKTproxk1} that
		\small
		$$
		\widehat\by^{(k)} = -(\bA\bA^\top)^{-1}\bA\left(\nabla f_0(\widehat\bx^{(k)}) + \gamma_k(\widehat\bx^{(k)} -\bar\bx^{(k)} )+ J_\vf(\widehat\bx^{(k)})^\top\widehat\blambda^{(k)} +\widehat\bzt^{(k)}\right).
		$$
		\normalsize
		Therefore,
		\small
		\begin{align}
			\nonumber
			\|\widehat\by^{(k)}\|%\\\nonumber
			\leq&~ \frac{\sqrt{\lambda_{\max}(\bA\bA^\top)}}{\lambda_{\min}(\bA\bA^\top)}\left(\left\|\nabla f_0(\widehat\bx^{(k)}) + \gamma_k(\widehat\bx^{(k)} -\bar\bx^{(k)} )\right\|+\|\widehat\bzt^{(k)}\|+\left\|J_\vf(\widehat\bx^{(k)})^\top\widehat\blambda^{(k)}\right\|\right)\\\nonumber
			\leq&~\frac{\sqrt{\lambda_{\max}(\bA\bA^\top)}}{\lambda_{\min}(\bA\bA^\top)}\left(B_{f_0}+\gamma_k D+M+\frac{Q_k}{\mathrm{dist}(\bx_{\mathrm{feas}},\partial\mathcal{X})}+\|\widehat\blambda^{(k)}\|_1\max_{i}\|\nabla f_i(\widehat\bx^{(k)})\|\right)\\\nonumber
			%\leq&\frac{\sqrt{\lambda_{\max}(\bA\bA^\top)}}{\lambda_{\min}(\bA\bA^\top)}\left(B_{f_0}+\gamma_k D+M+\frac{DB_{f_0}+\gamma_k D^2+DM}{\mathrm{dist}(\bx_{\mathrm{feas}},\partial\mathcal{X})}+\max_{i}B_{f_i}\frac{DB_{f_0}+\gamma_k D^2+DM}{\min_{i}|f_i(\bx_{\mathrm{feas}})| }\right)\\\nonumber
			\le&~M_y(\gamma_k),
		\end{align}
		\normalsize
		where the second inequality is from \eqref{eq:boundlmabda1-2}, Assumption~\ref{assume:stochastic}B, and \eqref{eq:bd-fi}, and the third inequality is from  \eqref{eq:boundlambdak}, Assumption~\ref{assume:stochastic}B, and the definition of $M_y(\gamma_k)$.
	\end{proof}

	The next lemma bounds the feasibility violation of the solution returned by  Algorithm~\ref{alg:adap-APG1} when applied to \eqref{eq:iALMsub}.
	\begin{lemma}\label{thm:ippconverge}
		Suppose Assumptions~\ref{assume:stochastic} and \ref{assume:convexconstraint} hold.
		Given $\gamma_k>\rho_0$ and $\beta_k>0$ for $k\ge0$, let $\phi_k$ be defined in \eqref{eq:compositepart} with $\vc(\bx)=\bA\bx-\bb$, $\widehat\bx^{(k)}$ be defined in \eqref{eq:phit}, and $\bar\bx^{(k+1)}$ be generated as in Algorithm~\ref{alg:iPPP}. Then for any $k\geq0$, 
		\small
		\begin{equation}\label{eq:ippconverge-2}
			\textstyle	\|\bA\bar\bx^{(k+1)}-\bb\|^2+\big\|[\vf(\bar\bx^{(k+1)})]_+\big\|^2\leq \frac{4 \hat{\varepsilon}_kD}{\beta_k}+\frac{4\|\widehat\by^{(k)}\|^2}{\beta_k^2}+\frac{4\|\widehat\blambda^{(k)}\|^2}{\beta_k^2}.
		\end{equation}
		\normalsize
	\end{lemma}
	\begin{proof}
		Notice that when $\gamma_k>\rho_0$, $\phi_k$ in \eqref{eq:compositepart} is convex. Hence, letting	 $\bx=\widehat\bx^{(k)}$ in \eqref{eq:nesterovconv1}, we have from the feasibility of $\widehat\bx^{(k)}$ for \eqref{eq:phit} that
		\small
		\begin{align}
			\nonumber
			%&~f_0(\bar\bx^{(k+1)})+g(\bar\bx^{(k+1)})+\frac{\gamma_k}{2}\|\bar\bx^{(k+1)}-\bar\bx^{(k)}\|^2\\\nonumber
			%\leq
			f_0(\bar\bx^{(k+1)})+&~g(\bar\bx^{(k+1)})\frac{\gamma_k}{2}\|\bar \bx^{(k+1)}-\bar\bx^{(k)}\|^2%\\\nonumber
			+\frac{\beta_k}{2}\left(\|\bA\bar \bx^{(k+1)}-\bb\|^2 + \big\|[\vf(\bar \bx^{(k+1)})]_+\big\|^2\right)\\\label{eq:ippconverge-1}%\label{eq:nesterovconv2}
			\leq&~ f_0(\widehat\bx^{(k)})+g(\widehat\bx^{(k)})+\frac{\gamma_k}{2}\|\widehat\bx^{(k)}-\bar\bx^{(k)}\|^2+ \hat{\varepsilon}_kD.
		\end{align}
		\normalsize

		Recall that $\widehat\by^{(k)}$ and $\widehat\blambda^{(k)}$ are the Lagrangian multipliers satisfying \eqref{eq:KKTproxk}. Hence, from the convexity of the objective function of \eqref{eq:phit}, we have
		\small
		\begin{align}
			\nonumber
			\textstyle	f_0(\widehat\bx^{(k)})+g(\widehat\bx^{(k)})+\frac{\gamma_k}{2}\|\widehat\bx^{(k)}-\bar\bx^{(k)} \|^2%\\\nonumber
			\textstyle	\le&~f_0(\bar \bx^{(k+1)})+g(\bar \bx^{(k+1)})+\frac{\gamma_k}{2}\|\bar \bx^{(k+1)}-\bar\bx^{(k)}\|^2\\\nonumber%\label{eq:sdl-ineq-cvx}	
			&\textstyle+(\widehat\by^{(k)})^\top(\bA\bar \bx^{(k+1)}-\bb)+\sum_{i=1}^m\widehat\lambda_i^{(k)}f_i(\bar \bx^{(k+1)}).
		\end{align}
		\normalsize
		The above inequality, together with \eqref{eq:ippconverge-1}, implies
		\small
		\begin{align}\nonumber
			\textstyle	\hat{\varepsilon}_kD\geq&~\textstyle\frac{\beta_k}{2}\left(\|\bA\bar \bx^{(k+1)}-\bb\|^2+\big\|[\vf(\bar \bx^{(k+1)})]_+\big\|^2\right)-(\widehat\by^{(k)})^\top(\bA\bar \bx^{(k+1)}-\bb) \\ \label{eq:ineq-tmp1} &~\textstyle-\sum_{i=1}^m\widehat\lambda_i^{(k)}f_i(\bar \bx^{(k+1)}).%\\\nonumber
			%\geq&~\frac{\beta_k}{2}\left(\|\bA\bar \bx^{(k+1)}-\bb\|^2+\big\|[\vf(\bar \bx^{(k+1)})]_+\big\|^2\right).%\\\nonumber
			%&~-\frac{\|\widehat\by^{(k)}\|^2}{\beta_k}-\frac{\beta_k}{4}\|\bA\bar \bx^{(k+1)}-\bb\|^2-\sum_{i=1}^m\frac{(\widehat\lambda_i^{(k)})^2}{\beta_k}-\sum_{i=1}^m\frac{\beta_k}{4}[f_i(\bar \bx^{(k+1)})]_+^2,
		\end{align}
		\normalsize
		By the Young's inequality, it holds 
		\small
		\begin{align*}
			&\textstyle-(\widehat\by^{(k)})^\top(\bA\bar \bx^{(k+1)}-\bb)\ge-\frac{\|\widehat\by^{(k)}\|^2}{\beta_k}-\frac{\beta_k}{4}\|\bA\bar \bx^{(k+1)}-\bb\|^2,\\
			& \textstyle-\sum_{i=1}^m\widehat\lambda_i^{(k)}f_i(\bar \bx^{(k+1)})\ge -\sum_{i=1}^m\frac{(\widehat\lambda_i^{(k)})^2}{\beta_k}-\sum_{i=1}^m\frac{\beta_k}{4}[f_i(\bar \bx^{(k+1)})]_+^2.
		\end{align*}
		\normalsize
		Plugging the above two inequalities into \eqref{eq:ineq-tmp1} gives the desired result.  	
		%where we have applied the Young's inequality to obtain the second inequality. Therefore, the desired inequality \eqref{eq:ippconverge-2} follows.
	\end{proof}

	\subsection{The complexity of the iPPP method}
	In this subsection, we specify the parameters in  Algorithm~\ref{alg:iPPP} and estimate its complexity in order to find an $\varepsilon$-stationary point of \eqref{eq:gco-cvx}. %We first present a generic convergence property.
	
	\begin{theorem}\label{thm:complexity-varying-beta-new}
		Suppose that Assumptions~\ref{assume:stochastic} and \ref{assume:convexconstraint} hold and the parameters $\{\gamma_k\}$, $\{\beta_k\}$ and  $\{\hat{\varepsilon}_k\}$ in Algorithm~\ref{alg:iPPP} are chosen as 
		\begin{equation}\label{eq:set-beta-vary}
			\gamma_k=\gamma>\rho_0,\quad \beta_k=\beta\sqrt{k+1},\quad\text{ and } \quad\hat{\varepsilon}_k= \frac{1}{\beta_k (k+1)},
		\end{equation}  
		where $\beta>0$ is a constant. If $\{R_k\}$ is generated by Option I, it holds for any $K\ge1$ that 
		\small
		\begin{align}\label{eq:convergeinK1}
			%\nonumber
			%	&&\max\left\{
			%	\begin{array}{c}
			%	\left\|\nabla f_0(\bar\bx^{(R_K)})+\bA^\top\bar\by^{(R_K)}+\sum_{i=1}^m \bar\lambda^{(R_K)}_i \nabla f_i(\bar\bx^{(R_K)})\right\|,\\
			%	\sqrt{\|\bA\bar\bx^{(R_K)}-\bb\|^2+\big\|[\vf(\bar\bx^{(R_K)})]_+\big\|^2},\sum_{i=1}^m|\bar\lambda^{(R_K)}_i f_i(\bar\bx^{(R_K)})|\\
			%	\end{array}
			%	\right\}\\
			\max\left\{\textbf{S}_{R_K},\textbf{F}_{R_K},\textbf{C}_{R_K}\right\}
			\leq& \textstyle\frac{3}{\beta K}+\sqrt{\frac{2\gamma\mathcal{C}_1}{K}}+\frac{4\sqrt{D+M_y^2+M_{\lambda}^2}}{\beta \sqrt{K}}+ \frac{8(D+M_y^2+M_\lambda^2)}{\beta \sqrt{K}},
		\end{align}
		\normalsize
		where $\{(\textbf{S}_k,\textbf{F}_k,\textbf{C}_k)\}_{k\ge 1}$ is defined in \eqref{eq:def-sfc}, $M_y = M_y(\gamma)$, $M_\lambda= M_\lambda(\gamma)$ defined in \eqref{eq:def-Mlam}, and 
		\small	
		\begin{eqnarray}
			\label{eq:constantC1}
			\mathcal{C}_1=2B_{f_0}+2G+\frac{\beta}{2}\|\bA\bar \bx^{(0)}-\bb\|^2+\frac{\beta}{2}\big\|[\vf(\bar \bx^{(0)})]_+\big\|^2+\frac{3}{\beta}\left(2D+M_y^2+M_\lambda^2\right).
		\end{eqnarray}
		\normalsize
	\end{theorem}
	
	\begin{proof}
		Notice that $\phi_k$ is convex when $\gamma_k > \rho_0$. Hence, \eqref{eq:nesterovconv2-vary-beta-nc-constant} holds.
		%	letting $\bx=\bar\bx^{(k)}$ in \eqref{eq:nesterovconv1} and using the definition of $\phi_k$ in \eqref{eq:compositepart}, we have
		%	\small
		%	\begin{align*}
		%	&~f_0(\bar\bx^{(k+1)})+g(\bar\bx^{(k+1)})+\frac{\gamma_k}{2}\|\bar \bx^{(k+1)}-\bar\bx^{(k)}\|^2+\frac{\beta_k}{2}\left(\|\bA\bar \bx^{(k+1)}-\bb\|^2+\big\|[\vf(\bar \bx^{(k+1)})]_+\big\|^2\right)\\
		%	\le &~f_0(\bar\bx^{(k)})+g(\bar\bx^{(k)})+\frac{\beta_k}{2}\left(\|\bA\bar \bx^{(k)}-\bb\|^2+\big\|[\vf(\bar \bx^{(k)})]_+\big\|^2\right) +\hat{\varepsilon}_kD.
		%	\end{align*}
		%	\normalsize
		%	Summing up the above inequality over $k=0,1,\dots,K-1$ and dropping the non-negative term $\frac{\beta_{K-1}}{2}\left(\|\bA\bar \bx^{(K)}-\bb\|^2+\big\|[\vf(\bar \bx^{(K)})]_+\big\|^2\right)$ give
		%	\small
		%	\begin{align}
		%	\nonumber
		%	&~f_0(\bar\bx^{(K)})+g(\bar\bx^{(K)})+\sum_{k=0}^{K-1}\frac{\gamma_k}{2}\|\bar \bx^{(k+1)}-\bar\bx^{(k)}\|^2\\\nonumber
		%	\le &~f_0(\bar\bx^{(0)})+g(\bar\bx^{(0)})+\frac{\beta_0}{2}\left(\|\bA\bar \bx^{(0)}-\bb\|^2+\big\|[\vf(\bar \bx^{(0)})]_+\big\|^2\right)\\\label{eq:nesterovconv2-vary-beta}
		%	&~ +\frac{1}{2}\sum_{k=1}^{K-1}(\beta_k-\beta_{k-1})\left(\|\bA\bar \bx^{(k)}-\bb\|^2+\big\|[\vf(\bar \bx^{(k)})]_+\big\|^2\right)+ \left(\sum_{k=0}^{K-1}\hat{\varepsilon}_k\right)D.
		%	\end{align}
		%	\normalsize
		
		Since $\gamma_k=\gamma$ for all $k$, we have from Lemma~\ref{thm:boundlambda} that $\|\blambda^{(k)}\|\le M_\lambda$ and $\|\by^{(k)}\|\le M_y$ for all $k$. Hence, it follows from \eqref{eq:ippconverge-2} and \eqref{eq:set-beta-vary} that
		{\small
			\begin{equation}
				\label{eq:xtilde-x-slack-feasible-adapbeta-option3}
				\textstyle	\|\bA\bar\bx^{(k+1)}-\bb\|^2+\big\|[\vf(\bar\bx^{(k+1)})]_+\big\|^2\leq \frac{4\hat\varepsilon_k D}{\beta_k}+\frac{4(M_y^2+M_{\lambda}^2)}{\beta_k^2}
				\leq \frac{4(D+M_y^2+M_{\lambda}^2)}{\beta_k^2},
			\end{equation}
		}for any $k\ge 0$. Since $\beta_k = \beta\sqrt{k+1}$, we have from the above inequality that
		{\small
			\begin{equation}\label{eq:term-beta-case1}(\beta_k-\beta_{k-1})\left(\|\bA\bar \bx^{(k)}-\bb\|^2+\big\|[\vf(\bar \bx^{(k)})]_+\big\|^2\right)\le \frac{(\sqrt{k+1}-\sqrt{k})}{\beta k}\left(4D+4M_y^2+4M_\lambda^2\right).
			\end{equation}
		}Noting $\sqrt{k+1}-\sqrt{k}=\frac{1}{\sqrt{k}+\sqrt{k+1}}\leq \frac{1}{2\sqrt{k}}$ and $\sum_{k=1}^{K-1}k^{-\frac{3}{2}}\leq 1+\int_{1}^{K-1}x^{-\frac{3}{2}}dx\leq 3$, we sum up \eqref{eq:term-beta-case1} to have
		%	Because $\beta_k$ is increasing in $k$, we have %Assumption~\ref{assume:stochastic} 
		\small
		\begin{align}\label{eq:nesterovconv2-vary-sumbeta-0}
			&\sum_{k=1}^{K-1}(\beta_k-\beta_{k-1})\left(\|\bA\bar \bx^{(k)}-\bb\|^2+\big\|[\vf(\bar \bx^{(k)})]_+\big\|^2\right)
			%	\leq& \sum_{k=1}^{K-1}\frac{(\beta_k-\beta_{k-1})}{\beta_{k-1}^2}\left(4D+4\|\widehat\by^{(k-1)}\|^2+4\|\widehat\blambda^{(k-1)}\|^2\right)\\\nonumber
			%	\leq& \sum_{k=1}^{K-1}\frac{(\sqrt{k+1}-\sqrt{k})}{\beta k}\left(4D+4M_y^2+4M_\lambda^2\right)\\\nonumber
			%	\leq&\sum_{k=1}^{K-1}\frac{1}{2\beta k^{\frac{3}{2}}}\left(4D+4M_y^2+4M_\lambda^2\right)\\
			\leq \frac{6}{\beta}\left(D+M_y^2+M_\lambda^2\right).
		\end{align} 
		\normalsize
		%where the first inequality is by \eqref{eq:ippconverge-2} and the fact that $\hat{\varepsilon}_{k-1} D=\frac{D}{\beta_{k-1} k }\leq\frac{D}{\beta_{k-1}} $, the second by the definition of $\beta_k$ and Lemma~\ref{thm:boundlambda}, the third by the fact that $\sqrt{k+1}-\sqrt{k}=\frac{1}{\sqrt{k}+\sqrt{k+1}}\leq \frac{1}{2\sqrt{k}}$, and the last by the fact that $\sum_{k=1}^{K-1}k^{-\frac{3}{2}}\leq 1+\int_{1}^{K-1}x^{-\frac{3}{2}}dx\leq 3$.
		In addition, by the setting of $\hat{\varepsilon}_k$ in \eqref{eq:set-beta-vary}, it holds
		\small
		\begin{align}
			\label{eq:nesterovconv2-vary-sumepsilonhat}
			\textstyle	\sum_{k=0}^{K-1}\hat{\varepsilon}_k=\frac{1}{\beta }\sum_{k=0}^{K-1}(k+1)^{-\frac{3}{2}} \leq \frac{1}{\beta}\left(1+\int_{1}^{K}x^{-\frac{3}{2}}dx\right)\leq\frac{3}{\beta }.
		\end{align}
		\normalsize
		
		Now plugging \eqref{eq:nesterovconv2-vary-sumbeta-0} and \eqref{eq:nesterovconv2-vary-sumepsilonhat} into \eqref{eq:nesterovconv2-vary-beta-nc-constant} with $\vc(\bx)=\bA\bar \bx-\bb$,  	
		%Applying \eqref{eq:nesterovconv2-vary-sumbeta} and \eqref{eq:nesterovconv2-vary-sumepsilonhat} to \eqref{eq:nesterovconv2-vary-beta},   
		we obtain
		\small
		\begin{align}\label{eq:nesterovconv2-vary-beta-new}
			%\nonumber
			\textstyle \sum_{k=0}^{K-1}\frac{\gamma_k}{2}\|\bar \bx^{(k+1)}-\bar\bx^{(k)}\|^2\le \mathcal{C}_1,
			%\le &~f_0(\bar\bx^{(0)})+g(\bar\bx^{(0)})-f_0(\bar\bx^{(K)})-g(\bar\bx^{(K)})\\\nonumber
			%&~+\frac{\beta}{2}\|\bA\bar \bx^{(0)}-\bb\|^2+\frac{\beta}{2}\big\|[\vf(\bar \bx^{(0)})]_+\big\|^2+\frac{3}{\beta}\left(2D+M_y^2+M_\lambda^2\right)\\
			%\le&~2B_{f_0}+2G+\frac{\beta}{2}\|\bA\bar \bx^{(0)}-\bb\|^2+\frac{\beta}{2}\big\|[\vf(\bar \bx^{(0)})]_+\big\|^2+\frac{3}{\beta}\left(2D+M_y^2+M_\lambda^2\right)=\mathcal{C}_1,
		\end{align}
		\normalsize
		where we have used the definition of $\mathcal{C}_1$ in \eqref{eq:constantC1}.
		%where the second inequality is because $f_0(\bar\bx^{(0)})+g(\bar\bx^{(0)})-f_0(\bar\bx^{(K)})-g(\bar\bx^{(K)})\leq 2B_{f_0}+2G$ by Assumption~\ref{assume:stochastic}C and \eqref{eq:bd-cj-fi}. 
		From \eqref{eq:nesterovconv2-vary-beta-new}, the Cauchy-Schwarz inequality, and the setting $\gamma_k=\gamma, \forall k$, it follows
		\small
		\begin{eqnarray}
			\textstyle	\frac{1}{K}\sum_{k=0}^{K-1}\gamma_k\|\bar \bx^{(k+1)}-\bar\bx^{(k)}\|
			&\leq& \textstyle\frac{1}{K}\sqrt{\sum_{k=0}^{K-1}\gamma_k}\sqrt{\sum_{k=0}^{K-1}\gamma_k\|\bar\bx^{(k+1)}-\bar\bx^{(k)}\|^2}\label{eq:ineq5-varybeta-new-l1}\\
			&\leq&\textstyle\sqrt{\frac{2\gamma\mathcal{C}_1}{K}}.
			\label{eq:ineq5-varybeta-new}
		\end{eqnarray}
		\normalsize	
		Applying \eqref{eq:nesterovconv2-vary-sumepsilonhat} and \eqref{eq:ineq5-varybeta-new} to \eqref{eq:xtilde-x-dual-feas-option3} gives 
		\small
		\begin{align}
			\label{eq:xtilde-x-dual-feas-option3-avg}
			\textstyle\frac{1}{K}\sum_{k=0}^{K-1}\textbf{S}_{k+1}
			\leq\frac{3}{\beta K}+\sqrt{\frac{2\gamma\mathcal{C}_1}{K}}.
		\end{align}
		\normalsize
		
		%\noindent\textbf{Bounding }$\mathbf{T_2}$ and $\mathbf{T_3}$:
		
		%Lemma~\ref{thm:boundlambda} has indicated that $\|\widehat\by^{(k)}\|\le M_y$ and  $\|\widehat\blambda^{(k)}\|\le M_\lambda$ which, together with \eqref{eq:ippconverge-2} from Lemma~\ref{thm:ippconverge}, imply
		
		Futhermore, by the definition of $\textbf{F}_{k+1}$ in \eqref{eq:def-F}, we have from \eqref{eq:xtilde-x-slack-feasible-adapbeta-option3} that	
		\small
		\begin{align}
			\label{eq:xtilde-x-stationary-adapbeta-option3-avg}
			\textstyle \frac{1}{K}\sum_{k=0}^{K-1}\textbf{F}_{k+1}\leq \frac{1}{K}\sum_{k=0}^{K-1}\frac{2\sqrt{D+M_y^2+M_{\lambda}^2}}{\beta_k}\leq\frac{4\sqrt{D+M_y^2+M_{\lambda}^2}}{\beta \sqrt{K}},
		\end{align}
		\normalsize
		where we have used the following arguments:
		\small
		\begin{equation}\label{eq:sum-1-betak}
			\textstyle	\frac{1}{ K}\sum_{k=0}^{K-1}\frac{1}{\beta_k}=\frac{1}{\beta K}\sum_{k=0}^{K-1}\frac{1}{\sqrt{k+1}}\leq \frac{1}{\beta K}\int_{0}^{K}x^{-\frac{1}{2}}dx\leq \frac{2}{\beta \sqrt{K}}.
		\end{equation}
		\normalsize
		
		Finally, by $\bar\lambda^{(k+1)}_i=\beta_k [f_i(\bar\bx^{(k+1)})]_+$ and also using \eqref{eq:xtilde-x-slack-feasible-adapbeta-option3}, we have
		%	where the second inequality is because of \eqref{eq:set-beta-vary}. Inequality~\eqref{eq:xtilde-x-slack-feasible-adapbeta-option3} further implies
		%	%\small
		%	\begin{eqnarray}
		%	\label{eq:xtilde-x-stationary-adapbeta-option3-k}
		%	\sqrt{\|\bA\bar\bx^{(k+1)}-\bb\|^2+\big\|[\vf(\bar\bx^{(k+1)})]_+\big\|^2}\leq \frac{2\sqrt{D+M_y^2+M_{\lambda}^2}}{\beta_k}
		%	%\leq \frac{2\sqrt{D+M_y^2+M_{\lambda}^2}}{\beta\sqrt{k+1}}
		%	\end{eqnarray}
		%	and
		\small
		\begin{equation}
			\label{eq:xtilde-x-slack-adapbeta-option3-k}
			\textstyle	\sum_{i=1}^m|\bar\lambda^{(k+1)}_i f_i(\bar\bx^{(k+1)})|=\beta_k\big\|[\vf(\bar\bx^{(k+1)})]_+\big\|^2
			\leq\frac{ 4(D+M_y^2+M_\lambda^2)}{\beta_k},
			%\leq\frac{ 4(D+M_y^2+M_\lambda^2)}{\beta\sqrt{k+1}},
		\end{equation}
		\normalsize
		%	where the first equality in \eqref{eq:xtilde-x-slack-adapbeta-option3-k} holds because of the definition of $\bar\lambda^{(k+1)}_i$. 
		%	Averaging both sides of \eqref{eq:xtilde-x-stationary-adapbeta-option3-k} and \eqref{eq:xtilde-x-slack-adapbeta-option3-k} for $k=0,1,\dots,K-1$, we have 
		%	%\small
		%	
		%	%\normalsize
		%	and
		Hence, from the definition of $\textbf{C}_{k+1}$ in \eqref{eq:def-C}, we average the above inequality to have
		\begin{equation}
			\label{eq:xtilde-x-slack-adapbeta-option3-avg}
			\textstyle	\frac{1}{K}\sum_{k=0}^{K-1}\textbf{C}_{k+1}
			\leq\frac{1}{K}\sum_{k=0}^{K-1}\frac{ 4(D+M_y^2+M_\lambda^2)}{\beta_k}\leq \frac{8(D+M_y^2+M_\lambda^2)}{\beta \sqrt{K}},
		\end{equation}
		where we have used \eqref{eq:sum-1-betak} again. %the last inequalities in both \eqref{eq:xtilde-x-stationary-adapbeta-option3-avg} and \eqref{eq:xtilde-x-slack-adapbeta-option3-avg} are because of the definition of $\beta_k$ and the fact that . 
		
		Now the result in \eqref{eq:convergeinK1} follows by plugging \eqref{eq:xtilde-x-dual-feas-option3-avg}, \eqref{eq:xtilde-x-stationary-adapbeta-option3-avg}, and \eqref{eq:xtilde-x-slack-adapbeta-option3-avg} into \eqref{eq:threeterms}. 
	\end{proof}
	
	\begin{remark}
		In the parameter setting \eqref{eq:set-beta-vary}, we require the knowledge of the weak-convexity constant $\rho_0$. In case it is unknown but the smoothness constant $L_{f_0}$ is known, we can set $\gamma > L_{f_0}$. Without knowledge of $\rho_0$ or $L_{f_0}$, we cannot guarantee strong convexity of the function $\phi_k$ given in \eqref{eq:compositepart}. To the best of our knowledge, smoothness constants are assumed in all existing works on the complexity analysis of first-order methods for non-convex problems, e.g., \cite{DBLP:journals/mp/GhadimiL16, kong2019complexity, cartis2011evaluation}.
	\end{remark}

	According to Theorem~\ref{thm:complexity-varying-beta-new}, the convergence rate of Algorithm~\ref{alg:iPPP} is $O(\frac{1}{\sqrt{K}})$, in terms of the number of outer iterations. Together with the complexity result of the AdapAPG subroutine, we below give the overall computational complexity of Algorithm~\ref{alg:iPPP} for finding an $\varepsilon$-stationary point of \eqref{eq:gco-cvx}. %and its counterpart in the sense of expectation.
	\begin{corollary}[complexity result]
		Under the assumptions of Theorem~\ref{thm:complexity-varying-beta-new},
		let 
		{\small	$$K=\left\lceil \max\left\{\frac{6}{\beta\varepsilon},\, \frac{4}{\varepsilon^2}\left[\sqrt{2\gamma\mathcal{C}_1}+\frac{4\sqrt{D+M_y^2+M_{\lambda}^2}}{\beta }+ \frac{8(D+M_y^2+M_\lambda^2)}{\beta }\right]^2\right\}\right\rceil = O({1}/{\varepsilon^2}),$$ }where $\mathcal{C}_1$ is defined as in \eqref{eq:constantC1}. Then
		%	the following statements hold:
		%\begin{enumerate}
		%\item 
		$\bar\bx^{(R_K)}$ is an $\varepsilon$-stationary point of \eqref{eq:gco-cvx}, and 
		%\item 
		the total complexity to produce $\bar\bx^{(R_K)}$ is $\tilde O\left(1/\varepsilon^{\frac{5}{2}}\right)$.
		%\end{enumerate}
	\end{corollary}
	\begin{proof}
		%When $K\geq\max\big\{K_1,K_2\big\}$, 
		With the given $K$, the right hand side of  \eqref{eq:convergeinK1} is upper bounded by $\varepsilon$. Hence,  $\bar\bx^{(R_K)}$ is an $\varepsilon$-stationary point of \eqref{eq:gco-cvx}.  
		
		Let $T_k$ be the number of proximal gradient steps performed by Algorithm~\ref{alg:adap-APG1} in the $k$-th call by Algorithm~\ref{alg:iPPP}.
		Then according to Theorem~\ref{thm:conv-adap-APG1}, the settings of $\gamma_k$ and $\beta_k$ in \eqref{eq:set-beta-vary}, and the formula of $L_{\phi_k}$ in \eqref{eq:Lk}, we have 
		$
		\textstyle	T_k=\tilde O\left(\sqrt{\frac{L_{\phi_k}}{\gamma_k-\rho_0}}\right) =\tilde O\left(\sqrt{\frac{\beta_k}{\gamma-\rho_0}}\right)=\tilde O\left(\frac{(k+1)^{\frac{1}{4}}}{\sqrt{\gamma-\rho_0}}\right), 
		$
		for $k=0,1,\dots,K-1$.
		Therefore, the total complexity is
		$\textstyle T_{\mathrm{total}} = \sum_{k=0}^{K-1} T_k = \sum_{k=0}^{K-1} \tilde O\left(\frac{(k+1)^{\frac{1}{4}}}{\sqrt{\gamma-\rho_0}}\right)=\tilde O\left(\frac{K^{\frac{5}{4}}}{\sqrt{\gamma-\rho_0}}\right) =  \tilde O\left(1/\varepsilon^{\frac{5}{2}}\right),$
		which completes the proof.
		%	Plugging $K=\max\big\{\lceil K_1 \rceil, \ \lceil K_2\rceil \big\}=O(\frac{1}{\varepsilon^2})$ into the above equation, we obtain a total complexity $\tilde O\left(1/\varepsilon^{\frac{5}{2}}\right)$ for Algorithm~\ref{alg:iPPP}.
	\end{proof}

	\section{Complexity of the iPPP method with non-convex constraints}
	\label{sec:penaltymethodnonconvex}
	In this section, we consider the problem in \eqref{eq:gco} with a non-convex objective and non-convex constraints. Instead of Assumption~\ref{assume:convexconstraint}, %is not assumed anymore but new assumptions will be made.  One of the new assumptions we need is the following. 
	we make the following assumption.
	\begin{assumption}
		\label{assume:nonconvexconstraintnew}
		$f_i$ is $\rho_i$-weakly convex for $\rho_i \ge 0$ for $i=0,1,\dots,m$. $c_j$ is $\sigma_j$-weakly convex for $\sigma_j\ge0$ for $j=1,\ldots,n$.
	\end{assumption}
	
	The non-convexity of the constraints further increases the difficulty of finding a stationary point of \eqref{eq:gco}. Fortunately, with a sufficiently large $\gamma_k$, the proximal-point penalty subproblem \eqref{eq:iALMsub} is strongly convex under Assumption~\ref{assume:nonconvexconstraintnew} and thus can be effectively solved by Algorithm~\ref{alg:adap-APG1}. By this observation, we show that Algorithm~\ref{alg:iPPP} can still guarantee an approximate stationary solution of \eqref{eq:gco} within a polynomial time. %under additional assumptions. 

	\subsection{Technical lemmas}\label{sec:bd-dual-nc}
	
	To show the complexity result, we first establish a few technical lemmas. A proof of the following lemma has been given in \cite[Lemma 2]{Drusvyatskiy2018}. We present it here for the readers' convenience.
	\begin{lemma}
		\label{thm:weakconvexity}
		Suppose Assumptions~\ref{assume:stochastic} and \ref{assume:nonconvexconstraintnew} hold. For any $\beta>0$, the function $\frac{\beta}{2}[f_i(\bx)]_+^2$ is $(\beta\rho_iB_{f_i})$-weakly convex for $i=1,\dots,m$, and $\frac{\beta}{2}[c_j(\bx)]^2$ is $(\beta\sigma_j B_{c_j})$-weakly convex for $j=1,\ldots,n$.
	\end{lemma}
	\begin{proof}	
		Since $f_i(\bx)$ is $\rho_i$-weakly convex, we have 
		\small
		$$f_i(\bx)- f_i(\bx')\geq\left\langle \nabla f_i(\bx'),\bx-\bx' \right\rangle-\frac{\rho_i}{2}\|\bx'-\bx\|^2, \, \forall \, \bx, \bx' \in \X.$$ 
		\normalsize
		%for any $\bx$ and $\bx'$ in $\X$. 
		Using this inequality, the fact $|f_i(\bx)|\leq B_{f_i}$, and also the convexity of $[t]_+^2$ about $t$, we have 
		\small
		\begin{eqnarray*}
			\textstyle		\frac{\beta}{2}[f_i(\bx)]_+^2&\geq& \textstyle\frac{\beta}{2}[f_i(\bx')]_+^2+\beta [f_i(\bx')]_+ \big(f_i(\bx)-f_i(\bx')\big)\\
			&\geq& \textstyle \frac{\beta}{2}[f_i(\bx')]_+^2+\beta [f_i(\bx')]_+\big\langle \nabla f_i(\bx'),\bx-\bx' \big\rangle-\frac{\beta\rho_iB_{f_i}}{2}\|\bx'-\bx\|^2,
		\end{eqnarray*}
		\normalsize
		which implies the $(\beta\rho_iB_{f_i})$-weak convexity of $\frac{\beta}{2}[f_i(\bx)]_+^2$. Similarly, we can show the $(\beta\sigma_j B_{c_j})$-weak convexity of $\frac{\beta}{2}[c_j(\bx)]^2$ for each $j$ and thus complete the proof. 
	\end{proof}
	
	%Since  $g(\bx)+\frac{\rho}{2}\|\bx-\bar\bx\|^2+\frac{\beta}{2}\|\bA\bx-\bb\|^2$ is convex and $f_0(\bx)$ is $\rho_0$-weakly convex, Lemma~\ref{thm:weakconvexity} implies that $\mathcal{L}_{\beta}(\bx;\bar\bx)$ is $(\rho-\rho_0-\sum_{i=1}^m\beta \rho_i B_{f_i})$-strongly convex when $\rho>\rho_0+\sum_{i=1}^m\beta \rho_i B_{f_i}$. Utilizing this property, we can show that, when $\rho$ and $\beta_k$ are appropriately chosen, Algorithm~\ref{alg:iPPP} can still guarantee a nearly stationary point defined slightly different from Definition~\ref{def:stationary}. We define this type of nearly stationary point as follows.
	
	With a little abuse of notation, under Assumption~\ref{assume:nonconvexconstraintnew}, $\phi_k$ defined in \eqref{eq:compositepart}   
	is $L_{\phi_k}$-smooth with
	\small
	\begin{eqnarray}
		\label{eq:Lk-case2}
		\textstyle L_{\phi_k} = L_{f_0}+\gamma_k+\beta_k\left(\sum_{i=1}^m B_{f_i}(B_{f_i}+L_{f_i})+\sum_{j=1}^n B_{c_j}(B_{c_j}+L_{c_j})\right).
	\end{eqnarray}
	\normalsize
	Note that the value of $L_{\phi_k}$ is different from that defined in \eqref{eq:Lk}. 
	In addition, by Assumption~\ref{assume:nonconvexconstraintnew} and Lemma~\ref{thm:weakconvexity}, the function $f_0(\bx)+\frac{\beta_k}{2}\left(\|\vc(\bx)\|^2+\big\|[\vf(\bx)]_+\big\|^2\right)$
	is $\Gamma_k$-weakly convex with 
	\small
	\begin{equation}\label{eq:def-gamma-k}
		\textstyle \Gamma_k:=\rho_0 + \beta_k \rho_c,\quad \rho_c = \left(\sum_{i=1}^m \rho_iB_{f_i}+\sum_{j=1}^n \sigma_jB_{c_j}\right).
	\end{equation}
	\normalsize
	%Hence, $\phi_k$ is $(\gamma_k-\Gamma_k)$-strongly convex with 
	%if $\gamma_k > \Gamma_k$. With these observations, we establish an inequality similar to that in \eqref{eq:nesterovconv1}. The proof is omitted because it is almost identical to that of \eqref{eq:nesterovconv1}.
	%
	%\begin{lemma}\label{thm:ippconverge-nc}
	%	Suppose Assumption~\ref{assume:stochastic} and \ref{assume:nonconvexconstraintnew} hold. Given 
	%	$\gamma_k>\Gamma_k$ with $\Gamma_k$ defined in \eqref{eq:def-gamma-k} and $\beta_k>0$ for $k\ge0$, let $\phi_k$ be defined in \eqref{eq:compositepart} and $\bar\bx^{(k+1)}$ be generated as in Algorithm~\ref{alg:iPPP}. For any $k\geq0$, \eqref{eq:nesterovconv1} holds.
	%\end{lemma}

	\subsection{The complexity of the iPPP method under a non-singularity condition}
	In this subsection, we assume the follows in addition to Assumptions~\ref{assume:stochastic} and \ref{assume:nonconvexconstraintnew}.  
	\begin{assumption}
		\label{assume:nonconvexconstraintsingular}
		There exists a constant $\nu>0$ such that for any $\bx\in\X$, the following inequality holds
		{\small
			\begin{eqnarray}
				\label{eq:mfcq}
				\nu\sqrt{\|[\vf(\bx)]_+\|^2+\|\vc(\bx)\|^2}\leq\mathrm{dist}\left(J_\vc(\bx)^\top \vc(\bx)+J_\vf(\bx)^\top[\vf(\bx)]_+,\,-\mathcal{N}_{\mathcal{X}}(\bx)\right).
			\end{eqnarray}
		}
	\end{assumption}
	This assumption is inspired by a similar assumption made in \cite{NIPS2019_9545,xie2019complexity}, where only equality constraints are considered. It is related to the linear independence constraint qualification (LICQ). Suppose $\bx$ is a feasible solution to \eqref{eq:gco} and satisfies  $\mathcal{N}_{\mathcal{X}}(\bx)=\{\mathbf{0}\}$, i.e., $\bx\in\mathrm{int}(\X)$. Let $I(\bx)=\{1\leq i\leq m|f_i(\bx)\geq 0\}$ be the index set of active inequality constraints at $\bx$. We say that \eqref{eq:gco} satisfies LICQ at $\bx$ if $\{\nabla f_i(\bx)\}_{i\in I(\bx)}\cup\{\nabla c_j(\bx)\}_{j=1}^n$ are linearly independent. Now we extend this condition to any $\bx\in\mathrm{int}(\X)$, which are not necessarily feasible to \eqref{eq:gco}. As a result, there exists $\nu_\bx>0$ such that $\nu_\bx \sqrt{\|[\vf(\bx)]_+\|^2+\|\vc(\bx)\|^2} \le \|J_\vf(\bx)^\top[\vf(\bx)]_++J_\vc(\bx)^\top\vc(\bx)\|$, i.e., \eqref{eq:mfcq} holds with $\nu$ depending on $\bx\in\mathrm{int}(\X)$. If $\nu_\bx$ can be lower bounded by some $\nu>0$ for any $\bx\in\mathrm{int}(\X)$, \eqref{eq:mfcq} will hold on $\mathrm{int}(\X)$. We make Assumption \ref{assume:nonconvexconstraintsingular} as a \emph{global} condition because we aim at establishing a \emph{global} (rather than \emph{local}) complexity result. In Appendix, we show that Assumption \ref{assume:nonconvexconstraintsingular} can hold for the application we test in the numerical experiment. 

	Under  Assumptions~\ref{assume:stochastic}, \ref{assume:nonconvexconstraintnew}, and \ref{assume:nonconvexconstraintsingular}, we are able to show that our iPPP method can find an $\varepsilon$-stationary point of  \eqref{eq:gco} in a complexity of $\tilde O(\frac{1}{\varepsilon^{3}})$. Similar to Theorem~\ref{thm:complexity-varying-beta-new}, we first show a convergence rate result, in terms of the number of outer iterations. %We first present a generic convergence property of Algorithm~\ref{alg:iPPP} when the output is generated using Option I.
	
	\begin{theorem}\label{thm:complexity-varying-beta-nc}
		Suppose that Assumptions~\ref{assume:stochastic}, \ref{assume:nonconvexconstraintnew} and \ref{assume:nonconvexconstraintsingular} hold and the parameters $\{\gamma_k\}$, $\{\beta_k\}$ and  $\{\hat{\varepsilon}_k\}$ in Algorithm~\ref{alg:iPPP} are taken as 
		\begin{equation}\label{eq:set-beta-vary-nc}
			\beta_k=\beta(k+1)^{\frac{1}{3}},\quad \gamma_k=2\Gamma_k, \quad\text{ and } \quad \textstyle \hat{\varepsilon}_k= \frac{1}{\beta (k+1)^{\frac{4}{3}} },
		\end{equation}  
		where $\beta>0$ is a constant, and $\Gamma_k$ is defined in \eqref{eq:def-gamma-k}. 
		If $\{R_k\}$ is generated by Option I, then  for any $K\ge1$, it holds
		\small
		\begin{align}\label{eq:convergeinK1-nc-02}
			\max\left\{\textbf{S}_{R_K},\textbf{F}_{R_K},\textbf{C}_{R_K}\right\}
			&\le & \textstyle \frac{1}{K}\left(\frac{4}{\beta }+\frac{4}{\nu \beta^2 }+\frac{9}{2\nu^2\beta^3 }+\frac{6\mathcal{C}_2(\rho_0/\beta+\rho_c)}{\nu^2}\right) + \frac{\sqrt{2\rho_0\mathcal{C}_2}}{\sqrt K}\left(1+\frac{1}{\nu\beta}\right)\\\nonumber
			&& \textstyle +\frac{1}{K^{1/3}}\left(\left(1+\frac{1}{\nu\beta}\right)\sqrt{2\beta \rho_c\mathcal{C}_2}+\frac{3B_{f_0}+3M}{2\nu \beta}+\frac{9(B_{f_0}+M)^2}{2\nu^2 \beta }\right)
		\end{align}
		\normalsize
		where $\{(\textbf{S}_k,\textbf{F}_k,\textbf{C}_k)\}_{k\ge 1}$ is defined in \eqref{eq:def-sfc}, $\rho_c$ is defined in \eqref{eq:def-gamma-k}, and~\footnote{Note that $\mathcal{C}_2$ is a constant independent of $K$.} 
		\small
		\begin{align}
			\label{eq:defconstantC-nc}
			\mathcal{C}_2:=&4\left[\textstyle 2B_{f_0}+2G+\frac{\beta}{2}\|\vc(\bar \bx^{(0)})\|^2+\frac{\beta}{2}\big\|[\vf(\bar \bx^{(0)})]_+\big\|^2+\frac{8}{3\nu^2\beta}(1/\beta+B_{f_0}+M)^2+\frac{4D}{\beta }\right] 	
			\\& \textstyle +\sum_{k=0}^{K'-1}\frac{8\gamma_k^2 D^2}{3\nu^2\beta (k+1)^{\frac{4}{3}}} \nonumber
		\end{align}
		\normalsize
		with
		\small
		\begin{align}
			\label{eq:defrangeofk}
			\textstyle	K':=\left\lceil\max\left\{\left(\frac{32\rho_0}{3\nu^2\beta}\right)^{\frac{3}{4}},\frac{32\rho_c}{3\nu^2}\right\}\right\rceil-1.
		\end{align}
		\normalsize
	\end{theorem}
	
	\begin{proof}
		%Suppose $R_k$ is generated by Option I. By the definitions of $\textbf{S}_{k+1}$, $\textbf{F}_{k+1}$, and $\textbf{C}_{k+1}$ in Algorithm~\ref{alg:iPPP}, inequality \eqref{eq:threeterms} still holds here. Next, 
		Similar to the proof of Theorem~\ref{thm:complexity-varying-beta-new}, we go to bound the three summations on the right hand side of \eqref{eq:threeterms}.
		%	
		%	Following the same argument used to prove \eqref{eq:xtilde-x-dual-feas-option3} in the proof of Theorem~\ref{thm:complexity-varying-beta-new}, we can show that \eqref{eq:xtilde-x-dual-feas-option3} also holds here.  
		%	
		%	Letting $\bx=\bar\bx^{(k)}$ in \eqref{eq:nesterovconv1} and using the definition of $\phi_k$ in \eqref{eq:compositepart}, we have
		%	\small
		%	\begin{align*}
		%	&~f_0(\bar\bx^{(k+1)})+g(\bar\bx^{(k+1)})+\frac{\gamma_k}{2}\|\bar \bx^{(k+1)}-\bar\bx^{(k)}\|^2+\frac{\beta_k}{2}\left(\|\vc(\bar \bx^{(k+1)})\|^2+\big\|[\vf(\bar \bx^{(k+1)})]_+\big\|^2\right)\\
		%	\le &~f_0(\bar\bx^{(k)})+g(\bar\bx^{(k)})+\frac{\beta_k}{2}\left(\|\vc(\bar \bx^{(k)})\|^2+\big\|[\vf(\bar \bx^{(k)})]_+\big\|^2\right) +\hat{\varepsilon}_kD.
		%	\end{align*}
		%	\normalsize
		%	Summing up the above inequality over $k=0,1,\dots,K-1$ and dropping the non-negative term $\frac{\beta_{K-1}}{2}\left(\|\vc(\bar \bx^{(K)})\|^2+\big\|[\vf(\bar \bx^{(K)})]_+\big\|^2\right)$  give
		%	\small
		%	\begin{align}
		%	\nonumber
		%	&~f_0(\bar\bx^{(K)})+g(\bar\bx^{(K)})+\sum_{k=0}^{K-1}\frac{\gamma_k}{2}\|\bar \bx^{(k+1)}-\bar\bx^{(k)}\|^2\\\nonumber
		%	\le &~f_0(\bar\bx^{(0)})+g(\bar\bx^{(0)})+\frac{\beta_0}{2}\left(\|\vc(\bx^{(0)})\|^2+\big\|[\vf(\bar \bx^{(0)})]_+\big\|^2\right)\\\label{eq:nesterovconv2-vary-beta-nc}
		%	&~+\frac{1}{2}\sum_{k=1}^{K-1}(\beta_k-\beta_{k-1})\left(\|\vc(\bx^{(k)})\|^2+\big\|[\vf(\bar \bx^{(k)})]_+\big\|^2\right)+\left(\sum_{k=0}^{K-1}\hat{\varepsilon}_k\right)D.
		%	\end{align}
		%	\normalsize
		
		By the stopping condition of Algorithm~\ref{alg:adap-APG1}, there must exist $\bar\bxi^{(k+1)}\in\partial g(\bar\bx^{(k+1)})$ such that $\|\nabla \phi_k(\bar\bx^{(k+1)})+\bar\bxi^{(k+1)}\|\le \hat\varepsilon_k$.
		%	\small
		%	$$
		%	\left\|
		%	\begin{array}{c}
		%	\nabla f_0(\bar\bx^{(k+1)})+\bar\bxi^{(k+1)} + \gamma (\bar\bx^{(k+1)}-\bar\bx^{(k)})\\
		%	+\beta_k J_\vc(\bar\bx^{(k+1)})^\top \vc(\bar\bx^{(k+1)})+\beta_k J_\vf(\bar\bx^{(k+1)})^\top [\vf(\bar\bx^{(k+1)})]_+
		%	\end{array}
		%	\right\|\leq\hat\varepsilon_k.
		%	$$
		%	\normalsize
		From Assumption~\ref{assume:stochastic}C, we have $\bar\bxi^{(k+1)}=\bar\bxi_1+\bar\bxi_2$ where $\bar\bxi_1\in\mathcal{N}_{\mathcal{X}}(\bar\bx^{(k+1)})$ and $\|\bar\bxi_2\|\leq M$. Hence, it follows from $\|\nabla \phi_k(\bar\bx^{(k+1)})+\bar\bxi^{(k+1)}\|\le \hat\varepsilon_k$ and \eqref{eq:grad-phik} that 
		\small
		\begin{align}
			\nonumber
			&~\left\|\bar\bxi_1+\beta_k J_\vc(\bar\bx^{(k+1)})^\top \vc(\bar\bx^{(k+1)})+\beta_k J_\vf(\bar\bx^{(k+1)})^\top [\vf(\bar\bx^{(k+1)})]_+\right\|\\\nonumber
			&~\leq\hat\varepsilon_k+\gamma_k \|\bar\bx^{(k+1)}-\bar\bx^{(k)}\|+\|\nabla f_0(\bar\bx^{(k+1)})\|+\|\bar\bxi_2\|\\\label{eq:scaledistancetonormalcone}
			&~\leq\hat\varepsilon_k+\gamma_k \|\bar\bx^{(k+1)}-\bar\bx^{(k)}\|+B_{f_0}+M,
		\end{align}
		\normalsize
		where we have used \eqref{eq:bd-fi} in the last inequality. %and bounded $\|\bar\bx^{(k+1)}-\bar\bx^{(k)}\|$ by $D$ according to Assumption~\ref{assume:stochastic}B. By , we have 
		Now noting $\frac{\bar\bxi_1}{\beta_k}\in\mathcal{N}_{\mathcal{X}}(\bar\bx^{(k+1)})$, we have from \eqref{eq:scaledistancetonormalcone} and Assumption~\ref{assume:nonconvexconstraintnew} that
		\small
		\begin{align}\label{eq:boundinfeasibilitybycondition-nc}
			%\nonumber
			&\nu\sqrt{\|\vc(\bar\bx^{(k+1)})\|^2+\big\|[\vf(\bar \bx^{(k+1)})]_+\big\|^2}
			%\leq&
			%\mathrm{dist}\left(\sum_{j=1}^nc_j(\bar\bx^{(k+1)})\nabla c_j(\bar\bx^{(k+1)})+\sum_{i=1}^m[f_i(\bar\bx^{(k+1)})]_+\nabla f_i(\bar\bx^{(k+1)}),-\mathcal{N}_{\X}(\bar\bx^{(k+1)})\right)\\\nonumber
			%\leq&\left\|\frac{\bar\bxi_1}{\beta_k}+\sum_{j=1}^nc_j(\bar\bx^{(k+1)})\nabla c_j(\bar\bx^{(k+1)})+\sum_{i=1}^m[f_i(\bar\bx^{(k+1)})]_+\nabla f_i(\bar\bx^{(k+1)})\right\|\\
			\leq \frac{\hat\varepsilon_k+\gamma_k \|\bar\bx^{(k+1)}-\bar\bx^{(k)}\|+B_{f_0}+M}{\beta_k}, \, \forall\, k\ge0.
		\end{align}
		\normalsize
		%where the second inequality is because $\frac{\bar\bxi_1}{\beta_k}\in\mathcal{N}_{\mathcal{X}}(\bar\bx^{(k+1)})$ and the last inequality is from \eqref{eq:scaledistancetonormalcone}.
		Since $\hat\varepsilon_k\le \frac{1}{\beta}$ for all $k$, \eqref{eq:boundinfeasibilitybycondition-nc} implies
		{\small\begin{align}\label{eq:thm3-ineq-temp}
				\|\vc(\bx^{(k)})\|^2+\big\|[\vf(\bar \bx^{(k)})]_+\big\|^2
				\le &~ \frac{1}{\nu^2\beta_{k-1}^2}(1/\beta+\gamma_{k-1} \|\bar\bx^{(k)}-\bar\bx^{(k-1)}\|+B_{f_0}+M)^2\cr
				\le & ~ \frac{2}{\nu^2\beta_{k-1}^2}\left[(1/\beta+B_{f_0}+M)^2+ \gamma_{k-1}^2\|\bar\bx^{(k)}-\bar\bx^{(k-1)}\|^2\right],
			\end{align}
		}for all $k\ge1$. Notice $(k+1)^{\frac{1}{3}}-k^{\frac{1}{3}}=\frac{1}{k^{\frac{2}{3}}+k^{\frac{1}{3}}(k+1)^{\frac{1}{3}}+(k+1)^{\frac{2}{3}}}\leq \frac{1}{3k^{\frac{2}{3}}}$. Hence, by the setting of $\{\beta_k\}$ in \eqref{eq:set-beta-vary-nc}, it holds $\frac{\beta_k-\beta_{k-1}}{\beta_{k-1}^2}\le \frac{1}{3\beta k^{\frac{4}{3}}}$. Therefore, multiplying $\beta_k-\beta_{k-1}$ to both sides of \eqref{eq:thm3-ineq-temp} and summing it over $k=1$ to $K-1$, we have
		%Because $\beta_k$ is increasing in $k$, we have %Assumption~\ref{assume:stochastic} 
		\small
		\begin{align}\label{eq:nesterovconv2-vary-sumbeta}
			\nonumber
			&\textstyle \sum_{k=1}^{K-1}(\beta_k-\beta_{k-1})\left(\|\vc(\bx^{(k)})\|^2+\big\|[\vf(\bar \bx^{(k)})]_+\big\|^2\right) \nonumber \\%\\\nonumber
			%\leq& \sum_{k=1}^{K-1}\frac{(\beta_k-\beta_{k-1})}{\nu^2\beta_{k-1}^2}(\hat\varepsilon_{k-1}+\gamma_{k-1} \|\bar\bx^{(k)}-\bar\bx^{(k-1)}\|+B_{f_0}+M)^2\\\nonumber
			%\leq& \sum_{k=1}^{K-1}\frac{(k+1)^{\frac{1}{3}}-k^{\frac{1}{3}}}{\nu^2\beta k^{\frac{2}{3}}}(1/\beta+\gamma_{k-1} \|\bar\bx^{(k)}-\bar\bx^{(k-1)}\|+B_{f_0}+M)^2\\\nonumber
			\leq&~\textstyle \sum_{k=1}^{K-1}\frac{2}{3\nu^2\beta k^{\frac{4}{3}}}\left[(1/\beta+B_{f_0}+M)^2+\gamma_{k-1}^2\|\bar\bx^{(k)}-\bar\bx^{(k-1)}\|^2\right]\\%(1/\beta+\gamma_{k-1} \|\bar\bx^{(k)}-\bar\bx^{(k-1)}\|+B_{f_0}+M)^2\\
			\leq&~\textstyle \frac{8}{3\nu^2\beta}(1/\beta+B_{f_0}+M)^2+\sum_{k=1}^{K-1}\frac{2\gamma_{k-1}^2 \|\bar\bx^{(k)}-\bar\bx^{(k-1)}\|^2}{3\nu^2\beta k^{\frac{4}{3}}},
		\end{align} 
		\normalsize
		where %the first inequality is by \eqref{eq:boundinfeasibilitybycondition-nc}, the second by the definitions of $\beta_k$ and $\hat\varepsilon_k$, the third by the fact that $(k+1)^{\frac{1}{3}}-k^{\frac{1}{3}}=\frac{1}{k^{\frac{2}{3}}+k^{\frac{1}{3}}(k+1)^{\frac{1}{3}}+(k+1)^{\frac{2}{3}}}\leq \frac{1}{3k^{\frac{2}{3}}}$, and 
		we have used $\sum_{k=1}^{K-1}k^{-\frac{4}{3}}\leq 1+\int_{1}^{K-1}x^{-\frac{4}{3}}dx\leq 4$ in the last inequality.
		In addition, it follows from \eqref{eq:set-beta-vary-nc} that
		\small
		\begin{align}
			\label{eq:nesterovconv2-vary-sumepsilonhat-nc}
			\textstyle	\sum_{k=0}^{K-1}\hat{\varepsilon}_k=\frac{1}{\beta }\sum_{k=0}^{K-1}(k+1)^{-\frac{4}{3}} \leq \frac{1}{\beta}\left(1+\int_{1}^{K}x^{-\frac{4}{3}}dx\right)\leq\frac{4}{\beta }.
		\end{align}
		\normalsize
		Since $\gamma_k > \Gamma_k$, $\phi_k$ defined in \eqref{eq:compositepart} is convex, and thus \eqref{eq:nesterovconv2-vary-beta-nc-constant} holds.	
		Adding \eqref{eq:nesterovconv2-vary-sumbeta} and \eqref{eq:nesterovconv2-vary-sumepsilonhat-nc} to \eqref{eq:nesterovconv2-vary-beta-nc-constant}, we obtain
		\small
		\begin{align}\label{eq:nesterovconv2-vary-beta-new-nc}
			%\nonumber
			&~\textstyle \sum_{k=0}^{K-2}\left(\frac{\gamma_k}{2}-\frac{2\gamma_k^2}{3\nu^2\beta (k+1)^{\frac{4}{3}}}\right)\|\bar \bx^{(k+1)}-\bar\bx^{(k)}\|^2+\frac{\gamma_{K-1}}{2}\|\bar \bx^{(K)}-\bar\bx^{(K-1)}\|^2\\\nonumber
			%\le &~f_0(\bar\bx^{(0)})+g(\bar\bx^{(0)})-f_0(\bar\bx^{(K)})-g(\bar\bx^{(K)})+\frac{\beta_0}{2}\left(\|\vc(\bx^{(0)})\|^2+\big\|[\vf(\bar \bx^{(0)})]_+\big\|^2\right)\\\nonumber
			%&~+\frac{8}{3\nu^2\beta}(1/\beta+B_{f_0}+M)^2+\frac{4D}{\beta }\\
			\le &~2B_{f_0}+2G+\frac{\beta_0}{2}\left(\|\vc(\bx^{(0)})\|^2+\big\|[\vf(\bar \bx^{(0)})]_+\big\|^2\right)+\frac{8}{3\nu^2\beta}(1/\beta+B_{f_0}+M)^2+\frac{4D}{\beta }.
		\end{align}
		\normalsize
		%	where the second inequality is because $f_0(\bar\bx^{(0)})+g(\bar\bx^{(0)})-f_0(\bar\bx^{(K)})-g(\bar\bx^{(K)})\leq 2B_{f_0}+2G$ according to 
		% Assumption~\ref{assume:stochastic}C and \eqref{eq:bd-cj-fi}.

		By $\gamma_k=2\Gamma_k$ with $\Gamma_k$ defined in \eqref{eq:def-gamma-k}, we have
		\small
		\begin{align*}
			%\label{eq:gammaksmall}
			\textstyle	\frac{2\gamma_k}{3\nu^2\beta (k+1)^{\frac{4}{3}}}=\frac{4\rho_0}{3\nu^2\beta (k+1)^{\frac{4}{3}}}+\frac{4\rho_c}{3\nu^2 (k+1)}.
		\end{align*}
		\normalsize
		When $k\geq K'$  with $K'$ defined in \eqref{eq:defrangeofk}, it holds that $\frac{4\rho_0}{3\nu^2\beta (k+1)^{\frac{4}{3}}}\leq\frac{1}{8}$ and $\frac{4\rho_c}{3\nu^2 (k+1)}\leq\frac{1}{8}$, %which happens for any 
		%	\begin{align}
		%	\label{eq:rangeofk}
		%	k\geq K':=\left\lceil\max\left\{\left(\frac{32\rho_0}{3\nu^2\beta}\right)^{\frac{3}{4}},\frac{32\left(\sum_{i=1}^m \rho_iB_{f_i}+\sum_{j=1}^n \sigma_jB_{c_j}\right)}{3\nu^2}\right\}\right\rceil-1,
		%	\end{align}
		and thus
		\small
		\begin{align}
			\label{eq:gammaksmall}
			\textstyle	\frac{2\gamma_k}{3\nu^2\beta (k+1)^{\frac{4}{3}}}\leq \frac{1}{4}, \, \forall\, k\ge K'.
		\end{align}
		\normalsize
		Applying \eqref{eq:gammaksmall} for $K' \le k\le K-2$ in \eqref{eq:nesterovconv2-vary-beta-new-nc} and also noting $\frac{\gamma_k}{2}\geq \frac{\gamma_k}{4}$, we obtain
		\small
		\begin{align}
			\nonumber
			&\sum_{k=0}^{K-1}\frac{\gamma_k}{4}\|\bar \bx^{(k+1)}-\bar\bx^{(k)}\|^2\\\nonumber
			\le&~ 2B_{f_0}+2G+\frac{\beta_0}{2}\left(\|\vc(\bx^{(0)})\|^2+\big\|[\vf(\bar \bx^{(0)})]_+\big\|^2\right)\\\label{eq:nesterovconv2-vary-beta-new1-nc}
			&~+\frac{8}{3\nu^2\beta}(1/\beta+B_{f_0}+M)^2+\frac{4D}{\beta }+\sum_{k=0}^{K'-1}\frac{2\gamma_k^2}{3\nu^2\beta (k+1)^{\frac{4}{3}}}\|\bar \bx^{(k+1)}-\bar\bx^{(k)}\|^2\leq\frac{\mathcal{C}_2}{4},
		\end{align}	
		\normalsize
		where we have used the definition of $\mathcal{C}_2$ in \eqref{eq:defconstantC-nc} and Assumption~\ref{assume:stochastic}B.
		%Note that the term $\sum_{k=0}^{K'-1}\frac{2\gamma_k^2}{3\nu^2\beta (k+1)^{\frac{4}{3}}}\|\bar \bx^{(k+1)}-\bar\bx^{(k)}\|^2$ in \eqref{eq:defconstantC-nc}
		%is a constant because $K'$ is finite and $\|\bar \bx^{(k+1)}-\bar\bx^{(k)}\|^2\leq D^2$.
		
		Using \eqref{eq:set-beta-vary-nc} and recalling $\Gamma_k$ in \eqref{eq:def-gamma-k}, we have $\sum_{k=0}^{K-1}\gamma_k=2\rho_0K+2\rho_c \sum_{k=0}^{K-1}\beta_k$, and	in addition, 
		$\sum_{k=0}^{K-1}\beta_k=\beta\sum_{k=0}^{K-1}(k+1)^{\frac{1}{3}}%\leq \beta\int_{1}^{K+1}x^{\frac{1}{3}}dx
		\leq {\beta}K^{\frac{4}{3}}.$ Therefore, by \eqref{eq:ineq5-varybeta-new-l1} and \eqref{eq:nesterovconv2-vary-beta-new1-nc}, it holds
		%	Inequality \eqref{eq:nesterovconv2-vary-beta-new1-nc} further implies
		\small
		\begin{equation}\label{eq:ineq5-varybeta-new-nc}
			%\nonumber
			\textstyle	\frac{1}{K}\sum_{k=0}^{K-1}\gamma_k\|\bar \bx^{(k+1)}-\bar\bx^{(k)}\|%&\leq&\frac{1}{K}\sqrt{\sum_{k=0}^{K-1}\gamma_k\|\bar\bx^{(k+1)}-\bar\bx^{(k)}\|^2}\sqrt{\sum_{k=0}^{K-1}\gamma_k}
			\leq\frac{\sqrt{\mathcal{C}_2}}{K}\sqrt{2\rho_0K + 2\rho_c\beta K^{\frac{4}{3}}}%\\%\nonumber
			%&\leq&\frac{\sqrt{\mathcal{C}_2}}{K}\sqrt{2\rho_0K + \frac{3\beta}{2}\left(\sum_{i=1}^m \rho_iB_{f_i}+\sum_{j=1}^n \sigma_jB_{c_j}\right)(K+1)^{\frac{4}{3}}}\\
			\leq \textstyle \sqrt{\frac{2\rho_0\mathcal{C}_2}{K}}+\frac{\sqrt{2\rho_c\beta\mathcal{C}_2}}{K^{1/3}}.
		\end{equation}
		\normalsize	
		%	where the first inequality is by Cauchy-Schwarz inequality, the second inequality is due to \eqref{eq:nesterovconv2-vary-beta-new1-nc} and the definition of $\gamma_k$ in \eqref{eq:set-beta-vary-nc}, and the last inequality is because . 
		%	
		%	By the definition of $\hat{\varepsilon}_k$ in \eqref{eq:set-beta-vary-nc}, we have
		%	\begin{align}
		%	\label{eq:nesterovconv2-vary-sumepsilonhat-nc}
		%	\sum_{k=0}^{K-1}\hat{\varepsilon}_k=\frac{1}{\beta }\sum_{k=0}^{K-1}(k+1)^{-\frac{4}{3}} \leq \frac{1}{\beta}\left(1+\int_{1}^{K}x^{-\frac{4}{3}}dx\right)\leq\frac{4}{\beta }.
		%	\end{align}
		Now apply \eqref{eq:nesterovconv2-vary-sumepsilonhat-nc} and \eqref{eq:ineq5-varybeta-new-nc} to \eqref{eq:xtilde-x-dual-feas-option3} to have
		\small
		\begin{eqnarray}\label{eq:xtilde-x-dual-feas-option3-avg-nc}
			\textstyle \frac{1}{K}\sum_{k=0}^{K-1}\textbf{S}_{k+1}
			\label{eq:term1avgk-nc}
			&\leq& \textstyle \frac{4}{\beta K}+\sqrt{\frac{2\rho_0\mathcal{C}_2}{K}}+\frac{\sqrt{2\rho_c\beta\mathcal{C}_2}}{K^{1/3}}.
		\end{eqnarray}
		\normalsize
		
		%\noindent\textbf{Bounding }$\mathbf{T_2}$ and $\mathbf{T_3}$:
		From the definition of $\textbf{F}_{k+1}$ in \eqref{eq:def-F}, we use \eqref{eq:boundinfeasibilitybycondition-nc} to have 
		\small
		\begin{align}
			\nonumber
			\textstyle	\frac{1}{K}\sum_{k=0}^{K-1}\textbf{F}_{k+1}
			\leq&~ \textstyle \frac{1}{K}\sum_{k=0}^{K-1}\frac{\hat\varepsilon_k+B_{f_0}+\gamma_k \|\bar\bx^{(k+1)}-\bar\bx^{(k)}\|+M}{\nu\beta_k}\\\nonumber
			\leq&~ \textstyle \frac{1}{\nu \beta K}\sum_{k=0}^{K-1}\hat\varepsilon_k+\frac{1}{\nu \beta K}\sum_{k=0}^{K-1}\gamma_k\|\bar \bx^{(k+1)}-\bar\bx^{(k)}\|+\frac{B_{f_0}+M}{\nu K}\sum_{k=0}^{K-1}\frac{1}{\beta_k}\\\label{eq:xtilde-x-stationary-adapbeta-option3-avg-nc}
			\leq&~ \textstyle \frac{4}{\nu \beta^2 K}+\frac{\sqrt{2\rho_0\mathcal{C}_2}}{\nu \beta \sqrt K}+\frac{\sqrt{2\rho_c\beta\mathcal{C}_2}}{\nu\beta K^{1/3}}+\frac{3B_{f_0}+3M}{2\nu \beta K^{1/3}},
		\end{align}
		\normalsize
		where the second inequality follows from $\beta_k\ge \beta$, and the last inequality holds because of \eqref{eq:nesterovconv2-vary-sumepsilonhat-nc},  \eqref{eq:ineq5-varybeta-new-nc}, and the fact that 
		\small
		\begin{equation}\label{eq:bd-beta-inverse}
			\textstyle \frac{1}{ K}\sum_{k=0}^{K-1}\frac{1}{\beta_k}=\frac{1}{\beta K}\sum_{k=0}^{K-1}\frac{1}{(k+1)^{1/3}}\leq \frac{1}{\beta K}\int_{0}^{K}x^{-\frac{1}{3}}dx\leq\frac{3}{2\beta K^{1/3}}.
		\end{equation} 
		\normalsize
		
		By \eqref{eq:boundinfeasibilitybycondition-nc} and the definition of $\lambda^{(k+1)}_i$ in Algorithm~\ref{alg:iPPP}, it holds
		\small
		\begin{eqnarray*}
			%\label{eq:xtilde-x-slack-adapbeta-option3-k-nc}
			\sum_{i=1}^m|\bar\lambda^{(k+1)}_i f_i(\bar\bx^{(k+1)})|=\beta_k\big\|[\vf(\bar\bx^{(k+1)})]_+\big\|^2
			\leq\frac{(\hat\varepsilon_k+B_{f_0}+\gamma_k \|\bar\bx^{(k+1)}-\bar\bx^{(k)}\|+M)^2}{\nu^2\beta_k}.
		\end{eqnarray*}
		\normalsize 
		From the definition of $\textbf{C}_{k+1}$ in \eqref{eq:def-C}, we 
		average both sides of the above inequality to have %\eqref{eq:xtilde-x-slack-adapbeta-option3-k-nc}  and for $k=0,1,\dots,K-1$, we have
		\small
		\begin{align}
			\nonumber
			\textstyle		\frac{1}{K}\sum_{k=0}^{K-1}\textbf{C}_{k+1}
			\leq&~ \textstyle \frac{1}{K}\sum_{k=0}^{K-1}\frac{(\hat\varepsilon_k+B_{f_0}+\gamma_k \|\bar\bx^{(k+1)}-\bar\bx^{(k)}\|+M)^2}{\nu^2\beta_k}\\\nonumber
			\leq&~ \textstyle \frac{3}{\nu^2K}\sum_{k=0}^{K-1}\frac{\hat\varepsilon_k^2}{\beta_k}+\frac{3}{K}\sum_{k=0}^{K-1}\frac{(B_{f_0}+M)^2}{\nu^2\beta_k}+\frac{3}{K}\sum_{k=0}^{K-1}\frac{\gamma_k^2 \|\bar\bx^{(k+1)}-\bar\bx^{(k)}\|^2}{\nu^2\beta_k}\\\label{eq:xtilde-x-slack-adapbeta-option3-avg-nc-temp1}
			\leq&~\textstyle \frac{9}{2\nu^2\beta^3 K}+\frac{9(B_{f_0}+M)^2}{2\nu^2 \beta K^{1/3}}+\frac{3}{K}\sum_{k=0}^{K-1}\frac{\gamma_k^2 \|\bar\bx^{(k+1)}-\bar\bx^{(k)}\|^2}{\nu^2\beta_k},
		\end{align}
		\normalsize
		where the second inequality uses $(a+b+c)^2\leq 3a^2+3b^2+3c^2$, and the third inequality follows from \eqref{eq:bd-beta-inverse} %and $\frac{1}{ K}\sum_{k=0}^{K-1}\frac{1}{\beta_k}=\frac{1}{\beta K}\sum_{k=0}^{K-1}\frac{1}{(k+1)^{1/3}}\leq \frac{1}{\beta K}\int_{0}^{K}x^{-\frac{1}{3}}dx\leq\frac{3}{2\beta K^{1/3}}$ 
		and 
		\small	
		$$\textstyle \sum_{k=0}^{K-1}\frac{\hat\varepsilon_k^2}{\beta_k}=\frac{1}{\beta^3}\sum_{k=0}^{K-1}\frac{1}{(k+1)^3}\leq \frac{1}{\beta^3}\left(1+\int_1^{K}\frac{1}{x^3}dx\right)\leq \frac{3}{2\beta^3}.$$ 
		\normalsize	
		Recall $\gamma_k=2\rho_0+2\beta_k\rho_c$ for all $k\ge0$ and also note $\beta_k\ge \beta$. We have $\frac{\gamma_k^2}{\beta_k} \le \gamma_k(\frac{2\rho_0}{\beta}+2\rho_c)$, and thus by \eqref{eq:nesterovconv2-vary-beta-new1-nc}, it holds
		\small
		\begin{align}\label{eq:xtilde-x-slack-adapbeta-option3-avg-nc-temp2}
			&~\textstyle \frac{3}{K}\sum_{k=0}^{K-1}\frac{\gamma_k^2 \|\bar\bx^{(k+1)}-\bar\bx^{(k)}\|^2}{\nu^2\beta_k}
			%=&~\frac{6\rho_0}{\nu^2K}\sum_{k=0}^{K-1}\frac{\gamma_k\|\bar\bx^{(k+1)}-\bar\bx^{(k)}\|^2}{\beta_k}
			%+\frac{6\rho_c}{\nu^2K}\sum_{k=0}^{K-1}\gamma_k\|\bar\bx^{(k+1)}-\bar\bx^{(k)}\|^2\\
			\leq\frac{6\mathcal{C}_2(\rho_0/\beta+\rho_c)}{\nu^2K},
		\end{align}
		\normalsize
		%	where the last inequality is because of the fact that $\beta_k\geq\beta$ and inequality \eqref{eq:nesterovconv2-vary-beta-new1-nc}. 
		Now apply \eqref{eq:xtilde-x-slack-adapbeta-option3-avg-nc-temp2} to \eqref{eq:xtilde-x-slack-adapbeta-option3-avg-nc-temp1} to have
		\small
		\begin{align}\label{eq:xtilde-x-slack-adapbeta-option3-avg-nc}
			\textstyle \frac{1}{K}\sum_{k=0}^{K-1}\textbf{C}_{k+1}
			\leq\frac{9}{2\nu^2\beta^3 K}+\frac{9(B_{f_0}+M)^2}{2\nu^2 \beta K^{1/3}}+\frac{6\mathcal{C}_2(\rho_0/\beta+\rho_c)}{\nu^2K}.
		\end{align}
		\normalsize
		Plugging \eqref{eq:term1avgk-nc}, \eqref{eq:xtilde-x-stationary-adapbeta-option3-avg-nc}, and \eqref{eq:xtilde-x-slack-adapbeta-option3-avg-nc} into \eqref{eq:xtilde-x-dual-feas-option3} gives \eqref{eq:convergeinK1-nc-02}.
	\end{proof}

	\begin{corollary}[complexity result]
		Under the same assumptions of Theorem~\ref{thm:complexity-varying-beta-nc}, let  
		$K=\left\lceil\max\big\{ K_3 , K_4, K_5 \big\} \right\rceil = O(1/\varepsilon^3)$ with
		\small
		\begin{align*}
			\textstyle	K_3=\frac{3}{\varepsilon}\left(\frac{4}{\beta }+\frac{4}{\nu \beta^2 }+\frac{9}{2\nu^2\beta^3 }+\frac{6\mathcal{C}_2(\rho_0/\beta+\rho_c)}{\nu^2}\right),\ K_4=\frac{18\rho_0\mathcal{C}_2}{\varepsilon^2}\left(1+\frac{1}{\nu\beta}\right)^2,\\ 
			\textstyle	K_5=\frac{27}{\varepsilon^3}\left(\left(1+\frac{1}{\nu\beta}\right)\sqrt{2\beta \rho_c\mathcal{C}_2}+\frac{3B_{f_0}+3M}{2\nu \beta}+\frac{9(B_{f_0}+M)^2}{2\nu^2 \beta }\right)^3,
		\end{align*}
		\normalsize
		where $\mathcal{C}_2$ is defined as in \eqref{eq:defconstantC-nc}. Then 
		%	the following statements hold:
		%	\begin{enumerate}
		%	\item If $R_k$ is generated by Option I, then 
		$\bar\bx^{(R_K)}$ is an $\varepsilon$-stationary point of \eqref{eq:gco}, and 
		%\item 
		the total complexity to produce $\bar\bx^{(R_K)}$ is $\tilde O\left(1/\varepsilon^3\right)$.
		%\end{enumerate}
	\end{corollary}
	
	\begin{proof}
		With the given $K$, the right hand side of  \eqref{eq:convergeinK1-nc-02} is upper bounded by $\varepsilon$, so  $\bar\bx^{(R_K)}$ is an $\varepsilon$-stationary point of \eqref{eq:gco}. 
		
		Let $T_k$ be the number of proximal gradient steps performed by Algorithm~\ref{alg:adap-APG1} in the $k$-call by Algorithm~\ref{alg:iPPP}.
		Then	 according to Theorem~\ref{thm:conv-adap-APG1} and the definitions of $\Gamma_k$, $\gamma_k$,  $\beta_k$, and $L_{\phi_k}$ in \eqref{eq:Lk-case2}, \eqref{eq:def-gamma-k} and \eqref{eq:set-beta-vary-nc}, we have 
		$
		\textstyle	T_k=\tilde O\left(\sqrt{\frac{L_{\phi_k}}{\gamma_k-\Gamma_k}}\right) =\tilde O\left(\sqrt{\frac{\Gamma_k+\beta_k}{\Gamma_k}}\right)=\tilde O\left(1\right),
		$ 
		for $k=0,1,\dots,K-1$.
		Therefore, the total complexity is
		$\textstyle T_{\mathrm{total}} = \sum_{k=0}^{K-1} T_k = \tilde O\left(K\right) = \tilde O(1/\varepsilon^3),$
		which completes the proof.
		%	Plugging $K=\max\big\{\lceil K_3 \rceil, \ \lceil K_4\rceil, \ \lceil K_5\rceil \big\}=O(\frac{1}{\varepsilon^3})$ into the above equation, we obtain a total complexity $\tilde O\left(1/\varepsilon^3\right)$ for Algorithm~\ref{alg:iPPP}.
	\end{proof}

	\subsection{The complexity of the iPPP method under initial feasibility assumption}
	In this subsection, we drop Assumption~\ref{assume:nonconvexconstraintsingular} and analyze the complexity of the proposed iPPP method by starting from an initial feasible point, namely, in addition to Assumptions~\ref{assume:stochastic} and \ref{assume:nonconvexconstraintnew}, we assume the follows.  
	\begin{assumption}
		\label{assume:nonconvexconstraintfeasible}
		The initial solution $\bar\bx^{(0)}\in \X$ in Algorithm~\ref{alg:iPPP} is feasible, i.e., $f_i(\bar\bx^{(0)})\le 0$ for each $i = 1,\ldots,m$ and $c_j(\bar\bx^{(0)}) = 0$ for each $j=1,\ldots,n$. %there exists a solution $\bx_{\mathrm{feas}}\in\X$ satisfying $\bA\bx_{\mathrm{feas}}=\bb$ and $f_i(\bx_{\mathrm{feas}})\leq0$ for $i=1,\dots,m$ and $\bx_{\mathrm{feas}}$ can be computed.
	\end{assumption}
	
	\begin{remark}
		This feasibility assumption on $\bar\bx^{(0)}$ can be weakened to near-feasibility depending on the required accuracy. Unless with certain regularity conditions like the one we assumed in the previous subsection, or with certain special structures, it is generally impossible to find a (near) feasible solution of a nonlinear system in a polynomial time. %it can be found for some problems with special structures. 
		Existing works, such as \cite{cartis2011evaluation,boob2019proximal, ma2019proximally}, also need the (near)-feasibility assumption to guarantee a near-stationary point. 
	\end{remark}

	%\subsection{A complexity result by constant penalty parameters}
	%Under the new assumptions, we have to adopt a different notion of near-stationary point, namely the weak $\varepsilon$-stationary point defined in  Definition~\ref{def:weak-eps-stationary-pt}. An weak $\varepsilon$-stationary point is slightly weaker than Definition~\ref{def:eps-stationary-pt} by dropping the near-complementarity requirement \eqref{eq:ecomplementaryslackness}. 
	
	Below, we specify the parameters of  Algorithm~\ref{alg:iPPP} and analyze its complexity with Option II to find a weak $\varepsilon$-stationary point of \eqref{eq:gco}. %under Assumptions~\ref{assume:stochastic},  \ref{assume:nonconvexconstraintnew}, and \ref{assume:nonconvexconstraintfeasible}. 
	\begin{theorem}\label{thm:complexity-constant-beta-new}
		Suppose that Assumptions~\ref{assume:stochastic},  \ref{assume:nonconvexconstraintnew}, and \ref{assume:nonconvexconstraintfeasible} hold and the parameters $\{\gamma_k\}$, $\{\beta_k\}$ and  $\{\hat{\varepsilon}_k\}$ in Algorithm~\ref{alg:iPPP} are taken as 
		\begin{equation}\label{eq:set-beta-constant}
			\beta_k=\beta,\quad \gamma_k=2(\rho_0+\beta\rho_c), \quad\text{ and } \quad \textstyle \hat{\varepsilon}_k= \frac{1}{ (k+1)^2 },\, \forall\, k\ge 0,
		\end{equation}  
		where $\beta>0$ is a constant, and $\rho_c$ is defined in \eqref{eq:def-gamma-k}. 
		If $\{R_k\}$ is generated by Option II, then for any $K\ge1$, it holds that 
		\small
		\begin{eqnarray}
			%	\nonumber
			%	&&\max\left\{
			%	\begin{array}{c}
			%	\left\|\nabla f_0(\bar\bx^{(R_K)})+\sum_{j=1}^n\bar y_j^{(R_K)}\nabla c_j(\bar\bx^{(R_K)})+\sum_{i=1}^m \bar\lambda^{(R_K)}_i \nabla f_i(\bar\bx^{(R_K)})\right\|,\\
			%	\sqrt{\|\vc(\bar\bx^{(R_K)})\|^2+\big\|[\vf(\bar\bx^{(R_K)})]_+\big\|^2}\\
			%	\end{array}
			%	\right\}\\
			\max\left\{\textbf{S}_{R_K},\, \textbf{F}_{R_K}\right\}
			&\leq&\textstyle \frac{\pi^2}{6K}+ \frac{\sqrt{\mathcal{C}_3}}{\sqrt{K}}+\sqrt{\frac{4(B_{f_0}+G) + \pi^2D/3}{\beta}}\label{eq:convergeinK1-nc},
		\end{eqnarray}
		\normalsize
		where 
		\small
		\begin{eqnarray}
			\label{eq:defconstantC-nc-const}
			\textstyle	\mathcal{C}_3:=\left(2\rho_0 + 2\beta\rho_c\right)\left(4(B_{f_0}+G) + \frac{\pi^2D}{3}\right).
		\end{eqnarray}
		\normalsize
	\end{theorem}
	
	\begin{proof}
		%		Suppose $R_k$ is generated by Option II. By the definitions of $\textbf{S}_{k+1}$ and $\textbf{F}_{k+1}$ in Algorithm~\ref{alg:iPPP}, we can show the following inequality similar to \eqref{eq:threeterms}
		%					
		%		Next, we will bound the two summations on the right hand side of \eqref{eq:twoterms}.
		%	
		%	Following the same argument used to prove \eqref{eq:xtilde-x-dual-feas-option3} in the proof of Theorem~\ref{thm:complexity-varying-beta-new}, we can show that \eqref{eq:xtilde-x-dual-feas-option3} holds here.  Following the same argument used to prove \eqref{eq:nesterovconv2-vary-beta-nc} in the proof of Theorem~\ref{thm:complexity-varying-beta-nc} except dropping the non-negative term $\frac{\beta_{K-1}}{2}\left(\|\vc(\bar \bx^{(K)})\|^2+\big\|[\vf(\bar \bx^{(K)})]_+\big\|^2\right)$, we can show 	
		By the setting $\beta_k=\beta, \,\forall\, k\ge 0$ and $\sum_{k=0}^{K-1}\hat{\varepsilon}_k=\sum_{k=0}^{K-1}\frac{1}{(k+1)^2}\leq\frac{\pi^2}{6}$, we obtain from \eqref{eq:nesterovconv2-vary-beta-nc-constant} and also the feasibility of $\bar\bx^{(0)}$ that 
		\small
		\begin{align}
			%\nonumber
			\textstyle	\sum_{k=0}^{K-1}\frac{\gamma_k}{2}\|\bar\bx^{(k+1)}-\bar\bx^{(k)}\|^2+\frac{\beta}{2}\big\|\vc(\bar\bx^{(K)})\big\|^2+\frac{\beta}{2}\big\|[\vf(\bar\bx^{(K)})]_+\big\|^2 %\\%\nonumber
			%&\le &f_0(\bar\bx^{(0)})+g(\bar\bx^{(0)}) -f_0(\bar\bx^{(K)})-g(\bar\bx^{(K)})+\frac{\bar\beta}{2}\big\|\vc(\bar\bx^{(0)})\big\|^2+\frac{\bar\beta}{2}\big\|[\vf(\bar\bx^{(0)})]_+\big\|^2+  \frac{\pi^2D}{6}\\
			\le  2(B_{f_0}+G) + \frac{\pi^2D}{6}.\label{eq:nonconvexineq1-const-beta}
		\end{align}
		\normalsize
		%where the last inequality is from the feasibility of $\bar\bx^{(0)}$ and Assumption~\ref{assume:stochastic} and \eqref{eq:bd-fi} which ensure $|f_0(\bx)+g(\bx)|\le B_{f_0}+G,\,\forall\,\bx\in\X$.  
		Hence, from \eqref{eq:ineq5-varybeta-new-l1} and \eqref{eq:nonconvexineq1-const-beta} and also the setting of $\gamma_k$ in \eqref{eq:set-beta-constant}, we have
		%Inequality 	\eqref{eq:nonconvexineq1-const-beta} further implies
		\small
		\begin{align}\label{eq:ineq5-constbeta-new-nc}
			%\nonumber
			&\textstyle \frac{1}{K}\sum_{k=0}^{K-1}\gamma_k\|\bar \bx^{(k+1)}-\bar\bx^{(k)}\|%\\\nonumber
			%&\leq&\frac{1}{K}\sqrt{\sum_{k=0}^{K-1}\gamma_k\|\bar\bx^{(k+1)}-\bar\bx^{(k)}\|^2}\sqrt{\sum_{k=0}^{K-1}\gamma_k}\\
			\leq\frac{1}{K}\sqrt{4(B_{f_0}+G) + \frac{\pi^2D}{3}}\sqrt{2\rho_0K + 2K\beta\rho_c}=\frac{\sqrt{\mathcal{C}_3}}{\sqrt{K}}.
		\end{align}
		\normalsize	
		%where the first inequality is by Cauchy-Schwarz inequality and the second inequality is because of the definitions of $\gamma_k$ and $\Gamma_k$ (see  \eqref{eq:def-gamma-k}) and the fact that $\beta_k=\beta$. 	
		Applying \eqref{eq:ineq5-constbeta-new-nc} and the fact that $\sum_{k=0}^{K-1}\hat{\varepsilon}_k\leq\frac{\pi^2}{6}$ to \eqref{eq:xtilde-x-dual-feas-option3} leads to
		\small
		\begin{equation}%\label{eq:xtilde-x-dual-feas-option3-avg-nc-const}
			\textstyle	\frac{1}{K}\sum_{k=0}^{K-1}\textbf{S}_{k+1}
			\label{eq:term1avgk-nc-const}
			\leq\frac{\pi^2}{6K}+\frac{\sqrt{\mathcal{C}_3}}{\sqrt{K}}.
		\end{equation}
		\normalsize
		
		In addition, notice that \eqref{eq:nonconvexineq1-const-beta} actually holds for any $K\geq1$. Hence, %that, for all $k\geq0$, 
		\small
		\begin{eqnarray*}
			\label{eq:feasibility-t-nc}
			\textstyle	\frac{\beta}{2}\big\|\vc(\bar\bx^{(k+1)})\big\|^2+\frac{\beta}{2}\big\|[\vf(\bar\bx^{(k+1)})]_+\big\|^2
			\leq2(B_{f_0}+G) + \frac{\pi^2D}{6},\,\forall\, k\ge 0,
		\end{eqnarray*}
		\normalsize	
		which, together with the definition of $\textbf{F}_{k+1}$ in \eqref{eq:def-F}, implies
		\small
		\begin{eqnarray}
			\label{eq:feasibility-t-nc-bound}
			\textstyle	\frac{1}{K}\sum_{k=0}^{K-1}\textbf{F}_{k+1} \le \sqrt{\frac{4(B_{f_0}+G) + \pi^2D/3}{\beta}}.
		\end{eqnarray}
		\normalsize
		Now plugging \eqref{eq:term1avgk-nc-const} and \eqref{eq:feasibility-t-nc-bound} into \eqref{eq:twoterms} gives the desired result.
		%	Then the conclusion \eqref{eq:convergeinK1-nc} is implied by applying .
	\end{proof}
	
	\begin{corollary}[complexity result]\label{thm:eps-stationary-const-beta}
		Under the same assumptions of Theorem~\ref{thm:complexity-constant-beta-new}, let 
		$\beta=\frac{36(B_{f_0}+G) + 3\pi^2D}{\varepsilon^2}$ and 
		$\textstyle K = \left\lceil\max\Big\{ \frac{9\mathcal{C}_3}{\varepsilon^2} , \frac{\pi^2}{2\varepsilon} \Big\}\right\rceil = O\left(\frac{1} {\varepsilon^{4}}\right),$ 
		%with $K_6= \frac{9\mathcal{C}_3}{\varepsilon^2}$ and $K_7 = \frac{\pi^2}{2\varepsilon}$, 
		where $\mathcal{C}_3$ is defined in \eqref{eq:defconstantC-nc-const}. Then
		%	the following statements hold:
		%\begin{enumerate}
		%\item If $R_k$ is generated by Option II in Algorithm~\ref{alg:iPPP}, then 
		$\bar\bx^{(R_K)}$ is a weak $\varepsilon$-stationary point of \eqref{eq:gco-cvx}, and
		%\item 
		the total complexity to produce $\bar\bx^{(R_K)}$ %finding  a weak $\varepsilon$-stationary point of \eqref{eq:gco-cvx} 
		is $\tilde O\left(1/\varepsilon^4\right).$
		%\end{enumerate}
	\end{corollary}
	\begin{proof}
		%When $\beta\geq\frac{36(B_{f_0}+G) + 3\pi^2D}{\varepsilon^2}$, we have $\sqrt{\frac{4(B_{f_0}+G) + \pi^2D/3}{\beta}}\leq \frac{\varepsilon}{3}$.
		%When $K\geq\max\big\{K_6, K_7\}$, 
		With the chosen $\beta$ and $K$, it holds that $\sqrt{\frac{4(B_{f_0}+G) + \pi^2D/3}{\beta}}\leq \frac{\varepsilon}{3}$ and $\frac{\pi^2}{6K}+ \frac{\sqrt{\mathcal{C}_3}}{\sqrt{K}}\le \frac{2\varepsilon}{3}$.  %the first two terms on the left hand sides of  \eqref{eq:convergeinK1-nc} are bounded by $\frac{\varepsilon}{3}$ so that  
		Hence, by \eqref{eq:convergeinK1-nc}, $\bar\bx^{(R_K)}$ is a weak $\varepsilon$-stationary point of \eqref{eq:gco}.
		
		%Furthermore, 
		Let $T_k$ be the number of proximal gradient steps performed by Algorithm~\ref{alg:adap-APG1} in the $k$-th call by Algorithm~\ref{alg:iPPP}.
		%Given the fact that $\gamma_k=2\Gamma_k$, the fact that $\beta_k=\beta$, the definition of $L_{\phi_k}$ in \eqref{eq:Lk-case2}, and the definition of $\Gamma_k$ in \eqref{eq:def-gamma-k}, we have $\Gamma_k=\Theta(\beta)$ and $L_{\phi_k}=\Theta(\gamma_k+\beta_k)=\Theta(\beta)$ so that 
		Notice that with the parameters set in \eqref{eq:set-beta-constant}, $\phi_k$ defined in \eqref{eq:compositepart} is $(\rho_0+\beta\rho_c)$-strongly convex, and in addition, its smoothness constant $L_{\phi_k}=\Theta(\gamma_k+\beta_k)=\Theta(\beta)$. Hence, according to Theorem~\ref{thm:conv-adap-APG1},
		$
		T_k=\tilde O\left(\sqrt{\frac{L_{\phi_k}}{\rho_0+\beta\rho_c}}\right)%=\tilde O\left(\sqrt{\frac{L_{\phi_k}}{\Gamma_k}}\right)
		=\tilde O\left(1\right)
		$ 
		for all $k\ge 0$. Therefore, the total complexity  
		$T_{\mathrm{total}} = \sum_{k=0}^{K-1} T_k = \tilde O\left(K\right)= \tilde O(1/\varepsilon^4),$ which completes the proof.
		%	Note that $\mathcal{C}_3=\Theta(\beta)=\Theta(\frac{1}{\varepsilon^2})$, $K_6=\Theta(\frac{1}{\varepsilon^4})$, and $K_7=\Theta(\frac{1}{\varepsilon})$.
		%	Plugging $K=\max\big\{\lceil K_6 \rceil, \lceil K_7\rceil \big\}=O(\frac{\gamma}{\varepsilon^4})$ into the above equation, we obtain a total complexity $\tilde O\left(1/\varepsilon^{4}\right)$ for Algorithm~\ref{alg:iPPP}.
	\end{proof}

	\begin{remark}
		Notice that in Corollary~\ref{thm:eps-stationary-const-beta}, we only guarantee a weak $\varepsilon$-stationary point because no constraint qualification (CQ) is assumed. Without a CQ, even a global optimal solution is not guaranteed to be a KKT point. 
	\end{remark}

	\section{Numerical Experiments}
	\label{sec:exp}
	In spite of the theoretical focus of this paper, we want to evaluate the numerical performance of the iPPP method on a multi-class Neyman-Pearson classification (mNPC) problem in this section. Suppose there is a set of training data with $K$ classes, denoted by $\mathcal{D}_k\subseteq\mathbb{R}^d$ for $k=1,2,\dots,K$. %To classify each data into one of the $K$ classes, we want 
	The goal is to learn $K$ linear models $\bx_k$, $k=1,2,\dots,K$ and predict the class of a data point $\xi$ as 
	$\argmax_{k=1,2,\dots,K}\bx_k^\top \xi$. To achieve a high classification accuracy, $\{\bx_k\}$ is found such that $\bx_k^\top \xi-\bx_l^\top \xi$ is positively large for any $k\neq l$ and any $\xi\in\mathcal{D}_k$~\cite{weston1998multi,crammer2002learnability}. This leads to minimizing the average loss 
	$
	\frac{1}{|\mathcal{D}_k|}\sum_{l\neq k}\sum_{\xi\in\mathcal{D}_k}\phi(\bx_k^\top \xi-\bx_l^\top \xi),
	$  
	where $\phi$ is a non-increasing (potentially non-convex) loss function. Suppose misclassifying $\xi$ has a cost depending on its true class label $k$.
	When training these $K$ linear models, the mNPC prioritizes minimizing the loss on one class, say $\mathcal{D}_1$, and meanwhile controls the losses on other classes, namely, %by solving 
	\begin{eqnarray}
		\label{eq:NPclassification_multi}
		\min_{\|\bx_k\|\leq\lambda,k=1,\dots,K}&& \textstyle \frac{1}{|\mathcal{D}_1|}\sum_{l > 1}\sum_{\xi\in\mathcal{D}_1}\phi(\bx_1^\top \xi-\bx_l^\top \xi),\\\nonumber
		\mathrm{s.t.}&&\textstyle \frac{1}{|\mathcal{D}_k|}\sum_{l\neq k}\sum_{\xi\in\mathcal{D}_k}\phi(\bx_k^\top \xi-\bx_l^\top \xi)\leq r_k,\quad k=2,3,\dots,K.
	\end{eqnarray}
	Here, $r_k$ controls the loss for $\mathcal{D}_k$, and $\lambda > 0$ is a regularization parameter. 
	
	We created test instances of \eqref{eq:NPclassification_multi} using the LIBSVM  multi-class classification datasets  \textit{covtype} and \textit{mnist}, which have $K=7$ and $K=10$ classes, respectively. 
	%Their statistics are summarized in Table~\ref{tab:datasets}, including the number of classes, the number of instances, and the number of features. 
	The first class of each dataset is used to formulate the objective function in \eqref{eq:NPclassification_multi}, and the other classes are used to formulate the constraints. The function $\phi$ in \eqref{eq:NPclassification_multi} is chosen as the sigmoid function $\phi(z)=1/(1+\exp(z))$. We set $ r_k=0.5(K-1),\,\forall\, k=2,\dots,K$ and set $\lambda = 0.3$ for both datasets.
	
	%\begin{table}[htp]
	%	\begin{center}
	%		\begin{tabular}{|c|c|c|c|}
	%			\hline
	%			Dataset	&Number of  classes & Number of instances & Number of  features \\
	%			\hline
	%			{covtype}& 7 & 581012 &  54   \\
	%			{mnist} & 10 & 60000 & 780   \\
	%			{pendigits} & 10 & 7494 & 16   \\
	%			{segment}   &7 & 2310 &  7\\
	%			\hline
	%		\end{tabular}
	%		\caption{Characteristics of multi-class classification datasets
	%			from LIBSVM library}
	%		%\cite{LIBSVMdata}.
	%		%    (available from the LIBSVM web page:
	%		%    \texttt{http://www.csie.ntu.edu.tw/\~{}cjlin/libsvmtools/datasets}).}
	%		\label{tab:datasets}
	%	\end{center}
	%\end{table}
	%%\vspace{-10pt}
	
	We compare the proposed method to the exact penalty method proposed in \cite{cartis2011evaluation}. We choose \cite{cartis2011evaluation} because \eqref{eq:NPclassification_multi} satisfies their assumptions and the complexity of \cite{cartis2011evaluation} can be lower or higher than ours (depending on the boundedness of their penalty parameters) while most of other methods either have a higher complexity than ours or have no theoretical guarantee because \eqref{eq:NPclassification_multi} does not satisfy their assumptions (e.g. linear constraints). 
	
	Both methods are implemented in Matlab on a 64-bit MacOS Catalina machine with a 4.20 Ghz Intel Core i7-7700K CPU and 16GB of memory. For both algorithms, the initial iterate is set to $\bx^{(0)}=\mathbf{0}$ and we verify that it is a feasible solution of \eqref{eq:NPclassification_multi} with $r_k$'s chosen above. We also discuss how Assumption~\ref{assume:nonconvexconstraintsingular} can potentially hold for problem \eqref{eq:NPclassification_multi} in Appendix~\ref{sec:append}.
	
	On solving \eqref{eq:gco} with $g\equiv 0$, the method in \cite{cartis2011evaluation} applies a non-smooth trust-region method to solve a sequence of unconstrained subproblems in the form of %\comm{How we change the inequality-constrained problem in \eqref{eq:NPclassification_multi} to an equality-constrained one?Qihang: Runchao did not wrote their subproblem clearly. I have updated it. Their method did not need to change inequalities to equalities.}
	\small
	\begin{eqnarray}
		\label{eq:exactpenalty}
		\min_{\bx}f_0(\bx)+\rho {\textstyle \sum_{i=1}^m  \big[f_i(\bx)\big]_+}+\rho {\textstyle \sum_{j=1}^n  \big|c_j(\bx)\big|},
	\end{eqnarray}
	\normalsize
	where $\rho>0$ is a penalty parameter which will be increased sequentially.  At iteration $k$ of the non-smooth trust-region method for solving \eqref{eq:exactpenalty}, an updating direction is computed as\footnote{The method in \cite{cartis2011evaluation} allows using any norm in the ball constraint of \eqref{eq:exactpenalty-tr}. Here, we choose $\ell_1$-norm so that \eqref{eq:exactpenalty-tr} can be solved as a linear program.  } 
	\small
	\begin{align}
		\label{eq:exactpenalty-tr}
		\mathbf{s}^{(k)}\in\argmin_{\|\mathbf{s}\|_1\leq \Delta_k}
		\left\{
		\begin{array}{l}
			%\ell(\bx^{(k)},\mathbf{s}):=
			f_0(\bx^{(k)})+\nabla f_0(\bx^{(k)})^\top\mathbf{s} 
			+\rho \sum_{i=1}^m\left [f_i(\bx^{(k)})+\nabla f_i(\bx^{(k)})^\top\mathbf{s}\right ]_+  \\  
			+\rho \sum_{j=1}^n\left |c_j(\bx^{(k)})+\nabla c_j(\bx^{(k)})^\top\mathbf{s}\right | 
		\end{array}
		\right\},
	\end{align}
	\normalsize
	where $\Delta_k$ is the radius of the trust region. Upon obtaining $\bs^{(k)}$, the estimated solution is updated to $\bx^{(k+1)}=\bx^{(k)}+\mathbf{s}^{(k)}$ if this update significantly reduces the objective value of \eqref{eq:exactpenalty}. Once an $\varepsilon$-critical point of~\eqref{eq:exactpenalty} (see equation (2.2) in~\cite{cartis2011evaluation} for the definition) is found, a steering procedure~\cite{byrd2005convergence}
	is utilized to increase the penalty parameter $\rho$ in~\eqref{eq:exactpenalty}. 
	
	The values of algorithm-related parameters in both algorithms are selected from a discrete set of candidates based on the value of the objective function after 10,000 data passes. In our implementation,  we formulate the problem in \eqref{eq:exactpenalty-tr}  as a linear program %\comm{how it can be LP? The constraint is $\|\mathbf{s}\|_1\leq \Delta_k$?Qihang: Yes, we use L1 norm. See the footnote I added below} 
	and then use Matlab built-in LP solver to obtain $s_k$. The outer iterations in the method by \cite{cartis2011evaluation} require a steering parameter $\xi$, an increase factor to update $\rho$, an initial value of $\rho$, and a tolerance for solving subproblem \eqref{eq:exactpenalty}. Steering parameter $\xi$ is set to be 0.3 for \textit{covtype} and 0.1 for \textit{mnist}. The initial value of $\rho$ is set to be $ 1 / \xi$ for both datasets. We choose the increase factor to be 10 and tolerance $\varepsilon = 0.001$ for both datasets. Moreover, the trust-region method for solving subproblem~\eqref{eq:exactpenalty} requires five control parameters: $\Delta_0$, $\eta_1$, $\eta_2$, $\gamma_1$, and $\gamma_2$. For both datasets, we choose $\Delta_0 = 1, \eta_1 = 0.3, \eta_2 = 0.7$, $\gamma_1 = 0.3$, and $\gamma_2 = 0.7$. 
	
	\begin{figure*}[tbhp]
		\centering
		{\includegraphics[scale=.2]{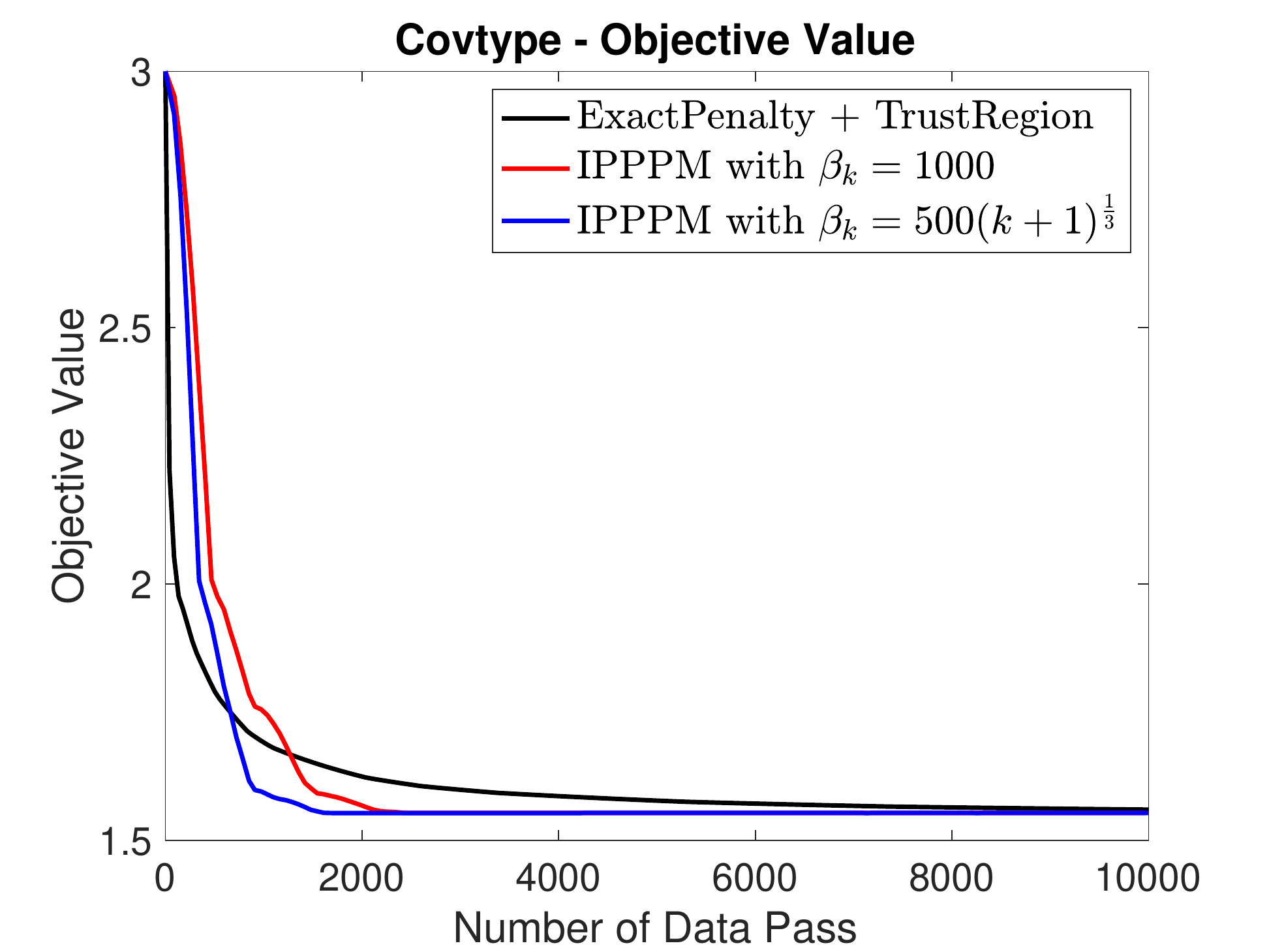}}
		{\includegraphics[scale=.2]{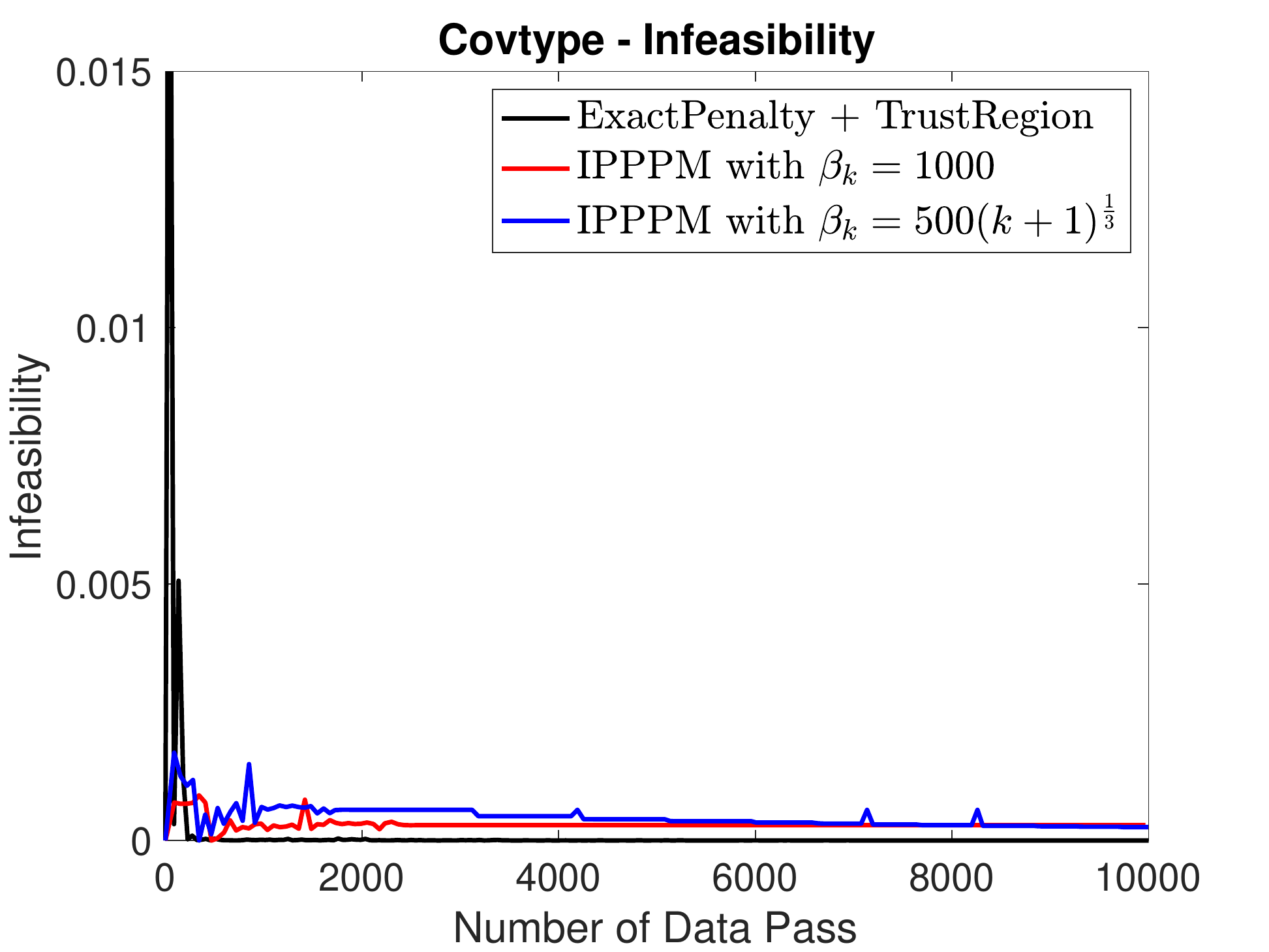}}
		{\includegraphics[scale=.2]{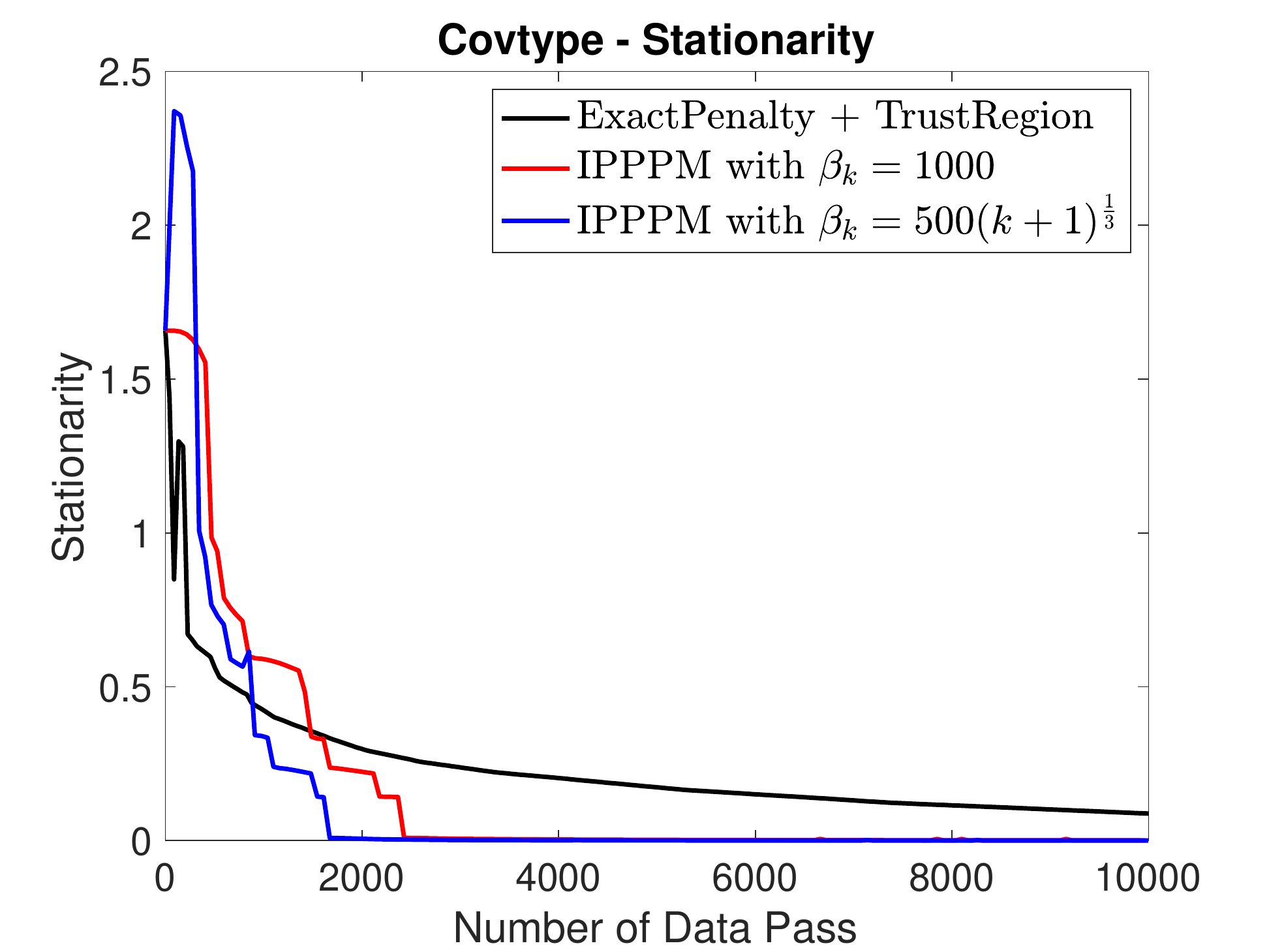}}\\
		{\includegraphics[scale=.2]{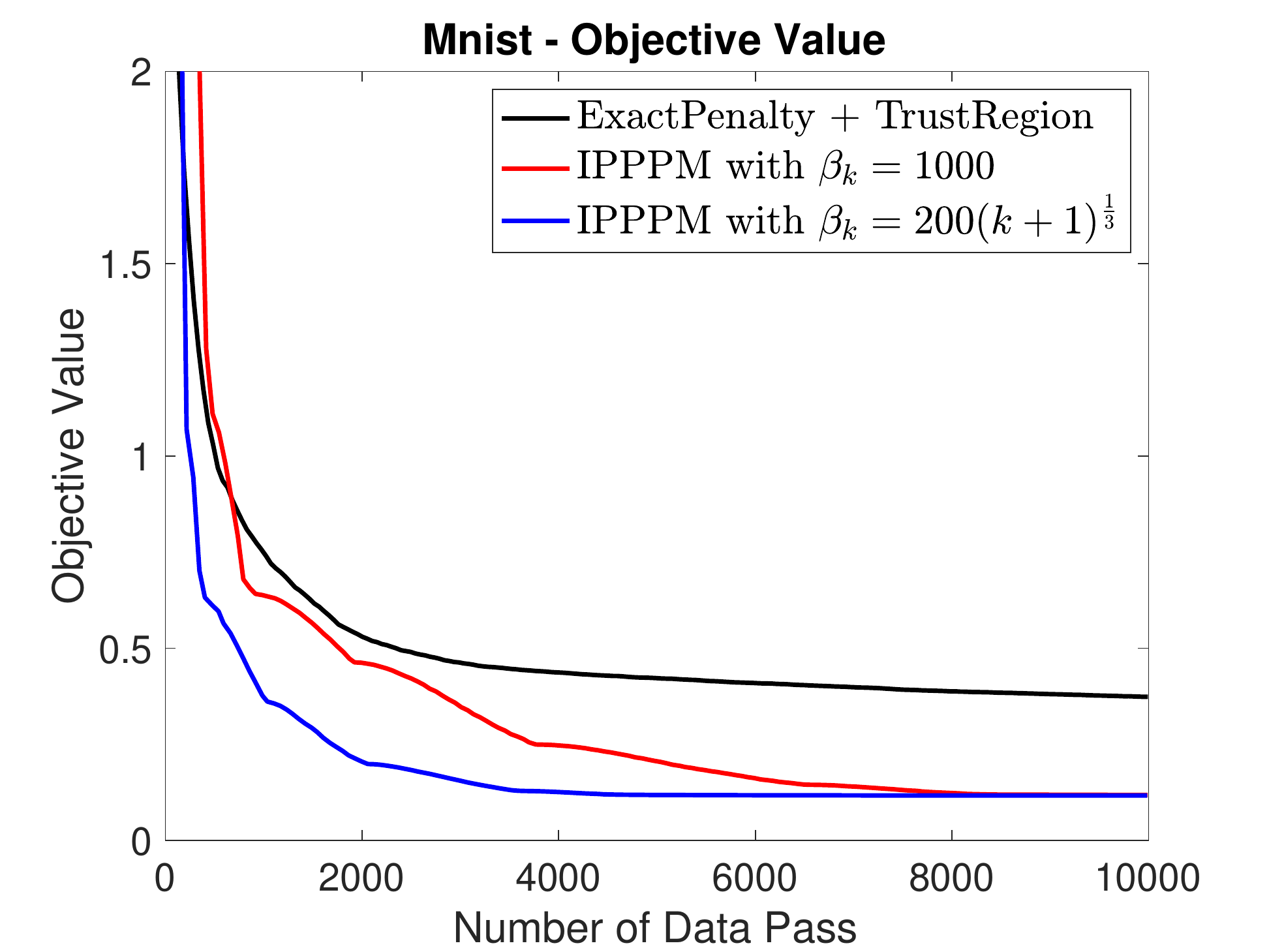}}
		{\includegraphics[scale=.2]{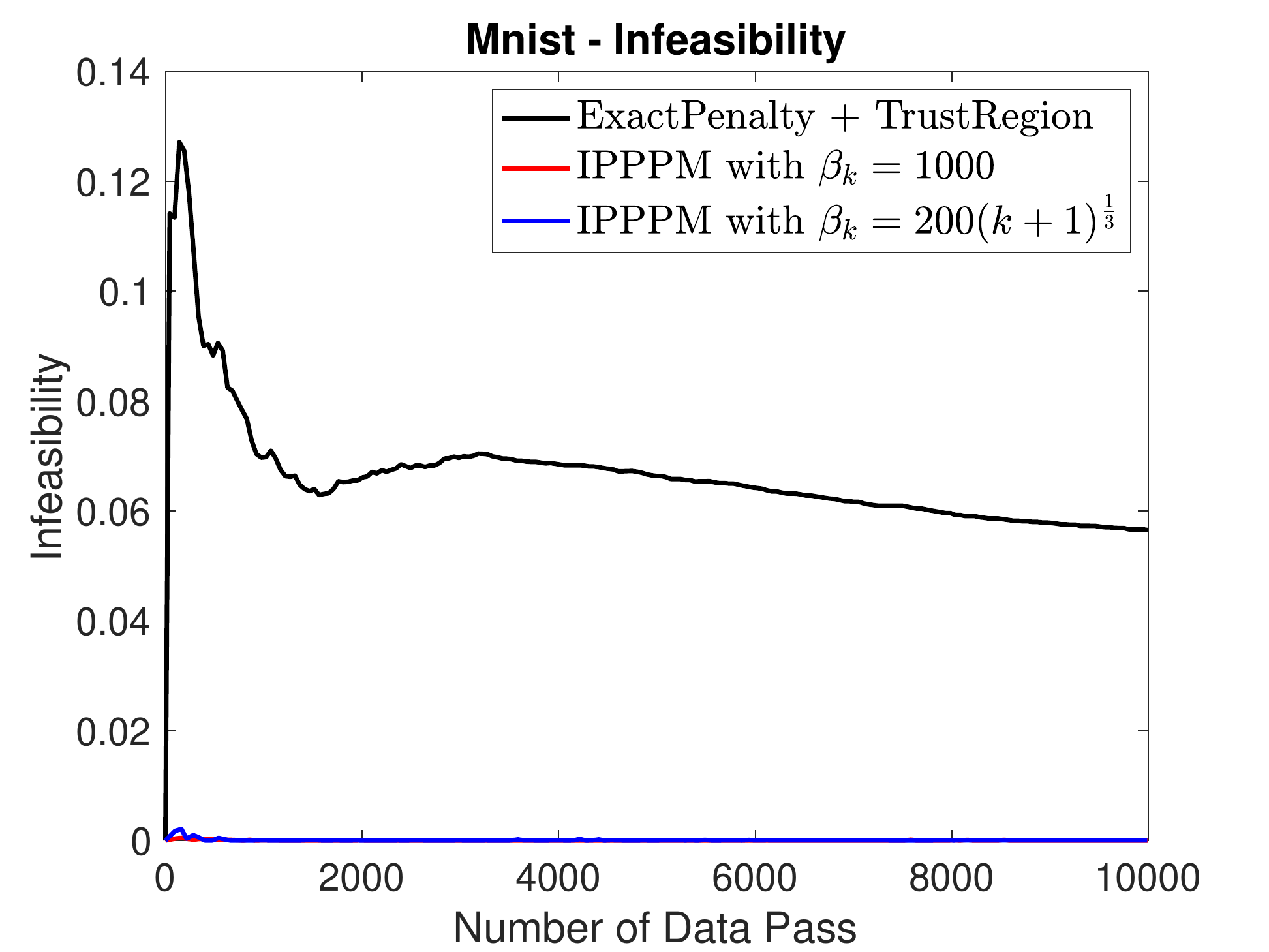}} 
		{\includegraphics[scale=.2]{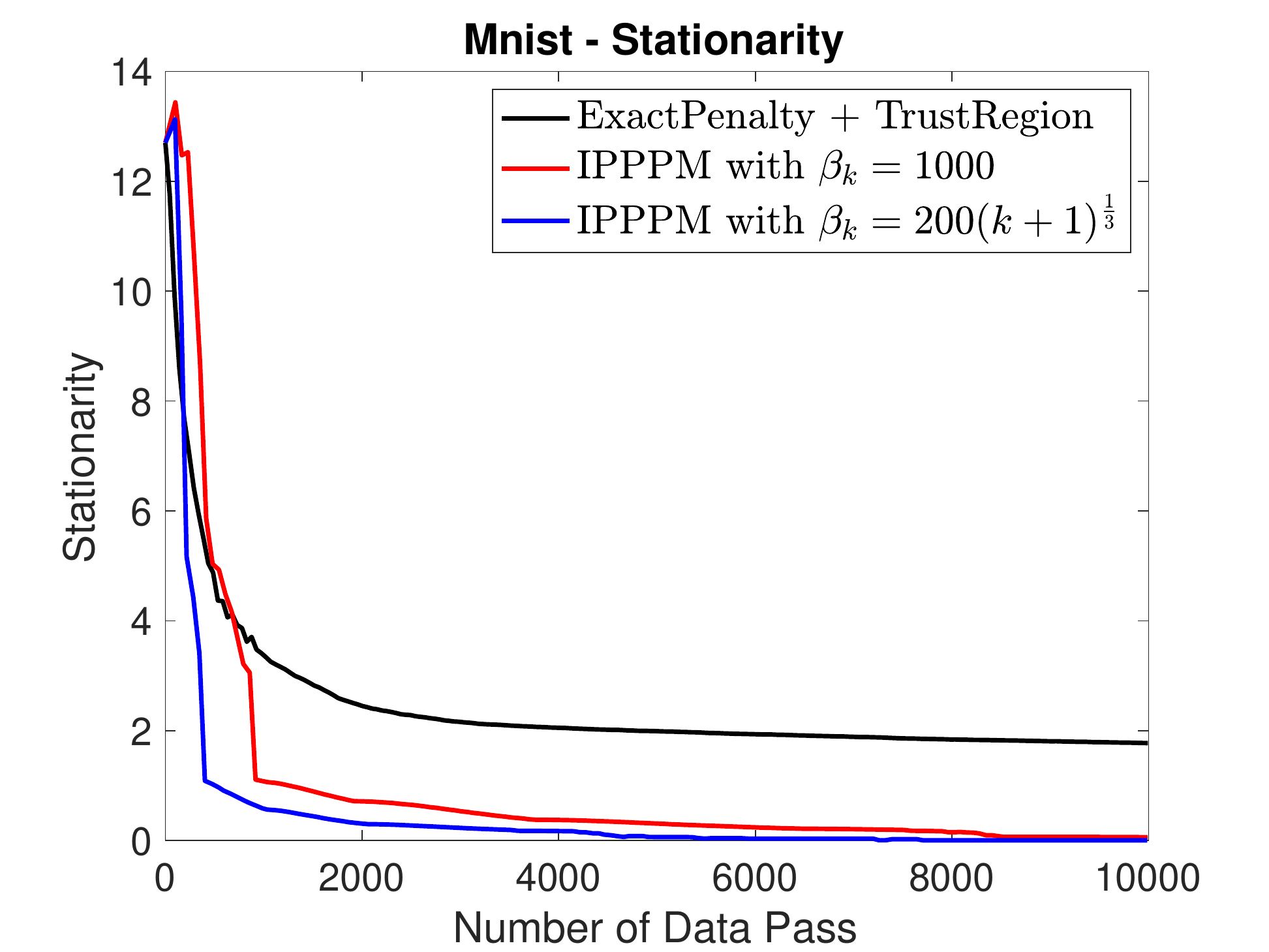}}
		%\\
		%	{\includegraphics[scale=.2]{pics/pendigits_objective.eps}}
		%	{\includegraphics[scale=.2]{pics/pendigits_infeasibility.eps}}
		%	{\includegraphics[scale=.2]{pics/pendigits_stationarity.eps}}
		%\\
		%	{\includegraphics[scale=.2]{pics/segment_objective.eps}}
		%	{\includegraphics[scale=.2]{pics/segment_infeasibility.eps}} 
		%	{\includegraphics[scale=.2]{pics/segment_stationarity.eps}}
		\caption{Comparison between the iPPP method and the trust-region-based penalty method in \cite{cartis2011evaluation} for solving multi-class Neyman-Pearson classification problem \eqref{eq:NPclassification_multi} on two datasets from LIBSVM.} 
		\label{fig:neyman_pearson}
	\end{figure*}

	For our iPPP method, we need to specify the parameters  $\hat{\varepsilon}_k, \gamma_k$ and $\beta_k$ for each $k$ as well as constant $M^{\text{ini}}$ and $\mu^{\text{ini}}$. The inner algorithms also require parameters $\gamma_{\text{inc}}, \gamma_{\text{dec}}, \gamma_{\text{sc}} $, and $\theta_{\text{sc}}$. We set $M^{\text{ini}}=10$, $\mu^{\text{ini}}=1$, $\gamma_{\text{inc}} = 1.5$, $\gamma_{\text{dec}} = \gamma_{\text{sc}}  = 1.2$, and $\theta_{\text{sc}} = 0.5$. For other parameters, we compare two different settings: one usinng $\hat{\varepsilon}_k = 1/(k+1)^2, \gamma_k=0.1, \beta_k = 1000, \forall\, k$, and the other using $\hat{\varepsilon}_k = \frac{1}{\beta(k+1)^{\frac{4}{3}}}, \gamma_k=0.1(k+1)^{\frac{1}{3}}, \beta_k = \beta(k+1)^\frac{1}{3}, \,\forall\, k$.
	%For each instance, we first fix $\beta_k$ to be a constant for each $k$. For all datasets, we choose $\hat{\varepsilon}_k = 1/(k+1)^2, \gamma_k=0.1$, $\beta_k = 1000$, . We also implemented the other variant where penalty parameter is changed in each iteration. For all datasets, we set $\hat{\varepsilon}_k = \frac{1}{\beta(k+1)^{\frac{4}{3}}}, \gamma_k=0.1(k+1)^{\frac{1}{3}}$, $\beta_k = \beta(k+1)^\frac{1}{3}$, $M^{\text{ini}}=10$, $\mu^{\text{ini}}=1$, $\gamma_{\text{inc}} = 1.5$, $\gamma_{\text{dec}} = \gamma_{\text{sc}}  = 1.2$, and $\theta_{\text{sc}} = 0.5$. 
	In the latter setting, we choose $\beta=200$ for \textit{mnist} and choose $\beta=500$ for \textit{covtype}.

	The numerical results are presented in Figure \ref{fig:neyman_pearson}. The $x$-axis represents the number of data passes each algorithm performs. The $y$-axis represents the objective value of iterates in the first column, the infeasibility of iterates (i.e., $\max_{i=1,\dots,m}\lbrace f_i(\bx), 0 \rbrace$) in the second column,  and the stationarity of iterates in the third column. Let $I(\bx)=\{1\leq i\leq m|f_i(\bx)\geq 0\}$ and $\X=\{\bx=(\bx_1,\dots,\bx_K)\, |\, \|\bx_k\|\leq\lambda,k=1,\dots,K\}$. We calculate the stationarity of a solution $\bx$ as the optimal objective value of the following convex optimization
	%\small
	\begin{eqnarray*}
		\min_{\blambda\in\mathbb{R}_+^{|I(\bx)|},\by\in\mathbb{R}^n}
		\text{dist}\Bigg(\nabla f_0(\bx^{(k)})+ \sum_{i\in I(\bx)}\lambda_i\nabla f_i(\bx^{(k)})+ \sum_{j=1}^ny_j\nabla c_j(\bx^{(k)}),-\mathcal{N}_\X(\bx)
		\Bigg),
	\end{eqnarray*}
	%\normalsize
	which can be solved as a convex quadratic program and we solve using Matlab built-in QP solver. We observe from Figure~\ref{fig:neyman_pearson} that, for these two instances, our iPPP method outperforms the trust-region-based penalty method by \cite{cartis2011evaluation} in terms of its capability of improving objective value, feasibility, and stationarity of the iterates simutaneously. Moreover, these two instances also suggest that the iPPP method using growing penalty parameters performs better than using a fixed penalty parameter.

	\section{Conclusion}
	\label{sec:conclusion}
	We proposed a gradient-based penalty method for a constrained non-convex optimization problem. The complexity of the proposed algorithm for finding an approximate stationary point is derived for two cases: (i) when the objective function is non-convex but the constraint functions are convex and, (ii) when the objective and constraint functions are all non-convex. For the first case, our method can produce an $\varepsilon$-stationary point with complexity of $\tilde O(\varepsilon^{-5/2})$ under Slater's condition. For the second case, the complexity is $\tilde O(\varepsilon^{-3})$ if a non-singularity condition holds on the constraints and otherwise $\tilde O(\varepsilon^{-4})$ if an initial feasible solution is assumed. 
	%\section{Numerical Experiments}
	%\label{sec:exp}
	%\vspace{-3mm}
	
	%\begin{acknowledgements}
	%\end{acknowledgements}


	\bibliographystyle{spmpsci}      % mathematics and physical sciences
	\bibliography{references}
	
	\appendix
	
	\section*{Appendix: Discussion on Assumption~\ref{assume:nonconvexconstraintsingular} for application \eqref{eq:NPclassification_multi}}
	\label{sec:append}
	\normalsize
	We explain that Assumption~\ref{assume:nonconvexconstraintsingular} can hold for the tested problem \eqref{eq:NPclassification_multi}. For simplicity, we consider 
	the case of $K=2$, and in this case, we have a single inequality constraint in the form of 
	$$f(\bx) := \frac{1}{N_2}\sum_{i=1}^{N_2}\phi(\bx_2^\top\xi_i - \bx_1^\top \xi_i) - r_2,$$
	where $\bx=[\bx_1;\bx_2]$, $\phi(z) = 1/(1+\exp(z))$, $N_2$ denotes the number of data points in $\mathcal{D}_2$, and $\xi_i$ is the $i$-th data point in $\mathcal{D}_2$. In addition, let $\X=\{(\bx_1,\bx_2): \|\bx_1\|\le \lambda, \|\bx_2\|\le \lambda\}$. It is easy to have
	%\small
	$$
	\mathcal{N}_\X(\bx)=\left\{
	\begin{array}{ll}
		\{\vzero\}, & \text{ if }\|\bx_1\| < \lambda, \|\bx_2\| < \lambda, \\[3pt]
		\{\vzero\} \times \{a_2 \bx_2: a_2\ge0\}, & \text{ if }\|\bx_1\| < \lambda, \|\bx_2\| = \lambda, \\[3pt]
		\{a_1 \bx_1: a_1\ge0\} \times \{\vzero\}, & \text{ if }\|\bx_1\| = \lambda, \|\bx_2\| < \lambda, \\[3pt]
		\{a_1 \bx_1: a_1\ge0\} \times \{a_2 \bx_2: a_2\ge0\}, & \text{ if }\|\bx_1\| = \lambda, \|\bx_2\| = \lambda.
	\end{array}
	\right.
	$$
	\normalsize
	The condition in Assumption~\ref{assume:nonconvexconstraintsingular} reduces to 
	\small
	\begin{equation}\label{eq:non-sing-cond-app}
		\exists\, \nu>0 \text{ such that }\nu [f(\bx)]_+ \le \mathrm{dist}\big([f(\bx)]_+ \nabla f(\bx), - \mathcal{N}_\X(\bx)\big), \,\forall\, \bx\in\X.
	\end{equation}
	\normalsize
	
	Let 
	\small
	$$\bE=\frac{1}{N_2}[\xi_1, \ldots, \xi_{N_2}], \quad u_i(\bx) = \phi'(\bx_2^\top\xi_i - \bx_1^\top \xi_i),\, \text{ for }i=1,\ldots, N_2.$$
	\normalsize
	Then 
	\small
	$$\nabla f(\bx) = \frac{1}{N_2}\sum_{i=1}^{N_2}\phi'(\bx_2^\top\xi_i - \bx_1^\top \xi_i)[-\xi_i; \xi_i] = [-\bE \bu(\bx); \bE \bu(\bx)].$$
	\normalsize
	When $f(\bx) \le 0$, the condition in \eqref{eq:non-sing-cond-app} trivially holds for any $\nu >0$. Below, we assume the feasibility of the origin, i.e., $f(\vzero) \le 0$ as in our numerical experiment and also
	\small
	\begin{equation}\label{eq:cor-data}
		\xi_i\neq \vzero, \,\forall\, i, \quad \xi_i^\top \xi_j \ge 0,\, \forall\, i, j, 
	\end{equation}
	\normalsize
	which can be ensured by lifting the data point one more dimension, i.e., $\xi_i \gets [\xi_i; c],\forall\, i$ for some $c>0$. We establish the condition in \eqref{eq:non-sing-cond-app} through discussing three cases on $\bx$ with $f(\bx) > 0$.
	
	%\vspace{0.2cm}
	
	\noindent\textbf{Case I:} $\|\bx_1\| < \lambda, \|\bx_2\| < \lambda$. In this case, $\mathcal{N}_\X(\bx)= \{\vzero\}$, and thus the inequality in \eqref{eq:non-sing-cond-app} becomes $\nu \le \|\nabla f(\bx)\|$. Let
	\small
	\begin{equation}\label{eq:nu-1}
		\nu_1 = \inf_{\substack{\|\bx_1\| < \lambda, \|\bx_2\| < \lambda}} \|\nabla f(\bx)\| = \min_{\bx\in\X} \sqrt{2}  \|\bE \bu(\bx)\|.
	\end{equation}
	\normalsize
	Since $\phi'(\bx)<0$ for any $\bx$, there exist $\eta_1>0, \eta_2 > 0$ such that $-\eta_2 \le u_i(\bx) \le -\eta_1, \, \forall\, i$ for all $\bx\in\X$. Hence $\nu_1 > 0$ by \eqref{eq:cor-data}. 
	
	%\vspace{0.2cm}
	
	\noindent\textbf{Case II:} $\|\bx_1\| < \lambda, \|\bx_2\| = \lambda$ or $\|\bx_1\| = \lambda, \|\bx_2\| < \lambda$. We only consider the former because the latter can be discussed in the same way. In the former case, $\mathcal{N}_\X(\bx)=\{\vzero\} \times \{a_2 \bx_2: a_2\ge0\}$, and the inequality in \eqref{eq:non-sing-cond-app} becomes 
	\small
	$$\big(\nu [f(\bx)]_+\big)^2 \le \big\|[f(\bx)]_+\bE \bu(\bx)\big\|^2 + \min_{a_2\ge 0}\big\|[f(\bx)]_+ \bE \bu(\bx) + a_2\bx_2\big\|^2,$$ 
	\normalsize
	which is implied by  the fact that $\nu \le  \|\bE \bu(\bx)\|$ with $\nu = \nu_1/\sqrt{2}$ and $\nu_1$ defined in \eqref{eq:nu-1}. 
	
	%\vspace{0.2cm}
	
	\noindent\textbf{Case III:} $\|\bx_1\| = \lambda, \|\bx_2\| = \lambda$. In this case, $\mathcal{N}_\X(\bx)=\{a_1 \bx_1: a_1\ge0\} \times \{a_2 \bx_2: a_2\ge0\}$, and the inequality in \eqref{eq:non-sing-cond-app} becomes
	\small 
	$$\big(\nu [f(\bx)]_+\big)^2 \le \min_{a_1\ge 0}\big\|[f(\bx)]_+ \bE \bu(\bx) - a_1\bx_1\big\|^2 + \min_{a_2\ge 0}\big\|[f(\bx)]_+ \bE \bu(\bx) + a_2\bx_2\big\|^2,$$
	\normalsize
	which, as $f(\bx)>0$, is equivalent to
	\small 
	$$\nu^2 \le \min_{a_1\ge 0}\big\|\bE \bu(\bx) - a_1\bx_1\big\|^2 + \min_{a_2\ge 0}\big\|\bE \bu(\bx) + a_2\bx_2\big\|^2.$$
	\normalsize
	Let $\nu_2\ge0$ be defined as 
	\small
	\begin{equation}\label{eq:nu-2}
		\nu_2^2 = \min_{\substack{\|\bx_1\| = \lambda\\ \|\bx_2\| = \lambda \\ f(\bx) \ge0}} \left\{\min_{a_1\ge 0}\big\|\bE \bu(\bx) - a_1\bx_1\big\|^2 + \min_{a_2\ge 0}\big\|\bE \bu(\bx) + a_2\bx_2\big\|^2\right\}.
	\end{equation}
	\normalsize
	Notice that the minimum of the above problem is reached at a point $\bar\bx$ and numbers $\bar a_1$ and $\bar a_2$. Suppose $\nu_2=0$. It must hold that $\bar\bx_1=\lambda \frac{\bE \bu(\bar\bx)}{\|\bE \bu(\bar\bx)\|}$ and $\bar\bx_2=-\lambda \frac{\bE \bu(\bar\bx)}{\|\bE \bu(\bar\bx)\|}$ with the corresponding $\bar a_1 = \bar a_2 = \frac{\|\bE \bu(\bar\bx)\|}{\lambda}$. Since $u_i(\bar\bx) < 0,\,\forall\, i$, we have from \eqref{eq:cor-data} that $\bar\bx_2^\top \xi_i - \bar\bx_1^\top \xi_i =-\frac{\lambda}{\|\bE \bu(\bar\bx)\|}\xi_i^\top \bE \bu(\bar\bx)>0$ for all $i$. This means $f(\bar\bx) < f(\vzero) \le 0$ by the monotonicity of $\phi$, which contradicts with the fact that $f(\bar\bx) \ge0$. Therefore, we must have $\nu_2>0$.
	
	By the above discussions, we can set $\nu = \min\{\nu_1/\sqrt{2},\, \nu_2\} >0$ to ensure condition \eqref{eq:non-sing-cond-app}, which gives the following conclusion.
	
	\vspace{0.2cm}
	\noindent\textbf{Claim:}  Assumption~\ref{assume:nonconvexconstraintsingular} can hold for the tested problem \eqref{eq:NPclassification_multi}.

\end{document}